\documentclass{amsart}

\usepackage{amssymb, amsmath,mathtools,mathabx}%
\usepackage{mathrsfs}%stix}
\usepackage{esint}
\usepackage{amscd}
\usepackage{verbatim}
\usepackage{stmaryrd}
\usepackage[dvipsnames]{xcolor}
\usepackage{pgf,tikz}
\usetikzlibrary{arrows}

\usepackage{enumerate}
\usepackage{a4wide}

\usepackage[colorlinks,linkcolor={blue},citecolor={blue},urlcolor={purple},]{hyperref}

\usepackage[colorinlistoftodos,prependcaption,textsize=tiny]{todonotes}

\usepackage{forest}

\usepackage{tabularx}
\newcolumntype{L}{>{\arraybackslash}X}
\usepackage{multirow}

\theoremstyle{plain}

\newtheorem{theorem}{Theorem}[section]
\theoremstyle{remark}
\newtheorem{remark}[theorem]{Remark}

\theoremstyle{plain}

\newtheorem{lemma}[theorem]{Lemma}
\newtheorem{proposition}[theorem]{Proposition}
\newtheorem{definition}[theorem]{Definition}

\newtheorem{assumption}[theorem]{Assumption}

\numberwithin{equation}{section}

\def\N{{\mathbb N}}

\def\R{{\mathbb R}}

\newcommand{\one}{{{\bf 1}}}

\newcommand{\E}{{\mathbf E}}
\renewcommand{\P}{{\mathbf P}}
\newcommand{\F}{{\mathscr F}}

\newcommand{\g}{\gamma}

\newcommand{\om}{\omega}
\renewcommand{\O}{\Omega}

\newcommand{\n}{{\rm N}}%\textbf{n}

\newcommand{\Dom}{\mathcal{O}}

\renewcommand{\emptyset}{\varnothing}

\newcommand{\Tor}{\mathbb{T}}

\newcommand{\A}{{\mathcal A}}

\newcommand{\loc}{\mathrm{loc}}

\newcommand{\calL}{{\mathscr L}}

\newcommand{\z}{x_3}

\newcommand{\h}{{\rm H}}
\newcommand{\D}{\mathscr{D}}

\newcommand{\Ls}{\mathbb{L}}

\newcommand{\Hs}{\mathbb{H}}

\newcommand{\wt}{\widetilde}

\usepackage{stmaryrd}

\newcommand{\forcetwo}{\mathcal{E}}
\newcommand{\force}{\mathcal{F}}
\renewcommand{\div}{\mathrm{div}}
\newcommand{\w}{w}

\newcommand{\p}{\mathbb{P}}
\newcommand{\q}{\mathbb{Q}}
\newcommand{\embed}{\hookrightarrow}

\newcommand{\Progress}{\mathscr{P}}

\newcommand{\cor}{\mathcal{C}}

\newcommand{\MT}{\mathcal{M}}

\newcommand{\Temp}{\theta}
\newcommand{\T}{\Temp}

\newcommand{\op}{\mathcal{J}}

\newcommand{\opt}{\mathcal{T}}

\newcommand{\ellip}{\nu}

\newcommand{\Hr}{H_{{\rm R}}}
\newcommand{\kone}{\kappa}
\newcommand{\ktwo}{\sigma}
\newcommand{\ktwon}{\sigma_n}
\newcommand{\Borel}{\mathscr{B}}

\newcommand{\wh}{\widehat}

\newcommand{\pr}{\mathbb{P}_{\h}}
\newcommand{\qr}{\mathbb{Q}_{\h}}
\newcommand{\ft}{F_{\T}}
\newcommand{\fvt}{f_v}
\newcommand{\gvtn}{g_{v,n}}

\newcommand{\fv}{F_{v}}
\newcommand{\gt}{G_{\T}}
\newcommand{\gv}{G_{v}}

\newcommand{\Lp}{\mathcal{P}_{\hp}}
\newcommand{\Lpp}{\mathcal{P}_{\hp,\phi}}
\newcommand{\Lt}{\mathcal{L}_{\tp,\hp}}
\newcommand{\Ltb}{\overline{\mathcal{L}}_{\tp,\hp}}
\newcommand{\Lpg}{\mathcal{P}_{\hp,G}}
\newcommand{\hp}{\gamma}

\newcommand{\Br}{\mathrm{B}}

\newcommand{\gx}{G_{u}}
\newcommand{\gtn}{G_{\T,n}}
\newcommand{\gvn}{G_{v,n}}

\newcommand{\y}{\Xi}
\newcommand{\Ht}{\mathcal{R}}
\newcommand{\vv}{u}

\newcommand{\stocgrav}{{k}}

\renewcommand{\v}{U_{*}}

\newcommand{\tp}{\pi}

\newcommand{\TT}{\Theta}
\newcommand{\xx}{\mathcal{X}}
\newcommand{\yy}{\mathcal{Y}}
\newcommand{\II}{\boxed{L}}
\newcommand{\III}{\boxed{I}}
\newcommand{\dop}{\mathcal{L}}

\newcommand{\low}{L}
\newcommand{\Fun}{\mathcal{F}_{\alpha}}

\newcommand{\dd}{\mathrm{d}}
\newcommand{\reference}{{\rm r}}

\allowdisplaybreaks

\usepackage{scalerel,stackengine}
\stackMath
\newcommand\reallywidehat[1]{%
\savestack{\tmpbox}{\stretchto{%
  \scaleto{%
    \scalerel*[\widthof{\ensuremath{#1}}]{\kern-.6pt\bigwedge\kern-.6pt}%
    {\rule[-\textheight/2]{1ex}{\textheight}}%WIDTH-LIMITED BIG WEDGE
  }{\textheight}% 
}{0.5ex}}%
\stackon[1pt]{#1}{\tmpbox}%
}

\begin{document}

\date\today

\title[Primitive equations with non-isothermal turbulent pressure]{The stochastic primitive equations with \\ non-isothermal turbulent pressure}

\keywords{stochastic partial differential equations, primitive equations, global well-posedness, gradient noise, stochastic maximal regularity, turbulent flows, Kraichnan's turbulence, thermal fluctuations.}

\thanks{Antonio Agresti has received funding from the European Research Council (ERC) under the European Union’s Horizon 2020 research and innovation programme (grant agreement No 948819) \includegraphics[height=0.4cm]{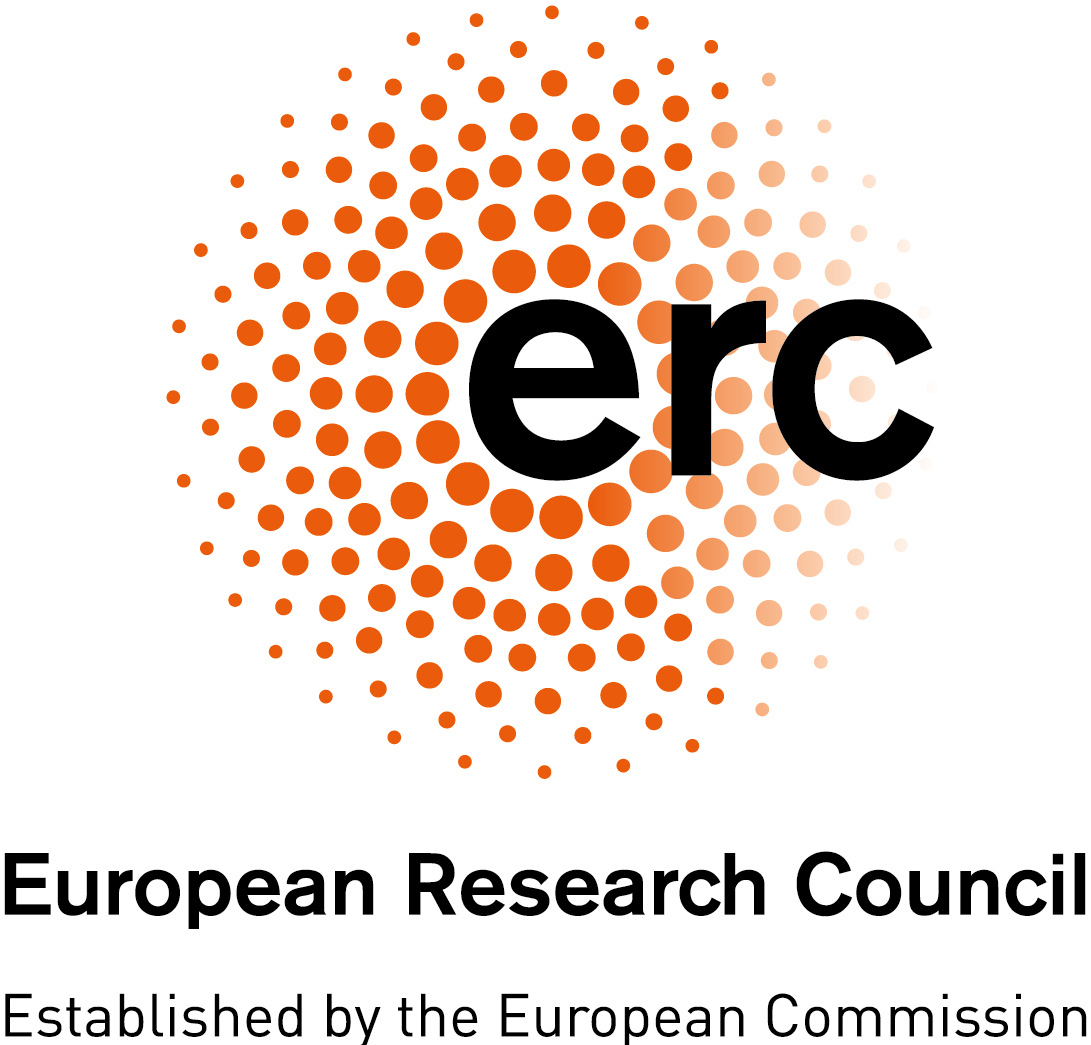}\,\includegraphics[height=0.4cm]{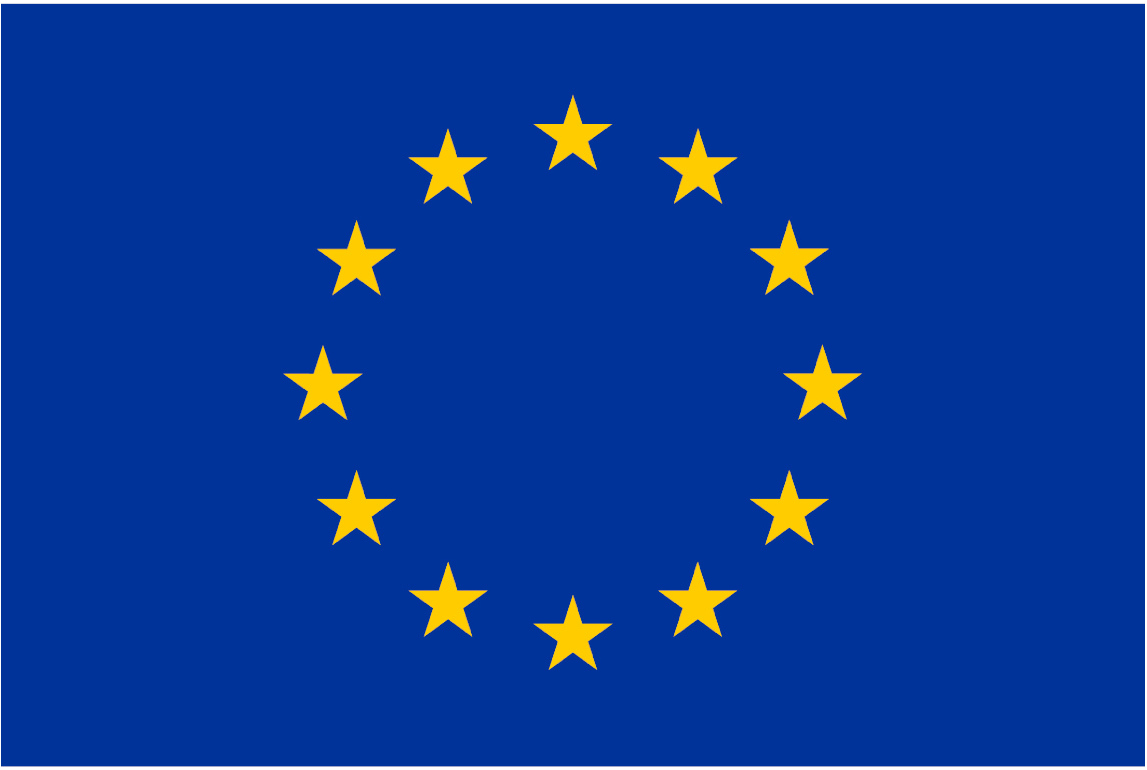}. Antonio Agresti is a member of GNAMPA (IN$\delta$AM).
\\
Matthias Hieber gratefully acknowledges the support by the Deutsche Forschungsgemeinschaft (DFG) through  the Research Unit 5528 -- project number 500072446.\\
Amru Hussein has been supported by 
Deutsche Forschungsgemeinschaft (DFG) -- project number 508634462 and by MathApp -- Mathematics Applied to Real-World Problems -- part of the Research Initiative of the Federal State of Rhineland-Palatinate, Germany.
\\
Martin Saal has been supported by Deutsche Forschungsgemeinschaft (DFG) -- project number 429483464.
}

\author[Agresti]{Antonio Agresti}
\address{Institute of Science and Technology Austria (ISTA), Am Campus 1,
	3400 Klosterneuburg, Austria}
\email{antonio.agresti92@gmail.com}
\curraddr{Delft Institute of Applied Mathematics, Delft University of Technology, P.O.\ Box 5031, 2600 GA Delft, The Netherlands}

\author[Hieber]{Matthias Hieber} 
\address{Department of Mathematics,
	TU Darmstadt, Schlossgartenstr. 7, 64289 Darmstadt, Germany}
\email{hieber@mathematik.tu-darmstadt.de}

\author[Hussein]{Amru Hussein}
\address{Department of Mathematics,
	RPTU Kaiserslautern-Landau, Paul-Ehrlich-Stra{\ss}e 31,
	67663 Kaiserslautern, Germany}
\email{hussein@mathematik.uni-kl.de}

\author[Saal]{Martin Saal}
\address{Department of Mathematics,
	TU Darmstadt, Schlossgartenstr. 7, 64289 Darmstadt, Germany}
	\email{msaal@mathematik.tu-darmstadt.de}

\subjclass[2010]{Primary 35Q86; Secondary 35R60, 60H15, 76M35, 76U60} 

\begin{abstract}
In this paper, we introduce and study the primitive equations with \emph{non}-isothermal turbulent pressure and transport noise. They are derived from the Navier-Stokes equations by employing stochastic versions of the Boussinesq and the hydrostatic approximations.
The temperature dependence of the turbulent pressure can be seen as a consequence of an additive noise acting on the small vertical dynamics.
For such a model we prove global well-posedness in $H^1$ where the noise is considered in both the It\^{o} and Stratonovich formulations.
Compared to previous variants of the primitive equations, the one considered here presents a more intricate coupling between the velocity field and the temperature. The corresponding analysis is seriously more involved than in the deterministic setting.
Finally, the continuous dependence on the initial data and the energy estimates proven here are new, even in the case of isothermal turbulent pressure.  
\end{abstract}

\maketitle

\addtocontents{toc}{\protect\setcounter{tocdepth}{1}}
\tableofcontents

\section{Introduction}\label{s:intro}

In this paper, we introduce and study the stochastic primitive equation with \emph{non}-isothermal turbulent pressure and transport noise. 
The primitive equations are one of the fundamental models 
for geophysical flows  used to describe
oceanic and atmospheric dynamics. They are derived from the
Navier-Stokes equations on domains where the vertical scale
is much smaller than the horizontal scale by the small aspect ratio limit.
%%%
Additional information for the various versions of the deterministic
primitive equations can be found, e.g.\ in \cite{Ped, Vallis06}. 
The introduction of additive and multiplicative noise
into models for geophysical flows can  
 be used on the one hand 
 to account for numerical
and empirical uncertainties and errors
and on the other hand as 
subgrid-scale parameterizations for data assimilation, and ensemble prediction 
as described in the review articles \cite{Delsole04, Franzke14, Palmer19}.
The primitive equations with non-isothermal turbulent pressure introduced here present a more intricate interplay between the velocity field and the temperature which leads to serious mathematical complications compared to the deterministic situation, see e.g.\ \cite{CT07,HH20_fluids_pressure}. The same difficulties also appear when comparing previously studied \emph{stochastic} perturbations of the primitive equations (see e.g.\  \cite{Primitive1,BS21,DEBUSSCHE20111123,Debussche_2012} and the references therein) with the one considered here. A discussion of these difficulties can be found in Subsection \ref{ss:novelty} below.  The presence of the temperature in the balance for the turbulent pressure can be thought of as the large-scale effect of thermal fluctuations acting on the small vertical dynamics. 
From a modelling point of view, a non-isothermal turbulent pressure may provide a new perspective on the contribution of the temperature on geophysical flows ruled by the primitive equations. For instance, we hope that the model introduced in the current paper can be used in the study of the influence of thermal fluctuations on oceanic streams.
 As in \cite{Primitive1}, we also consider dynamics driven by transport noise. 
 The latter was first introduced by  R.H.~Kraichanan in the study of turbulent flows
\cite{K68,K94}, and it has been widely studied in the context of the Navier-Stokes equations, see \cite{HLN21_annals,MR01,MR04} for a physical justification and also  \cite{AV21_NS,BCF91,BCF92,F_intro,HLN19, MR05} and the references therein for related mathematical results.
Let us stress that the difficulties arising from the non-isothermal turbulent pressure are still present in the absence of transport noise, see Subsection \ref{ss:novelty} for details.

The primitive equations with non-isothermal turbulent pressure in the domain $\Dom=\Tor^2\times (-h,0)$, where $ h>0$ and $\Tor^2$ denotes the two-dimensional flat torus,  are given by the following system:
\begin{subequations}
	\label{eq:primitive_full}
\begin{alignat}{6}
\label{eq:primitive_full_1}
		\begin{split}
		\dd v -\Delta v\,\dd t &=\Big[  -\nabla_{\h} P  -(v\cdot \nabla_{\h})v- w\partial_3 v + \fv
		\Big]\,\dd t \\
		&\quad +\sum_{n\geq 1}\Big[(\phi_{n}\cdot\nabla) v-\nabla_{\h}\wt{P}_n  +\gvn\Big] \, \dd \beta_t^n,
		\end{split}\\
\label{eq:primitive_full_2}
		\dd  \T -\Delta \T\,\dd t&=\Big[  -(v\cdot \nabla_{\h})\T- w\partial_3 \T+ \ft \Big]\, \dd t 
		+\sum_{n\geq 1}\Big[(\psi_{n}\cdot\nabla) \T+\gtn\Big]\,\dd \beta_t^n,\\
\label{eq:primitive_full_3}
\partial_{3} P+ \kone \T&=0,\\
\label{eq:primitive_full_4} 
\partial_3 \wt{P}_n+\ktwon \T&=0,\\
\label{eq:primitive_full_5}
\div_{\h} v+\partial_3 w&=0,\\
\label{eq:primitive_full_6}
v(0,\cdot)&=v_0,\quad \quad\T(0,\cdot)=\T_0.
\end{alignat}
\end{subequations}
Here $\kone,\ktwon$ and $\phi_{n}=(\phi_n^{j})_{j=1}^3,\psi_{n}=(\psi_n^{j})_{j=1}^3$ are assigned maps. 
 Moreover
$v=(v^{k})_{k=1}^2:[0,\infty)\times \O\times \Dom\to \R^2$ denotes the horizontal component of the unknown velocity field $u=(v,w)$ and $w:[0,\infty)\times \O\times \Dom\to \R$ the vertical one, $P:[0,\infty)\times \O\times  \Dom\to \R$ the unknown pressure, $\wt{P}_n:[0,\infty)\times \O\times  \Dom \to \R$ the components of the unknown turbulent pressure and $\T:[0,\infty)\times \O\times \Dom\to \R$ the unknown temperature, respectively.
Finally, 
$(\beta_t^n\,:\,t\geq 0)_{n\geq 1}$ is a sequence of independent standard Brownian motions on a given filtered probability space $(\O,\A,(\F_t)_{t\geq 0},\P)$, and $(\fv,\ft,\gvn,\gtn)$ are given maps possibly depending on $(v,\T,\nabla v,\nabla \T)$. These describe deterministic and stochastic forces, they also take into account lower-order effects like the Coriolis force. The reader is referred to Subsection \ref{ss:set_up} for the unexplained notation.

The problem \eqref{eq:primitive_full} is supplemented with the following boundary conditions
\begin{subequations}
	\label{eq:boundary_conditions_intro}
\begin{alignat}{2}
\label{eq:boundary_conditions_intro_1}
		\partial_3 v (\cdot,-h)=\partial_3 v(\cdot,0)=0 \ \  \text{ on }\Tor^2,&\\
\label{eq:boundary_conditions_intro_2}
		\partial_3 \T(\cdot,-h)= \partial_3 \T(\cdot,0)+\alpha \T(\cdot,0)=0\ \  \text{ on }\Tor^2,&
\end{alignat}
\end{subequations}
where $\alpha\in \R$ is given and 
\begin{equation}
	\label{eq:boundary_conditions_w}
	\begin{aligned}
		w(\cdot,-h) =w(\cdot,0)=0\ \ \text{ on }\Tor^2.
	\end{aligned}
\end{equation}

Actually, in our main results, we consider a generalization of the system in \eqref{eq:primitive_full}, see \eqref{eq:primitive} in the main text. Moreover, our arguments also cover the case where the boundary conditions \eqref{eq:boundary_conditions_intro} are replaced by periodic ones. Further comments are given in Remark \ref{r:periodic_BC}.

The aim of this paper is to show the \emph{global well-posedness} in the strong setting (both analytically and probabilistically) of the system \eqref{eq:primitive_full}-\eqref{eq:boundary_conditions_w}, see Theorems \ref{t:global_primitive_strong_strong} and \ref{t:continuous_dependence}. In these results, the noise is understood in the It\^{o}-sense. 
In Section \ref{s:Stratonovich} we also discuss the case of Stratonovich noise. 
In stochastic fluid mechanics, and in particular, for geophysical flows, the \emph{Stratonovich formulation} of the noise is relevant, and it is seen as a more realistic model compared to the It\^o one, see e.g.\ 
\cite{BiFla20,DP22_two_scale,FlaPa21,Franzke14,HL84,MR01,MR04,W_thesis}.
 From an analytic point of view, the Stratonovich noise is not more difficult than the It\^o one and, at least formally, one can convert the Stratonovich formulation into the It\^o one up to %consider 
 some additional corrective terms. The global well-posedness of \eqref{eq:primitive_full} in the strong setting with \emph{Stratonovich}  noise is proved in Section \ref{s:Stratonovich}.

For the reader's convenience, we state here a simplified version of the Theorems \ref{t:global_primitive_strong_strong} and \ref{t:continuous_dependence}. Below we write $\phi^j\stackrel{{\rm def}}{=}(\phi^j_n)_{n\geq 1}$, $\psi^j\stackrel{{\rm def}}{=}(\psi^j_n)_{n\geq 1}$ and $\R_+\stackrel{{\rm def}}{=}(0,\infty)$.

\begin{theorem}[Simplified version]
\label{t:intro}
Let 
$\kone$ be constant, $(\ktwon)_{n\geq 1}\in \ell^2$, $\gvn^k=\gtn=0$,  $\ft=0$, and let $\fv= k_0(v^2,-v^1)$ for $k_0\in\R$ be the Coriolis force.
For all $n\geq 1$ let the maps 
\begin{align*}
\phi_n,\psi_n\colon \R_+\times \O\times \Dom\to \R^3
\end{align*}
 be $\Progress\otimes \Borel$-measurable, and let for some $\delta>0$ and all $j\in \{1,2,3\}$ be
\begin{align*}
\phi^j,\psi^j\in L^\infty(\R_+\times \O;H^{1,3+\delta}(\Dom;\ell^2)) .
\end{align*}
Suppose that $(\phi^j_n,\psi_n^j)$ are independent of $x_3$ for $j\in \{1,2\}$.
Furthermore, assume that there exists $\ellip\in (0,2)$ such that,  a.s.\ for all $t\in \R_+$, $x\in \Dom$ and $\xi\in \R^3$ the parabolicity conditions
		\begin{align*}
			\sum_{n\geq 1} \Big(\sum_{1\leq j\leq 3} \phi^j_n(t,x) \xi_j\Big)^2\leq \ellip |\xi|^2
			\ \ \text{ and }\ \ 
			\sum_{n\geq 1} \Big(\sum_{1\leq j\leq 3} \psi^j_n(t,x) \xi_j\Big)^2 \leq \ellip |\xi|^2
		\end{align*}
hold. Then for each
	$
	v_0\in L^0_{\F_0}(\O;\Hs^1(\Dom)) $ and $ \T_0\in L^0_{\F_0}(\O;H^1(\Dom))
	$
	the following hold:
\begin{enumerate}[{\rm(1)}]
\item %{\sc (Global existence)} 
There exists a \emph{unique global} strong solution $(v,\T)$ to \eqref{eq:primitive_full}-\eqref{eq:boundary_conditions_w} satisfying
	$$
	(v,\T)\in L^2_{\loc}([0,\infty);\Hs_{\n}^2(\Dom)\times \Hr^2(\Dom))\cap C([0,\infty);\Hs^1(\Dom)\times H^1(\Dom)) 
	\text{ a.s.\ }
	$$
\item\label{it:estimate_intro} %{\sc (Energy estimates)}  
For all $T\in (0,\infty)$ and all $\g>e^{e}$,
\begin{align*}
\P\Big(\sup_{t\in [0,T]}\|v(t)\|_{H^1}^2+\int_{0}^T \|v(t)\|_{H^2}^2\,\dd t\geq \g \Big)
&\lesssim_T \frac{1+\E\|v_0\|_{H^1}^4+\E\|\T_0\|_{H^1}^4}{\log\log\log(\g)},\\
\P\Big(\sup_{t\in [0,T]}\|\T(t)\|_{H^1}^2+\int_{0}^T \|\T(t)\|_{H^2}^2\,\dd t\geq \g \Big)
&\lesssim_T \frac{1+\E\|v_0\|_{H^1}^4+\E\|\T_0\|_{H^1}^4}{\log\log\log(\g)}.
\end{align*}
\item\label{it:estimate_intro_1} %{\sc (Continuity w.r.t.\ initial data)}  
The assignment $(v_0,\T_0)\mapsto (v,\T)$ is continuous in probability in the sense of Theorem \ref{t:continuous_dependence}. 
\end{enumerate}
\end{theorem}

The reader is referred to Subsections \ref{ss:set_up} and \ref{ss:reformulation} for the definition of $\Progress\otimes \Borel$-measurable, $ L^0_{\F_0}(\O;X)$ and the notation for the function spaces. 
%\ref{ss:set_up}.
In the above, we have not specified the unknowns $w$, $P$ and $\wt{P}_n$  as they are uniquely determined by $v$ and $\T$ due to the divergence-free condition and the hydrostatic Helmholtz projection. The reader is referred to \cite[Section 1]{Primitive1} for comments on the relation between the regularity of the transport noise considered in this paper and Krainchan's noise.

Physical motivations for the independence of $(\phi_n^j,\psi_n^j)$ on the $x_3$-coordinate for $j\in \{1,2\}$ are discussed in Remarks \ref{r:phi_independence} and \ref{r:twod_turbulence}. 
In a nutshell, the small aspect ratio limit (i.e.\ the hydrostatic approximation discussed for the deterministic setting in \cite{KGHHKW20,LT19}) shows that the primitive equations can be derived by taking the limit  $\varepsilon\downarrow 0$ of the anisotropic Navier-Stokes equations on a thin domain $\Tor^2\times(-\varepsilon,0)$ (see Figure \ref{fig:domain}), and therefore the variability in the vertical direction of the coefficients disappear in the limit. Hence, the independence of $(\phi_n^j,\psi_n^j)$ on $x_3$ for $j\in \{1,2\}$ is justified.
In particular, the situation for geophysical flows is different from usual turbulence models concerning Navier-Stokes equations \cite{BE12,T02_twod}.

The logarithmic bounds of Theorem 	\ref{t:intro}\eqref{it:estimate_intro} seem rather weak. However, compared to the estimates in the deterministic setting (see e.g.\ \cite{CT07}), even in the absence of noise, it does not seem possible to obtain in \eqref{it:estimate_intro} more than a $\log\log$-decay due to three applications of Grownall's inequality. 
Moreover, it is unclear how to improve the estimates in \eqref{it:estimate_intro} without enforcing regularity assumptions on the noise. The reader is referred to the text below Theorem \ref{t:global_primitive_strong_strong} and to Remark \ref{r:comparison_primitive1} for more details. 
The bounds in Theorem \ref{t:intro}\eqref{it:estimate_intro} remind us of the estimates obtained in \cite[Theorem 4.2]{GHKVZ14}, where the authors proved logarithmic moment bounds in $H^2(\Dom)$ under additional assumptions on the noise. In particular, in \cite{GHKVZ14}, it is not possible to consider gradient or transport type noises (in particular, this forces $\ktwon\equiv0$, cf.\ Subsection \ref{ss:novelty} below). However, it seems that there is no direct relation between the estimates of \eqref{it:estimate_intro} and the above-mentioned estimate of \cite{GHKVZ14}. In the latter, the authors used logarithmic moment bounds to prove the existence of ergodic invariant measures 
in $H^1(\Dom)$. The extension of such result to the system \eqref{eq:primitive_full} goes beyond the scope of this manuscript. Finally, let us mention that the continuous dependence on the initial data in \eqref{it:estimate_intro_1} readily implies the \emph{Feller property} for \eqref{eq:primitive_full} which is a first step in the proof of the existence of ergodic measures, and it is based on the energy estimates in \eqref{it:estimate_intro}.  The reader is referred to Remark \ref{r:Feller} for more details on the Feller property. 

\subsection{Novelties and description of the main difficulty}
\label{ss:novelty}
Compared to the results in \cite{Primitive1}, the major novelty of the current work is the presence of $\ktwon\neq 0$.  Here we explain the main analytic difficulty behind this fact. For simplicity, as in Theorem \ref{t:intro}, in this subsection we assume that $(\sigma_n)_{n\geq 1}\in \ell^2$ is constant.
Note that \eqref{eq:primitive_full_4} yields, for all $(x_{\h},\z)\in \Dom$ (here and below $x_{\h}\in \Tor^2$ and $x_3\in (h,0)$ denote the horizontal and vertical variables, respectively) and $t\in \R_+$,
$$
\wt{P}_n(t,x_{\h},x_3)=\wt{p}_n(t,x_{\h})+\sigma_n\int_{-h}^{\z} \T(t,x_{\h},\zeta)\,\dd \zeta,
$$
where $\wt{p}_n$ depends only on $x_{\h}\in \Tor^2$ (typically referred as \emph{turbulent  surface pressure}). Using the above identity in \eqref{eq:primitive_full_1}, the following \emph{gradient} noise term appears in the $v$-dynamics:
\begin{equation}
\label{eq:gradient_sigma_n}
\sum_{n\geq 1}\ktwon \int_{-h}^{\z} \nabla_{\h}\T(t,x_{\h},\zeta)\,\dd \zeta\,\dd \beta_t^n,
\end{equation}
where $\nabla_{\h}=(\partial_1 ,\partial_2)$. 
In particular, as maximal $L^2$-regularity estimates show (see e.g.\ \cite[Proposition 6.8]{Primitive1} or Lemma \ref{l:smr}), to obtain a-priori $L^{\infty}_t(H^1_x)\cap L^2_t(H^2_x)$-bounds for $v$ (and hence global existence for \eqref{eq:primitive_full}), one needs $L^{\infty}_t(H^2_x)$-bounds for $\T$. This is dramatically different from the case of isothermal turbulent pressure (i.e.\ $\ktwon\equiv 0$), where it is sufficient to show $L^{\infty}_t(H^1_x)$-bounds for $\T$ to obtain $L^{\infty}_t(H^1_x)\cap L^2_t(H^2_x)$-estimates for $v$ (see \cite[Section 5]{Primitive1}). 
Since $L^{\infty}_t(H^1_x)$-bounds for $\T$ follow from standard energy estimates, from an analytic point of view, the proof of global existence of strong solutions in the case $\ktwon\equiv 0$ is essentially independent of the $\T$-dynamics, cf.\ \cite[Section 5]{Primitive1}.
This is not the case for \eqref{eq:primitive_full} with $\ktwon\neq 0$ where the coupling between the evolution of $v$ and the one of $\T$ is more subtle and $v$ cannot be decoupled from $\T$ in the $L^{\infty}_t(H^1_x)\cap L^{2}_t(H^2_x)$-estimates. Let us remark that these difficulties are also present even in the absence of transport noise in \eqref{eq:primitive_full_1}-\eqref{eq:primitive_full_2}, i.e.\ having $\phi_n=\psi_n\equiv 0$.

Before going further, let us mention some %further 
more differences compared with \cite{Primitive1}. The energy estimates and the continuous dependence on the initial data of Theorem \ref{t:intro}\eqref{it:estimate_intro}-\eqref{it:estimate_intro_1} were not contained in \cite{Primitive1} and are based on the use of a recent stochastic Grownall's lemma proven in \cite[Appendix A]{AV_variational}. 
Finally, due to the presence of the term \eqref{eq:gradient_sigma_n} in the $v$-dynamics \eqref{eq:primitive_full_1}, we cannot allow for a strong-weak setting as in \cite[Section 3]{Primitive1}, i.e.\ considering \eqref{eq:primitive_full_1} in the strong setting (in the sense of Sobolev spaces) and \eqref{eq:primitive_full_2} in the weak analytic one. Hence we only consider the strong setting, i.e.\ both \eqref{eq:primitive_full_1} and \eqref{eq:primitive_full_2} are understood in the strong sense.

To conclude, let us anticipate that in Theorems \ref{t:global_primitive_strong_strong} and \ref{t:continuous_dependence} we can even allow $(\ktwon)_{n\geq 1}$ to depend on $(t,\om,x_{\h})$, but not on $x_3$. The physical relevance of the $x_3$-independence of $\ktwon$ is discussed in Remark \ref{r:sigma_independence}.
As for the $x_3$-independence of $\phi^j_n,\psi^j_n$ for $j\in \{1,2\}$ in Theorem \ref{t:intro}, the justification is via the hydrostatic approximation.

\subsection{On the physical derivations} 
\label{ss:physical_intro}
Besides the symmetry of the relations \eqref{eq:primitive_full_3}-\eqref{eq:primitive_full_4}, to motivate the presence of the non-isothermal balance \eqref{eq:primitive_full_4}, in Section \ref{s:physical_derivation} we provide two physical derivations of \eqref{eq:primitive_full}. In both derivations the condition \eqref{eq:primitive_full_4} appears naturally.
Following the strategy used in the deterministic framework, we derive \eqref{eq:primitive_full} by employing suitable stochastic variants of the \emph{Boussinesq} and the \emph{hydrostatic approximations}. In both cases, the main ideas are in the Boussinesq approximation. In fluid dynamics, the  Boussinesq approximation is employed in the study of \emph{buoyancy-driven} flows (also referred to as natural convection), and it is typically a good approximation in the context of oceanic flows. 
Roughly speaking, the idea behind the Boussinesq approximation is that, in a natural convection regime, the role of the compressibility is negligible in the inertial and the convection terms, but \emph{not} in the gravity term. More precisely, in the \emph{compressible} Navier-Stokes equations one assumes
\begin{equation}
\label{eq:boussinesq_approx_intro_1}
(\rho-\rho_{\reference})\big(\partial_t U -(U\cdot\nabla) U\big)\approx 0
\end{equation}
for some reference density $\rho_{\reference}>0$.
Here, $U$ and $\rho$ denote the velocity and density of the fluid, respectively. In our first approach to derive \eqref{eq:primitive_full}, borrowing some ideas from \emph{stochastic climate modeling} (see e.g.\ \cite{MTVE01}), we replace the right hand side in \eqref{eq:boussinesq_approx_intro_1} by a noisy term:
\begin{equation}
\label{eq:boussinesq_approx_intro_2}
(\rho-\rho_{\reference})\big(\partial_t U -(U\cdot\nabla) U\big) 
\approx\sum_{n\geq 1} \big[ (\rho-\rho_{\reference})\stocgrav_n-\nabla \wt{Q}_n\big]\,\dot{\beta}_t^n.
\end{equation} 
Here $\stocgrav_n\in \R^3$ is given and $\wt{Q}_n$'s are turbulent pressures that make the modelling assumption on the right-hand side in \eqref{eq:boussinesq_approx_intro_2} compatible with the divergence-free condition which follows from assuming $\rho\approx \rho_{\reference}$ in the density balance, cf.\ \eqref{eq:primitive_full_5} and \eqref{eq:Navier_Stokes_det_2}.  

At least formally, the right-hand side in \eqref{eq:boussinesq_approx_intro_2} has zero expectation (if we interpret the noise in the It\^o formulation). Hence, the approximation in \eqref{eq:boussinesq_approx_intro_2} is consistent with \eqref{eq:boussinesq_approx_intro_1} when considering expected values, and it can be seen as a refinement of the usual Boussinesq approximation. Employing the approximation \eqref{eq:boussinesq_approx_intro_2} and the hydrostatic approximation used in the deterministic case (see e.g.\ \cite{AG01_approx,KGHHKW20,LT19}) one obtains \eqref{eq:primitive_full} where $\ktwon=-\lambda \stocgrav_n^3$ for some $\lambda \in \R$, where $\stocgrav_n^3$ is the third component of $\stocgrav_n\in\R^3$. The reader is referred to Subsection \ref{ss:hydrostatic_approx} for more details.

Our second derivation of \eqref{eq:primitive_full} is based on a two-scale interpretation of the primitive equations. Indeed, as the small aspect ratio limit suggests, in the context of the primitive equations the horizontal and the vertical directions can be thought of as small and large scales, respectively. 
Hence, as usual in the literature (see e.g.\ \cite{BE12,DP22_two_scale,FP20_small,MTVE01}), it is physically reasonable to consider an additive noise (per unit of mass) on the small-scale dynamics. Eventually, such choice and a further variant of the Boussinesq and hydrostatic approximations lead to the system \eqref{eq:primitive_full}. Details on this approach can be found in Subsection \ref{ss:derivation_two_scale}.

\subsection{Comments on the literature}
Here we collect further references to the literature on primitive equations. Since the literature is extensive, we restrict to literature particularly relevant to this work, referring to the references in the cited works for a more extensive and complete overview.
 
In the deterministic setting, the primitive equations were first studied by J.~L.\ Lions, R.\ Teman, and S.\ Wang in a series of articles \cite{LiTeWa1,LiTeWa2,LiTeWa3}. There, the authors proved the existence of global Leray-Hopf type solutions for initial data $v_0\in L^2$. 
As for the Navier-Stokes equations, the uniqueness of such solutions is still open. Under additional regularity assumptions uniqueness holds, see \cite{Ju17}.
In the deterministic setting, a breakthrough result has been proven independently by C.\ Cao and E.S.\ Titi \cite{CT07} and R.M.\ Kobelkov \cite{Ko07} where they proved the \emph{global well-posedness} %in $H^1$ 
of the primitive equations via $L^{\infty}_t(H^1_x)\cap L^2_t(H^2_x)$ a-priori estimates provided $v_0\in H^1$.
See also \cite{Kukavica_2007} for other boundary conditions. 
The results of \cite{CT07,Ko07} have been extended to the $L^p$-setting by the second author and T.\ Kashiwabara in \cite{HK16}. Further results can be found in \cite{GGHHK20_analytic,GGHHK20_bounded,GGHK21_scaling}. See also \cite{HH20_fluids_pressure} for an overview.

\emph{Stochastic} versions of the primitive equations have been studied by 
several authors. 
Global well-posedness for pathwise strong solutions 
has been established for multiplicative white noise in time by A.~Debussche, N.~Glatt-Holtz and R.~Temam in \cite{DEBUSSCHE20111123} and the same authors with M.~Ziane in \cite{Debussche_2012}.  
There, the authors used a Galerkin approach to first show the existence of 
martingale solutions, and then strong existence is deduced via pathwise uniqueness and a Yamada-Watanabe-type result. The global existence of solutions is then shown by energy estimates where the noise is seen as a perturbation of the linear system. The drawback of this approach is that it needs some smoothness for the noise which for instance excludes the case of gradient or transport noises. 
 Z.~Brze\'{z}niak and J.~Slav\'{i}k in \cite{BS21} employed a similar approach to show local and global well-posedness of the primitive equations with small transport noise. The stochastic perturbation of the primitive equations considered in \cite{BS21} is such that it does not act directly on the pressure when turning to the question of global existence. 
This allows the authors of \cite{BS21} to overcome some of the difficulties that arose in \cite{DEBUSSCHE20111123,Debussche_2012}.
In \cite{Primitive1}, by combining energy estimates and the functional analytic setting of \cite{AV19_QSEE_1,AV19_QSEE_2} we were able to overcome such drawbacks in the presence of gradient and transport-type noises.

\subsection{Strategy and overview}
As in \cite{Primitive1}, we take another point of view on stochastic primitive equations like \eqref{eq:primitive_full} as compared to \cite{BS21,DEBUSSCHE20111123,Debussche_2012}. More precisely, we interpret the transport and gradient noise terms as a part of the linearized system. Hence we only need to impose conditions guaranteeing that this linearization is parabolic. Such conditions are known to be optimal in the parabolic setting. With this perspective, the local existence and blow-up criterion of Theorem \ref{t:local_primitive_strong_strong} follow easily from the theory of critical spaces for stochastic evolution equations developed by the first author and M.C.~Veraar in \cite{AV19_QSEE_1,AV19_QSEE_2}. 

Once having obtained local existence and blow-up criteria from the abstract setting of \cite{AV19_QSEE_1,AV19_QSEE_2}, 
we turn our attention to the global well-posedness which is the main point of the present manuscript. 
Here we follow the arguments of \cite{CT07}, where the authors show a-priori estimates in $L^{\infty}_t(H^1_x)\cap L^2_t(H^2_x)$ for $v$ as a by-product of several concatenated estimates. In \cite{CT07}, the core of the argument is an intermediate estimate involving the \emph{barotropic} and \emph{baroclinic modes} given by 
\begin{equation*}
\overline{v}=\fint_{-h}^0 v(\cdot,\zeta)\,\dd \zeta
\qquad
\text{ and }
\qquad
\wt{v}=v-\overline{v},
\end{equation*}
respectively. Note that this is also the strategy used in our previous work \cite{Primitive1}. However, in \cite{Primitive1,CT07}, the temperature acts in the $v$-equations only as a lower order term, and therefore it does \emph{not} play any role in the estimates involving $(\overline{v},\wt{v})$, see the discussed below Theorem \ref{t:intro}. The presence of $\T$ in \eqref{eq:primitive_full_4} (and hence the term \eqref{eq:gradient_sigma_n} in the $v$-dynamics) creates several additional terms in the estimates for $(\overline{v},\wt{v})$ which cannot be treated as lower-order contributions.  Such terms will be described extensively at the beginning of Section \ref{s:intermediate_estimate}. In particular, we need to estimate $(\overline{v},\wt{v})$ and $\T$ \emph{jointly} exploiting some further (subtle) cancellations appearing in the energy balances. 
In our derivation of the energy estimates for $(\overline{v},\wt{v})$, here and in \cite{Primitive1}, we follow the simplified approach due to the second author and T.~Kashiwabara in \cite{HK16} (also used in \cite{HH20_fluids_pressure}). There, for instance, the  $L^6$-estimates proven in \cite{CT07} are replaced by the (apparently) weaker $L^4$-estimates.

The paper is organized as follows.
\begin{itemize}
\item Section \ref{s:physical_derivation}: Physical derivations of \eqref{eq:primitive_full}.
\item Section \ref{s:strong_strong}: Statements of the local and global well-posedness results of \eqref{eq:primitive_full} in $H^1$.
\item Section \ref{s:proof_main_results}: Proof of the main results of Section \ref{s:strong_strong} taking for granted the energy estimates of Proposition \ref{prop:energy_estimate_primitive_strong_strong}. 
\item Section \ref{s:basic_estimate}: Basic energy estimates for $(v,\T)$.
\item Section \ref{s:intermediate_estimate}: Proof of the crucial intermediate estimate involving $(\overline{v},\wt{v})$ and other unknowns.
\item Section \ref{s:proof_energy_estimate_conclusion}: Proof of the energy estimates of Proposition \ref{prop:energy_estimate_primitive_strong_strong}.
\item Section \ref{s:Stratonovich}: Global well-posedness of \eqref{eq:primitive_full} with noise in Stratonovich form.
\end{itemize}

\subsection{Notation}
\label{ss:set_up}
Here we collect the main notation which will be used throughout the paper. By $C$ we denote a constant that may change from line to line and depends only on the parameters introduced in our main assumption, namely Assumption \ref{ass:well_posedness_primitive_double_strong} below.

For any integer $k\geq 1$, $s\in (0,\infty)$ and $p\in(1,\infty)$, $L^p(\Dom;\R^k)=(L^p(\Dom))^k$ denotes the usual Lebesgue space and $H^{s,p}(\Dom;\R^k)$ the corresponding Bessel-potential spaces. In this paper we also use the common abbreviation $H^{s}(\Dom;\R^k)\stackrel{{\rm def}}{=}H^{s,2}(\Dom;\R^k)$. 
Appropriate function spaces of divergence-free velocity fields will be introduced in Subsection \ref{ss:reformulation} and are denoted by $\Hs^s(\Dom)$ or $\Ls^s(\Dom)$. Function spaces which take also into account the boundary conditions \eqref{eq:boundary_conditions_intro} are defined in \eqref{eq:def_Hn}-\eqref{eq:def_Hr}.

Since $\Dom=\Tor^2\times (-h,0)$, we employ the natural splitting $x\mapsto (x_{\h},\z)$ where $x_{\h}\in \Tor^2$, $\z\in (-h,0)$ and the subscript $\h$ stands for \emph{horizontal}. Similarly, we set 
$$
\div_{\h}=\partial_{1} +\partial_{2} , \qquad \nabla_{\h}=(\partial_{1},\partial_{2}), \qquad 
\Delta_{\h}=\div_{\h}\nabla_{\h},
$$
and for a vector $y=(y^j)_{j=1}^3\in \R^3$ we write $y_{\h}=(y^j)_{j=1}^2$ for its horizontal component. In the same spirit, we also set
\begin{align*}
	(v\cdot\nabla_{\h}) v &\stackrel{{\rm def}}{=}\Big(\sum_{1\leq j\leq 2}  v^j \partial_j v^k\Big)_{k=1}^2,  & 
	(\phi_{n}\cdot\nabla) v&\stackrel{{\rm def}}{=}\Big(\sum_{1\leq j\leq 3} \phi_n^j \partial_j v^k\Big)_{k=1}^2,\\
	(v\cdot\nabla_{\h}) \T &\stackrel{{\rm def}}{=}\sum_{1\leq j\leq 2} v^j \partial_j \T , &
	(\psi_{n}\cdot\nabla) \T&\stackrel{{\rm def}}{=}\sum_{1\leq j\leq 3} \psi_n^j \partial_j \T.
\end{align*} 
Moreover, we also employ the following usual notation for the vertical average: 
$$
\textstyle{\fint}_{-h}^0\cdot\,\dd \zeta\stackrel{{\rm def}}{=}\frac{1}{h}
\textstyle{\int}_{-h}^0\cdot\,\dd \zeta.
$$

If no confusion seems likely, we write $L^2$, $H^{k}$, $\Hs^{k}$, $ L^2(\ell^2)$ and $ H^k(\ell^2)$ instead of $L^2(\Dom;\R^m)$, $H^{k}(\Dom;\R^m)$, $\Hs^{k}(\Dom)$, $L^2(\Dom;\ell^2(\N;\R^m))$ and $ H^k(\Dom;\ell^2(\N;\R^m))$ for some $m\geq 1$ etc. 

Finally, we collect the main probabilistic notation. Throughout the paper we fix a filtered probability space $(\O,\A,(\F_t)_{t\geq 0},\P)$, and we let $\E[\cdot]\stackrel{{\rm def}}{=}\int_{\O}\cdot \,\dd\P$. Moreover, $(\beta^n)_{n\geq 1}=(\beta^n_t\,:\,t\geq 0)_{n\geq 1}$ denotes a sequence of independent standard Brownian motion on the above-mentioned probability space. 
A stopping time $\tau:\O\to [0,\infty]$ is a measurable map such that $\{\tau\leq t\}\in \F_t$ for all $t\geq 0$.
For a stopping time $\tau$, we let
$
[0,\tau]\times \O\stackrel{{\rm def}}{=}\{(t,\om)\,:\, 0\leq \tau(\om)\leq t\}
$
and use analogous definitions for $[0,\tau)\times \O$ etc. 
%%%%%%
By $\Progress$ and $\Borel$ we denote the progressive and the Borel $\sigma$-algebra, respectively. Finally, for brevity, we say that a map $\Phi: \R_+\times \O\times \R^m\to \R$ is $\Progress\otimes \Borel$-measurable if $\Phi$ is $\Progress\otimes \Borel(\Dom)\otimes \Borel(\R^m)$-measurable.

\section{Physical derivations}
\label{s:physical_derivation}
In this section, we derive the primitive equations with non-isothermal turbulent pressure \eqref{eq:primitive_full}. In the deterministic framework, the primitive equations are derived from the \emph{compressible} Navier-Stokes equations by means of the \emph{Boussinsesq} and \emph{hydrostatic approximations}.
In the current section, following the same path, we propose two derivations of \eqref{eq:primitive_full} both based on suitable stochastic variants of these approximations. 
In the first derivation, given in Subsections \ref{ss:boussinesq_approx}-\ref{ss:hydrostatic_approx}, we motivate the noise leading to the non-isothermal turbulence balance \eqref{eq:primitive_full_4} by borrowing ideas from stochastic climate modelling (see e.g.\ \cite{AFP21,MTVE01} and the reference therein). In the second one, worked out in Subsection \ref{ss:derivation_two_scale}, we derive \eqref{eq:primitive_full} by looking at the Navier-Stokes equations as a two-scale system, where large and small scales are given by the horizontal and vertical ones, respectively; see Figure \ref{fig:domain}. As explained in Subsection \ref{ss:novelty} (see the text around \eqref{eq:gradient_sigma_n}), the presence of $\T$ in the balance  \eqref{eq:primitive_full_4} gives rise to a gradient type noise for the unknown $\int_{-h}^{\cdot}\T\,\dd \zeta$ in the equations for the horizontal part of the velocity field $v$.
Hence, the two-scale viewpoint is somehow in accordance with the results obtained in \cite{DP22_two_scale,FlaPa21}, where an additive noise on small-scale dynamics gives rise to a transport (or gradient) noise on large-scale ones.
For exposition convenience, in the first derivation of Subsections \ref{ss:boussinesq_approx}-\ref{ss:hydrostatic_approx}, to emphasise the natural appearance of the non-isothermal turbulent balance \eqref{eq:primitive_full_4}, we do not consider transport noise in the equations for the velocity. The former is included in the second derivation of Subsection \ref{ss:derivation_two_scale}. Let us anticipate that the derivations below also naturally lead to $x_3$-independence of $(\ktwon,\phi_n^j,\psi_n^j)$ for $j\in \{1,2\}$ used in our global well-posedness results of Theorems \ref{t:intro} and \ref{t:global_primitive_strong_strong}, see Remarks \ref{r:sigma_independence}-\ref{r:phi_independence}.
Finally, we mention that the primitive equations for the ocean are often formulated by adding an equation for the salinity. We do not consider this here, as the equation for the salinity has the same structure as the one for $\T$ and does not provide any new mathematical difficulty (see e.g.\ \cite{HHK16_salinity} and the reference therein).

\subsection{Stochastic Boussinesq approximation}
\label{ss:boussinesq_approx}
In fluid dynamics, the {\sc Boussinesq approximation} is employed in the study of \emph{buoyancy-driven} flows (also referred to as natural convection). As already mentioned in Subsection \ref{ss:physical_intro}, the idea behind the Boussinesq approximation is that, in a natural convection regime, the role of the compressibility is negligible in the inertial and the convection terms, but \emph{not} in the gravity term. 
Next, we propose an extension of such approximation in the context of stochastic Navier-Stokes equations.
Let us consider the compressible anisotropic Navier-Stokes equations on the $\varepsilon$-dependent domain 
$$\Dom_{\varepsilon}\stackrel{{\rm def}}{=}\Tor^2\times (-\varepsilon,0),$$ 
where $\varepsilon>0$ is a \emph{small} parameter which measures the smallness of the vertical direction,
see Figure \ref{fig:domain}; hence the velocity field $U:\R_+\times \O\times \Dom_{\varepsilon}\to \R^3$, the pressure $Q:\R_+\times \O\times \Dom_{\varepsilon}\to \R$ and the density $\rho:\R_+\times \O\times \Dom_{\varepsilon}\to \R_+$ satisfy, on $\Dom_{\varepsilon}$,
\begin{subequations}
\label{eq:Navier_Stokes_compressible_det}
\begin{alignat}{2}
\label{eq:Navier_Stokes_compressible_det_1}
		\begin{split}
		\rho \big(\partial_t  U+ (U\cdot\nabla)U\big)
		&= -\nabla Q+ \mu_{\h}\Delta_{\h}  U + \varepsilon^2 \partial_3^2 U+ \mu\nabla (\div \,U) - g \rho\,e_3,
		\end{split}\\
\label{eq:Navier_Stokes_compressible_det_2}
\partial_t \rho + \div (\rho \, U)&=0,
\end{alignat}
\end{subequations}
where $g$ and $\mu_{\h},\mu$ denote the gravity and the dynamic viscosities, respectively. In the above, as usual, we let $e_3=(0,0,1)$.
The anisotropic behaviour of the viscosity in \eqref{eq:Navier_Stokes_compressible_det} is in accordance with physical observations of oceanic flows, see e.g.\ \cite[Subsection 1.2.3]{HH20_fluids_pressure}.

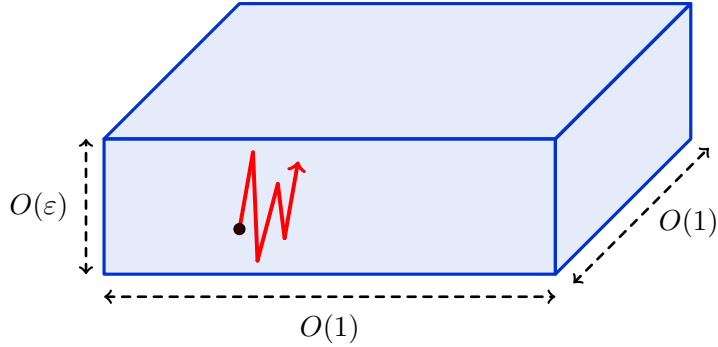
\begin{figure}
\definecolor{ffqqqq}{rgb}{1.,0.,0.}
\definecolor{ttqqqq}{rgb}{0.2,0.,0.}
\definecolor{qqttcc}{rgb}{0.,0.2,0.8}
\centering
\begin{tikzpicture}[line cap=round,line join=round,x=1.0cm,y=1.0cm,scale=0.6]
%\clip(4,-3) rectangle (16,7);
\path [line width=1.2pt,draw=qqttcc,fill=qqttcc,fill opacity=0.1] (0.,0.) rectangle (10.,3.);
\fill[line width=1.2pt,draw=qqttcc,fill=qqttcc,fill opacity=0.1] (0.,3.) -- (3.,6.) -- (13.,6.) -- (10.,3.) -- cycle;
\fill[line width=1.2pt,draw=qqttcc,fill=qqttcc,fill opacity=0.1] (10.,0.) -- (13.,3.) -- (13.,6.) -- (10.,3.) -- cycle;

%percorso particella
\draw [line width=1.5pt,color=ffqqqq] (3.,1.)-- (3.3,2.7);
\draw [line width=1.5pt,color=ffqqqq] (3.3,2.7)-- (3.4,0.3);
\draw [line width=1.5pt,color=ffqqqq] (3.4,0.3)-- (3.85,2.);
\draw [line width=1.5pt,color=ffqqqq] (3.85,2.)-- (4.,0.8);
\draw [->,line width=1.5pt,color=ffqqqq] (4.,0.8) -- (4.3,2.5);
\draw [fill=ttqqqq] (3.,1.) circle (3.5pt);

% vettori
\draw [<->,line width=1pt,dashed] (0.,-0.5) -- node[below,yshift=-0.1cm]{\large$O(1)$} (10.,-0.5);
\draw [<->,line width=1pt,dashed] (10.4,-0.2) -- node[right,xshift=0.1cm,yshift=-0.1cm]{\large$O(1)$} (13.4,2.8);
\draw [<->,line width=1pt,dashed] (-0.4,0) -- node[left,xshift=-0.1cm]{\large$O(\varepsilon)$} (-0.4,3);

\end{tikzpicture}
\caption{A particle subject to random forces in the thin domain $\Dom_{\varepsilon}=\Tor^2\times (-\varepsilon,0)$.}
\label{fig:domain}
\end{figure}

Let $\rho_{\reference}>0$ be a reference density, e.g.\ the density of the fluid in standard conditions. 
The {\sc stochastic Boussinesq approximation} consists of the following approximations:
\begin{enumerate}[{\rm(a)}]
\item\label{it:incompressible_boussinesq} Take $\rho\approx \rho_{\reference}$ in \eqref{eq:Navier_Stokes_compressible_det_2}, and therefore $\div\,U=0$.
\item\label{it:noise_boussinesq} Approximate all the terms in \eqref{eq:Navier_Stokes_compressible_det_1} which contain $\rho$ with a noise, expect for the buoyancy term $-g\rho\,e_3 $.
More precisely, in \eqref{eq:Navier_Stokes_compressible_det_1}, we use the following approximation
\begin{equation}
\label{eq:Boussinesq_approximation}
(\rho-\rho_{\reference})\big(\partial_t U -(U\cdot\nabla) U\big) 
\approx\sum_{n\geq 1} \big[ (\rho-\rho_{\reference})\stocgrav_{n,\varepsilon}-\nabla \wt{Q}_n\big]\,\dot{\beta}_t^n,
\end{equation}
where $(\beta^n)_{n\geq 1}$ is a family of independent standard Brownian motions, $\stocgrav_{n,\varepsilon}\in \R^3$ are given and $ \wt{Q}_n:\R_+\times \O\times \Dom_{\varepsilon}\to \R$ are suitable maps for $ n\geq 1$.
\end{enumerate}

Recall that, in the deterministic setting, the Boussinesq approximation consists in assuming \eqref{it:incompressible_boussinesq} and considering  $
(\rho-\rho_{\reference})\big(\partial_t U-(U\cdot\nabla) U\big)\approx 0$, see e.g.\ \cite[Subsection 1.2.2]{HH20_fluids_pressure}. The reason not to approximate the gravity term $- g\rho\, e_3$ is that, experimentally, in buoyancy-driven flows, such term is the most relevant in the dynamics and there is no natural approximation for it. At least formally, the right-hand side of \eqref{eq:Boussinesq_approximation} has zero expected value (if we interpret the noise in the It\^o formulation), cf.\ \cite[Assumption (A.4)]{MTVE01}. Hence, the modelling assumption on the right-hand side of \eqref{eq:Boussinesq_approximation} is consistent with the usual Boussinesq approximation when one considers expectations, and it can be seen as a refinement of the latter. The presence of the turbulent pressure $\wt{Q}_n$ on the right-hand side of \eqref{eq:Boussinesq_approximation} is necessary to obtain compatibility with the divergence-free condition $\div\,U=0$, see \eqref{it:incompressible_boussinesq} in the above list. 

To some extent, the approximation in \eqref{it:noise_boussinesq} follows the philosophy of \emph{stochastic climate modelling}, where there are certain unresolved variables (in our case $\rho-\rho_{\reference}$), and the main assumption is that the nonlinear interactions among unresolved variables can be represented stochastically. Such approximation has two basic advantages. Firstly, the noise keeps track of the approximations done in the balances ruling the dynamics, and secondly, the corresponding model has a reasonable complexity (both mathematically and computationally). The reader is referred to \cite{AFP21,MTVE01} and the references therein for more details on stochastic climate models.

Using the stochastic Boussinesq approximation of \eqref{it:incompressible_boussinesq}-\eqref{it:noise_boussinesq} in \eqref{eq:Navier_Stokes_compressible_det} we obtain, on $\Dom_{\varepsilon}$,
\begin{subequations}
\label{eq:Navier_Stokes_det}
\begin{alignat}{2}
\label{eq:Navier_Stokes_det_1}
		\begin{split}
		\dd  U&= \Big[ \nu_{\h}\Delta_{\h}  U + \frac{\varepsilon^2}{\rho_{\reference}} \partial_3^2 U-\frac{\nabla Q}{\rho_{\reference}} - ( U\cdot \nabla) U
		-g \frac{\rho}{\rho_{\reference}}\, e_3 \Big]\, \dd t \\
		& 
		+\sum_{n\geq 1} \Big[\frac{\rho-\rho_{\reference}}{\rho_{\reference}}  \stocgrav_{n,\varepsilon} - \frac{\nabla\wt{Q}_n  }{\rho_{\reference}}
		  \Big]\, \dd \beta_t^n,
		\end{split}\\
		\label{eq:Navier_Stokes_det_2}
		\div \, U&=0,
\end{alignat}
\end{subequations}
where, as usual, $\nu_{\h}\stackrel{{\rm def}}{=}\mu_{\h}/\rho_{\reference}$ denotes the viscosity while $\stocgrav_{n,\varepsilon}=(\varepsilon \stocgrav_{n,\h}, \stocgrav_n^3)$, $\stocgrav_{n,\h}\in \R^2$ and $\stocgrav_n^3\in \R$ are given. The anisotropic behaviour of $\stocgrav_n$ reflects the anisotropic viscosity in \eqref{eq:Navier_Stokes_det_1}.

To remove the dependency on $\rho$ in \eqref{eq:Navier_Stokes_det}, we use a state equation $\rho=\rho(\Theta)$ where $\Theta$ represents the fluid temperature. As standard in the context of primitive equations, we assume that $\Theta\mapsto\rho(\Theta)$ is linear, i.e.\ 
\begin{equation}
\label{eq:density_law_linear_det}
\rho= \rho_{\reference}+\lambda (\Theta-\Theta_{\reference}),
\end{equation}
where $\lambda\in \R$ and $\Theta_{\reference}$ denote a parameter to be determined experimentally and a reference temperature, respectively (other possible choices can be found in \cite{Korn21}).
To close the problem consisting of \eqref{eq:Navier_Stokes_compressible_det} and \eqref{eq:density_law_linear_det}, we need an equation for $\Theta$. By using the thermal balancing with constant density, one obtains, on $\Dom_{\varepsilon}$,
\begin{equation}
\label{eq:temperature_anisotropic_det}
\partial_t \Theta
=  \kappa_{\h}\Delta_{\h} \Theta + \varepsilon^2 \partial_3^2 \Theta  - (U\cdot \nabla) \Theta.
\end{equation}
%d \T
%= \Big[ \kappa_{\h}\Delta_{\h} \T + \varepsilon^2 \partial_3^2 \T  - (u\cdot \nabla) \T \Big]dt
%+\sum_{n\geq 1} \Big[ (\psi_{n,\h}\cdot\nabla_{\h}) \T + \varepsilon\phi^3_n \partial_3 \T \Big]d\beta_t^n.
In the above, as in \eqref{eq:Navier_Stokes_compressible_det}, we use anisotropic conductivity.  
%%%%
In the sequel, to simplify the presentation, we let 
\begin{equation}
\label{eq:simplification_in_coeffcients}
\nu_{\h}=\rho_{\reference}=\lambda=1 \quad \text{ and } \quad \Theta_{\reference}=0.
\end{equation}
The general case is similar (note that \eqref{eq:temperature_anisotropic_det} is also satisfied by $\Theta-\Theta_{\reference}$ for all $\Theta_{\reference}\in\R $).

\subsection{Stochastic hydrostatic approximation}
\label{ss:hydrostatic_approx}
Roughly speaking, the {\sc hydrostatic approximation} consists in
%of
 neglecting several terms in the dynamics for the vertical component of the velocity field  $U$. 
From a mathematical point of view, we would like to take the limit %as
 $\varepsilon\downarrow 0$ in \eqref{eq:Navier_Stokes_det} and \eqref{eq:temperature_anisotropic_det}. To this end, it is convenient to rescale the vertical variable $\z$ to obtain a problem on the fixed domain $\Dom\stackrel{{\rm def}}{=}\Dom_1=\Tor^2\times (-1,0)$. Moreover, to accommodate the anisotropic behaviour of viscosity and conductivity in \eqref{eq:Navier_Stokes_det} and \eqref{eq:temperature_anisotropic_det}, we let 
$$
U=(V,W) \ \  \text{ where }\   V\in 	\R^2 \ \text{ and } \ W\in \R.
$$ 
In other words, $V$ and $W$ are the horizontal and the vertical parts of the velocity field $U$, respectively.
Let $\varepsilon>0$ and consider the rescaled quantities: For $t\in [0,\infty)$, $x_{\h}\in \Tor^2$  and $\z\in (-1,0)$,
%%%%%%%
\begin{align}
\nonumber
 v_{\varepsilon}(t,x)&\stackrel{{\rm def}}{=}V(t,x_{\h},\varepsilon \z), &
w_{\varepsilon}(t,x)&\stackrel{{\rm def}}{=}\varepsilon^{-1} W(t,x_{\h},\varepsilon \z),\\ 
\label{eq:rescaling}
   \T_{\varepsilon}(t,x)&\stackrel{{\rm def}}{=}\varepsilon\,\Theta(t,x_{\h},\varepsilon \z),\qquad \qquad &\qquad \qquad  &\\
\nonumber
   P_{\varepsilon}(t,x)&\stackrel{{\rm def}}{=}Q(t,x_{\h},\varepsilon \z), 
& \wt{P}_{\varepsilon,n}(t,x)&\stackrel{{\rm def}}{=}
   \wt{Q}_n(t,x_{\h},\varepsilon \z).  
\end{align}
The choice of the rescaling is the one used in the deterministic setting, and it reflects the natural size of the corresponding quantities, see e.g.\ \cite{KGHHKW20,LT19} and \cite{PZ22} for the rescaling of $\Theta$. 
%The above rescaling also reflects the typical magnitude of the corresponding variables in oceanic flows. 

Note that $(v_{\varepsilon},w_{\varepsilon},P_{\varepsilon},\wt{P}_{\varepsilon},\T_{\varepsilon}) $ are defined on the fixed domain $\Dom=\Tor^2\times (-1,0)$. From \eqref{eq:Navier_Stokes_det} and \eqref{eq:temperature_anisotropic_det}, we infer that, on $\Dom$, 
\begin{subequations}
\label{eq:Navier_Stokes_varepsilon}
\begin{alignat}{4}
\label{eq:Navier_Stokes_varepsilon_1}
\begin{split}
\dd v_{\varepsilon}&= \Big[ \Delta_{\h} v_{\varepsilon} + \partial_{3}^2 v_{\varepsilon} -\nabla_{\h} P_{\varepsilon} 
- (u_{\varepsilon}\cdot \nabla) v_{\varepsilon} \Big]\, \dd t + \sum_{n\geq 1} \Big[ \T_{\varepsilon}\stocgrav_{n,\h} - \nabla_{\h} \wt{P}_{\varepsilon,n} \Big]\, \dd \beta_t^n,
\end{split}
\\
%%%%%
\label{eq:Navier_Stokes_varepsilon_2}
\begin{split}
\dd  (\varepsilon^2 w_{\varepsilon})
&=\Big[ \varepsilon^2 \big( \Delta_{\h} w_{\varepsilon} + \partial_{3}^2 w_{\varepsilon} 
- (u_{\varepsilon}\cdot \nabla) w_{\varepsilon}\big)-\partial_{3} P_{\varepsilon} -  g \T_{\varepsilon}+ \varepsilon g \Big]\, \dd t \\
& \ \ \quad +
\sum_{n\geq 1} \big[- \partial_{3}  \wt{P}_{\varepsilon,n}+  k_{n}^3 \T_{\varepsilon} \big]\, \dd \beta_t^n,
\end{split}\\
\label{eq:Navier_Stokes_varepsilon_3}
\dd  \T_{\varepsilon}
&= \big[ \Delta \T_{\varepsilon}  - (u_{\varepsilon}\cdot \nabla) \T_{\varepsilon} \big]\,  \dd t,\\
%%%
\div \,u_{\varepsilon}&=0.
\end{alignat}
\end{subequations}
%%%%%%%%
The {\sc stochastic hydrostatic approximation} consists in taking the formal limit %as 
$\varepsilon\downarrow 0$ in \eqref{eq:Navier_Stokes_varepsilon} and assuming that the quantities in $(v_{\varepsilon},P_{\varepsilon}, \wt{P}_{\varepsilon}, \T_{\varepsilon})$ converge (in suitable function spaces) and 
\begin{equation}
\label{eq:w_varepsilon_goes_to_zero}
\lim_{\varepsilon\to 0}\varepsilon^2 w_{\varepsilon}= 0, \qquad \qquad   \lim_{\varepsilon\to 0} \varepsilon^2 \big( \Delta_{\h} w_{\varepsilon} + \partial_{3}^2 w_{\varepsilon} 
- (u_{\varepsilon}\cdot \nabla) w_{\varepsilon} \big)= 0.
\end{equation}
The reader is referred to \cite{Ped} for physical reasons for the approximation to hold. We recall that the limits \eqref{eq:w_varepsilon_goes_to_zero} are justified in the deterministic setting, see e.g.\ \cite{KGHHKW20,LT19}.

\smallskip

Assume that the hydrostatic approximation holds and denote by $(v,P, \wt{P}, \T)$ the limit as $\varepsilon\downarrow 0$ of $(v_{\varepsilon},P_{\varepsilon}, \wt{P}_{\varepsilon}, \T_{\varepsilon})$.
By \eqref{eq:Navier_Stokes_varepsilon_1} and \eqref{eq:Navier_Stokes_varepsilon_2}, at least formally, one sees that $(v,P, \wt{P},\T)$ solve \eqref{eq:primitive_full_1} and \eqref{eq:primitive_full_2} where $\gvn=\T\stocgrav_{n,\h}$ and $\fv=\ft=\gtn\equiv 0$. While
%, 
using \eqref{eq:w_varepsilon_goes_to_zero} and \eqref{eq:Navier_Stokes_varepsilon_2}, one obtains \eqref{eq:primitive_full_3} and \eqref{eq:primitive_full_4} with $\kone= g$ and $\ktwon=- \stocgrav_n^3$, respectively. 
Therefore \eqref{eq:primitive_full} follows from \eqref{eq:Navier_Stokes_varepsilon} by means of the stochastic hydrostatic approximation.

\begin{remark}[$x_3$-independency of $\ktwon$]
\label{r:sigma_independence}
In our main result, i.e.\ Theorem \ref{t:global_primitive_strong_strong}, we assume that $\ktwon$ depends on $(t,\om,x_{\h})\in \R_+\times\O\times \Tor^2$, cf.\ Assumption \ref{ass:global_primitive_strong_strong} below. Here we discuss how the $x_3$-independence arises naturally from the stochastic hydrostatic approximation. Indeed, let us assume that the maps $k_{n,\varepsilon}$ are $(t,\om,x)$-dependent and \emph{consistent} in $\varepsilon>0$, i.e.\ there exists a map $K_n$ on $\R_+\times \O\times \Dom_{\varepsilon_0}$ where $\varepsilon_0>0$ satisfying $k_{n,\varepsilon}=K_n$ on $\R_+\times \O\times \Dom_{\varepsilon}$ for all $\varepsilon\leq \varepsilon_0$. From a modelling point of view, it is reasonable to assume that  $K_n$ is continuous in $x\in \Dom_{\varepsilon_0}$. Below, for simplicity, we take $\varepsilon_0=1$. Then repeating the argument in \eqref{eq:Navier_Stokes_det}-\eqref{eq:rescaling} leading to the stochastic primitive equations \eqref{eq:primitive_full}, one obtains in \eqref{eq:Navier_Stokes_varepsilon_2} that the stochastic perturbation is of the form 
$
\sum_{n\geq 1}\big[- \partial_{3}  \wt{P}_{\varepsilon,n}(t,x)+  K_{n}^3(t,x_{\h},\varepsilon x_3) \T_{\varepsilon}(t,x) \big]\, \dd  \beta^n_t.
$
In particular, if the stochastic hydrostatic approximation \eqref{eq:w_varepsilon_goes_to_zero} holds, then the limiting balance \eqref{eq:primitive_full_4} is satisfied with $\ktwon(t,\om,x_{\h})=- k_{n,\varepsilon}^3(t,\om,x_{\h},0)$ for any fixed $\varepsilon\in (0,1)$ (here we used the continuity of $K_n$ and that $k_{n,\varepsilon}(t,\om,x_{\h},0)$ is independent of $\varepsilon$ by consistency). A similar situation arises if we also assume $\lambda$ in \eqref{eq:density_law_linear_det} to be $(t,\om,x)$-dependent instead of \eqref{eq:simplification_in_coeffcients}.

Let us conclude by noticing that, if in the above argument, one assumes $k_n(t,\om,x_{\h},x_3)=K_n^3(t,\om,x_{\h},\varepsilon^{-1}x_3)$ for some mapping $K_n$ on $\R_+\times \O\times \Dom_1$, then the stochastic hydrostatic approximation eventually leads to $x_3$-dependent $\ktwon$'s. However, in the authors' opinion, the latter choice does not seem physically relevant.  Indeed, in the spirit of Boussinesq approximations, one wants to obtain a reduced model from the Navier-Stokes equations by neglecting detailed information about the vertical dynamics, and this is in contrast 
to
%with 
the rescaling of the vertical direction, which increases the effect of the vertical dynamics on the limiting SPDEs as  $\lim_{\varepsilon\to 0}\| K_n^3(x_\h,\varepsilon^{-1}\cdot)\|_{C^{\alpha}(0,\varepsilon)}= \infty$ for all $\alpha>0$ and $x_\h\in \Tor^2$ even if $K_n$ is smooth. To the authors' knowledge, in the literature, there is no derivation of the primitive equations with $x_3$-dependent coefficients, and therefore we cannot compare our situation with known results. 

We conclude this remark by highlighting that in Section \ref{s:intermediate_estimate} we show that the $x_3$-independence of $\ktwon$ allows us to obtain a meaningful splitting of the stochastic primitive equations \eqref{eq:primitive_full} in terms of the \emph{barotropic} and \emph{baroclinic modes}, whose relevance is commented in Remark \ref{r:phi_independence} below.
\end{remark}

\subsection{A related derivation and the two-scale viewpoint}
\label{ss:derivation_two_scale}
In this section, we give another derivation of \eqref{eq:primitive_full} still based on the Boussinesq and hydrostatic approximation. Here the main starting point is a \emph{two-scale} interpretation of the Navier-Stokes equations for the velocity field $U=(V,W)\in \R^2\times \R$ on the thin domain $\Dom_{\varepsilon}$. Indeed, as Figure \ref{fig:domain} suggests, the Navier-Stokes equations on the thin-domain $\Dom_{\varepsilon}$ can be seen as a two-scale system where the large-scale dynamics is the horizontal component of the velocity field, i.e.\ $V$, and the small dynamics is the vertical component, i.e.\ $W$. 
Since $W$ is somehow a small scale, from a physical point of view it is natural to consider additive noise on this component, see e.g.\ \cite{AFP21,DP22_two_scale,FP20_small,FlaPa21,MTVE01}. 

To make this rigorous, as a starting point, assume that $U:\R_+\times \O\times \Dom_{\varepsilon}\to \R^3$, the pressures $Q,\wt{Q}_n:\R_+\times \O\times \Dom_{\varepsilon}\to \R$ and the density $\rho:\R_+\times \O\times \Dom_{\varepsilon}\to \R_+$ satisfy, on $\Dom_{\varepsilon}$,
\begin{subequations}
\label{eq:Navier_Stokes_compressible}
\begin{alignat}{2}
\label{eq:Navier_Stokes_compressible_1}
		\begin{split}
		\rho\, \dd  U&= \Big[ \mu_{\h}\Delta_{\h}  U + \varepsilon^2 \partial_3^2 U+ \mu\nabla (\div \,U) -\nabla Q - 
		\rho\, (U\cdot\nabla)U- g \rho\,e_3\Big]\,\dd t\\
		&\quad \ \ +\sum_{n\geq 1} \Big[ (\Phi_{n,\h}\cdot\nabla_{\h}) U + \varepsilon\,\Phi^3_n \partial_3 U -
		 \nabla \wt{Q}_n+\stocgrav_{n,\varepsilon}( \rho) \Big]\, \dd \beta_t^n,
		\end{split}\\
\label{eq:Navier_Stokes_compressible_2}
\partial_t \rho + \div (\rho \, U)&=0.
\end{alignat}
\end{subequations}
A derivation of \eqref{eq:Navier_Stokes_compressible} is given, for instance, in \cite{MR01,MR04}. In the latter works, transport noise is a consequence of a stochastic dynamic at the level of fluid particles, see \cite[eq.\ (2.4)]{MR04}.  Here, as above, $(\beta^n)_{n\geq 1}$ is a sequence of independent standard Brownian motions on some probability space, $-g\,e_3$ is the gravity vector, $\mu_{\h},\mu$ are the dynamic viscosities and $\Phi_{n,\h}\in \R^2,\Phi^3_n\in \R$ are given.
In \eqref{eq:Navier_Stokes_compressible_1} we used anisotropic viscosity as in \eqref{eq:Navier_Stokes_compressible_det}, which is in accordance with measurements in oceanic flows.  
The anisotropic behaviour of the transport noise reflects the different order of the leading differential operators in the deterministic and stochastic terms. The latter fact is a consequence of the different 
 (time) scaling of the Brownian noise $\dd \beta_t^n$ and the time $\dd t$, see e.g.\ \cite[Subsection 1.1]{AV21_NS} for a discussion. Finally, $\stocgrav_n(\rho)$ is a given function of the density $\rho$. Results on compressible Navier-Stokes equations can be found in  \cite{BFH20_compressible1,BFH21_compressible2,BFHM19_compressible3} and the references therein.

Next we add a structural assumption on $\stocgrav_{n,\varepsilon}(\rho) =(\stocgrav_{n,\varepsilon,\h}(\rho),\stocgrav_{n,\varepsilon}^3(\rho))$, where $\stocgrav_{n,\varepsilon,\h}(\rho)\in \R^2$ and $\stocgrav_{n,\varepsilon}^3(\rho)\in \R$. More precisely, we assume that 
\begin{align}
\label{eq:k_n_assumption_j}
k_{n,\varepsilon,\h} (\rho)&=\wt{k}_{n,\h}(\varepsilon \rho) ,
& \text{ where }&\wt{k}_{n,\h}:\R_+\to \R\text{ is a given nonlinearity},\\
\label{eq:k_n_assumption_3}
k^3_{n,\varepsilon} (\rho) &=\wt{k}_n^3 \rho & \text{ where }&\wt{k}_n^3\in \R.
\end{align}
The condition \eqref{eq:k_n_assumption_3} tells us that, on the vertical component, an \emph{additive noise per unit of mass} is acting. As mentioned above, this is in accordance with the two-scale interpretation of the Navier-Stokes equations \eqref{eq:Navier_Stokes_compressible} in the thin domain $\Dom_{\varepsilon}$. 
The condition \eqref{eq:k_n_assumption_j} is somehow technical and it is motivated by the scaling argument as in \eqref{eq:rescaling} which will be used below. However, let us stress that, for our purposes, the crucial assumption is \eqref{eq:k_n_assumption_3}.

Now following the scheme of Subsections \ref{ss:boussinesq_approx}-\ref{ss:hydrostatic_approx}, one can derive \eqref{eq:primitive_full} from \eqref{eq:Navier_Stokes_compressible} and the structural assumptions \eqref{eq:k_n_assumption_j}-\eqref{eq:k_n_assumption_3} performing the following steps:
\begin{itemize}
\item {\sc (Stochastic Boussinesq approximation II)}. Assume that the density is  \emph{constant} (i.e.\  $\rho\approx \rho_{\reference}$ for some reference density $\rho_{\reference}>0$) in all terms in \eqref{eq:Navier_Stokes_compressible} expect in the buoyancy term $-g \rho \,e_3$ and its stochastic counterpart $\stocgrav_n (\rho) $.
%%%%
\item{\sc (Heat balance and state equation II)}. The heat balance shows that the temperature $\Theta$ evolves according to the equations 
\begin{equation}
\label{eq:temperature_eq_II}
\dd  \Theta
= \big[ \kappa_{\h}\Delta_{\h} \Theta + \varepsilon^2 \partial_3^2 \Theta  - (U\cdot \nabla) \Theta \big]\, \dd t
+\sum_{n\geq 1} \big[ (\Psi_{n,\h}\cdot\nabla_{\h}) \Theta + \varepsilon\,\Psi^3_n \partial_3 \Theta \big]\, \dd \beta_t^n,
\end{equation}
where $(\Psi_n)_{n\geq 1}$ is a sequence of vector fields. 
Finally, as a state equation $\rho=\rho(\Theta)$, use the linear map $\rho=\rho_{\reference}+\lambda(\Theta-\Theta_{\reference})$ where $\lambda,\Theta_{\reference}\in \R$ are given.
%%%%
\item{\sc (Stochastic hydrostatic approximation II)}. The hydrostatic approximation can be performed as in Subsection \ref{ss:hydrostatic_approx}, where one also needs to add in \eqref{eq:w_varepsilon_goes_to_zero} the requirement
\begin{equation}
\label{eq:gradient_convergence}
 \lim_{\varepsilon\to 0} \varepsilon^2 \big[(\Phi_n\cdot\nabla) w_{\varepsilon} \big]= 0,
\end{equation}
with $w_\varepsilon$ as in \eqref{eq:rescaling}. 
Let us stress that, in the deterministic setting  \cite{KGHHKW20,LT19},  one can even prove that \eqref{eq:w_varepsilon_goes_to_zero} holds  and $\varepsilon^2 \Delta w_{\varepsilon} \stackrel{\varepsilon\downarrow 0}{\to} 0$. Hence, it seems that \eqref{eq:gradient_convergence} is no more demanding than the requirements in  \eqref{eq:w_varepsilon_goes_to_zero}. 
 \end{itemize}

We conclude this subsection by admitting that there is no direct relation between the above derivation with the two-scale arguments in  \cite{DP22_two_scale,FlaPa21,MTVE01}. It would be interesting to study which contribution(s) need to be considered in the small scale equation of \cite[Subsection 7.3]{DP22_two_scale} to obtain the non-isothermal balance for $\wt{P}_n$ of \eqref{eq:primitive_full_4} for the effective dynamics.

\begin{remark}[$x_3$-independence of transport noise]
\label{r:phi_independence}
Arguing as in Remark \ref{r:sigma_independence}, if one assumes that $\Phi_{n,\h}$ and $\Phi_{n}^3$ are $(t,\om,x)$-dependent and consistent in $\varepsilon>0$ with consistency map that is continuous in $x\in \Dom_{\varepsilon_0}$ for some $\varepsilon_0>0$, then the stochastic hydrostatic approximation (i.e.\ \eqref{eq:w_varepsilon_goes_to_zero} and \eqref{eq:gradient_convergence} both hold) leads to the transport noise coefficients 
$
\phi_n(t,\om,x_{\h})=(\Phi_{n,\h}(t,\om,x_{\h},0),\Phi^3_n(t,\om,x_{\h},0))
$
in \eqref{eq:primitive_full_1}. Let us remark that the continuity of the consistency map is satisfied in the physically relevant case of the Kraichnan noise (see e.g.\ the discussion below \cite[eq.\ (1.3)]{MR05}). 
Therefore the $x_3$-independence condition of Assumption \ref{ass:global_primitive_strong_strong} is in accordance with the physical derivation.  

As in Remark \ref{r:sigma_independence}, the $x_3$-independence of $\phi_n$ arises if and only if one rescales also the vertical variable by $\varepsilon^{-1}$. As in the latter remark, in the authors' opinion, on the one hand, this seems unreasonable for the horizontal part of $\Phi_n$, i.e.\ $\Phi_{n,\h}$. On the other hand, rescaling of $\Phi_{n}^3$ might be physically relevant as in \eqref{eq:Navier_Stokes_compressible_1} we are weakening the strength of the contribution $\Phi^3_n \partial_3 v$ via the multiplication by $\varepsilon$. Thus, if one assumes $\Phi^3_n =\Phi^3_n(t,\om,x_{\h},\varepsilon^{-1}x_3)$, then this leads to an $x_3$-dependent $\phi^3_n=\Phi^3_n$. The latter situation is also covered by our results as in Assumption \ref{ass:global_primitive_strong_strong} no condition on the vertical component of $\phi$ is enforced. 

As it follows from Section \ref{s:intermediate_estimate}, the $x_3$-independence of $\phi_{n,\h}$ is \emph{necessary} for the stochastic primitive equations \eqref{eq:primitive_full} to behave well under the decomposition into barotropic and baroclinic modes, i.e.\ $v=\overline{v}+\wt{v}$ with $\overline{v}=\fint_{-h}^0 v(\cdot,\zeta)\,\dd \zeta$. The latter is very important in physics and in particular for the study of oceanic dynamics, see e.g.\ \cite{CH19_vbar,DHCWL95_vbar,HS97_vbar,OL07_vbar,SB99_vbar,YTLR17_vbar}. 

We conclude by noticing that the above arguments holds with $(\Phi_n,\phi_n)$ replaced by $(\Psi_n,\psi_n)$ which appear in the temperature balance \eqref{eq:temperature_eq_II} in case $\Psi_n$ is $x$-dependent.
\end{remark}

\begin{remark}[Two-dimensional turbulence]
\label{r:twod_turbulence}
The 2D nature of the transport noise for the stochastic primitive equations arose in the above introduced stochastic hydrostatic approximation (cf.\ Remark \ref{r:phi_independence}) is in accordance with physical measurements of turbulent flows in the ocean, as the latter show that turbulent oceanic flows are (approximately) two dimensional, see e.g.\ 
\cite{BE12,C01_twod,R73_twod,T02_twod,Y088_twod}.
\end{remark}

\section{Local and global well-posedness}
\label{s:strong_strong}
In this section, we state our main results on local and global well-posedness of \eqref{eq:primitive_full}. 
Actually, we will consider the following generalization of \eqref{eq:primitive_full}:
\begin{subequations}
\label{eq:primitive}
\begin{alignat}{7}
\label{eq:primitive_1}
		\begin{split}
		&\dd v -\Delta v\, \dd t=\Big[-\nabla_{\h} P  -(v\cdot \nabla_{\h})v- w\partial_{3} v + \fv(\cdot,v,\T,\nabla v,\nabla \T)
		+\partial_{\hp} \wt{p} \Big]\, \dd t \\
		&\qquad \qquad \ \ \ +\sum_{n\geq 1}\Big[(\phi_{n}\cdot\nabla) v-\nabla_{\h}\wt{P}_n  +\gvn(\cdot,v,\T)\Big] 			\, \dd \beta_t^n,
		\end{split}\\
\label{eq:primitive_2}
		\begin{split}
		&\dd  \T -\Delta \T\, \dd t=\Big[ -(v\cdot \nabla_{\h})\T- w\partial_{3} \T+ \ft(\cdot,v,\T,\nabla v,\nabla \T) 			
		\Big]\, \dd t \\
		&\qquad\qquad\ \ \  +\sum_{n\geq 1}\Big[(\psi_{n}\cdot\nabla) \T+\gtn(\cdot,v,\T)\Big]\,\dd \beta_t^n,\\
		\end{split}\\
\label{eq:primitive_3}
		&\partial_{\hp} \wt{p}\stackrel{{\rm def}}{=}\Big(\sum_{n\geq 1}\sum_{j=1}^2 \hp^{j,k}_n \big[\partial_j \wt{P}_n+ \partial_j \int_{-h}^{\cdot} \ktwon(\cdot,\zeta)\T(\cdot,\zeta)\,\dd \zeta\big]\Big)_{k=1}^2,\\
\label{eq:primitive_4}
&\partial_{3} P+ \kone \T+ (\tp \cdot\nabla )\T=0,\\
\label{eq:primitive_5} 
&\partial_{3} \wt{P}_n+\ktwon \T=0,\\
\label{eq:primitive_6}
&\div_{\h} v+\partial_{3} w=0,\\
\label{eq:primitive_7}
&v(0,\cdot)=v_0,\qquad \T(0,\cdot)=\T_0,
\end{alignat}
\end{subequations}
where the above equations hold on $\Dom=\Tor^2\times (-h,0)$. In \eqref{eq:primitive_3}, with $\int_{-h}^{\cdot} \ktwon(\cdot,\zeta)\T(\cdot,\zeta)\,\dd \zeta$ we understand the mapping
 $x\mapsto \int_{-h}^{x_3}\ktwon(x_{\h},\zeta)\T(x_{\h},\zeta)\,\dd \zeta $ where  $x=(x_{\h},x_3)\in \Dom$ with $x_{\h}\in \Tor^2$ and $x_3\in (-h,0)$. A similar notation will be also employed in the sequel if no confusion seems likely.

There are two additional terms in \eqref{eq:primitive} compared to \eqref{eq:primitive_full}. 
Firstly, \eqref{eq:primitive_1} contains the additional term $\partial_{\hp} \wt{p}$ defined in \eqref{eq:primitive_3} which takes into account the effect of the \emph{hydrostatic} turbulent pressure $\wt{p}_n$ (defined in \eqref{eq:integration_vertical_direction_3} below) on the deterministic component of the dynamics of $v$, i.e.\ \eqref{eq:primitive_1}. 
A similar term was also considered in \cite{Primitive1}. 
Secondly, in \eqref{eq:primitive_5}, there is an additional transport type term $(\tp \cdot\nabla)\T$ which is also due to the effect of the turbulent pressure. Both terms $\partial_{\hp} \wt{p}$ and $(\tp \cdot\nabla)\T$ are motivated by the Stratonovich formulation of \eqref{eq:primitive}. The reader is referred to Section \ref{s:Stratonovich} for further discussion. Let us mention that the term $(\tp \cdot\nabla)\T$ gives rise to the same mathematical difficulties of $\ktwon \T$ in \eqref{eq:primitive_5}, and therefore the problem \eqref{eq:primitive_1} is as (analytically) difficult as \eqref{eq:primitive_full}.
Finally let us note that, comparing \eqref{eq:primitive_full} and \eqref{eq:primitive}, the terms $(\fv,\ft,\gvn,\gtn)$ are $(v,\T)$-dependent nonlinearities. 

The system \eqref{eq:primitive} is complemented with the following boundary conditions on $\Tor^2$:
\begin{subequations}
\label{eq:boundary_conditions_full}
\begin{alignat}{3}
	\label{eq:boundary_conditions_full_1}
		\partial_{3} v (\cdot,-h)=\partial_{3} v(\cdot,0)&=0,\\
	\label{eq:boundary_conditions_full_2}
		\partial_{3} \T(\cdot,-h)= \partial_{3} \T(\cdot,0)+\alpha \T(\cdot,0)&=0,&\\
	\label{eq:boundary_conditions_full_3}
		w(\cdot,-h) =w(\cdot,0)&=0,
\end{alignat}
\end{subequations}
%%%%
where $\alpha\in \R$ is a given constant. 
As mentioned in Section \ref{s:intro}, the results below are also true in case \eqref{eq:boundary_conditions_full_1}-\eqref{eq:boundary_conditions_full_2} are replaced by periodic boundary conditions, see Remark \ref{r:periodic_BC}.
This section is organized as follows:
\begin{itemize}
\item In Subsection \ref{ss:reformulation} we reformulate \eqref{eq:primitive}-\eqref{eq:boundary_conditions_full} as a stochastic evolution equations for the unknown $(v,\T)$. To this end, we introduce the hydrostatic Helmholtz projection and appropriate function spaces of divergence-free vector fields.
\item In Subsection \ref{ss:strong_strong_statement} we collect the main assumptions and definitions. In particular, we provide a rigorous definition of solutions to \eqref{eq:primitive}-\eqref{eq:boundary_conditions_full} using It\^{o} calculus.
\item In Subsection \ref{ss:statement_main_results} we state local and global well-posedness results for 
\eqref{eq:primitive}-\eqref{eq:boundary_conditions_full}.
\end{itemize}

\subsection{Hydrostatic Helmholtz projection and reformulation of \eqref{eq:primitive}}
\label{ss:reformulation}
Let us begin by introducing the \emph{Helmholtz projection} on the horizontal variables $x_{\h}\in \Tor^2$ which will be denoted by $\pr$. Let $f\in L^2(\Dom;\R^2)$ and denoted by $\qr f\stackrel{{\rm def}}{=}\nabla_{\h} \Psi_f\in L^2(\Tor^2;\R^2)$ where $\Psi_f\in H^{1}(\Tor^2)$ is the unique solution to
$$
\Delta_{\h} \Psi_f=\div_{\h} f \ \  \text{ on }\Tor^2, \qquad\text{and }\qquad
\int_{\Tor^2} \Psi_f \,\dd x=0.
$$
Then the \emph{Helmholtz projection} on $\Tor^2$ is given by
$$
\pr f\stackrel{{\rm def}}{=} f-\qr f , \quad \text{ for }f\in L^2(\Tor^2;\R^2).
$$
It is easy to see that $\pr\in \calL(L^2(\Tor^2;\R^2))$ and that it is an orthogonal projection.
The \emph{hydrostatic Helmholtz projection} on $\Dom$ will be denote by $\p$, and it defined  for all $f\in L^2(\Dom;\R^2)$ by
(recall that $\fint_{-h}^0 \cdot\, \dd \zeta=\frac{1}{h}\int_{-h}^0\cdot\, \dd \zeta$)
\begin{equation}
\label{eq:Helmholtz_hydrostatic}
\q f= \qr \Big[\fint_{-h}^0 f (\cdot,\zeta)\, \dd \zeta \Big] \qquad \text{ and }\qquad
\p f  \stackrel{{\rm def}}{=} f -\q f .
\end{equation}
One can check that $\p\in \calL(L^2(\Dom;\R^2))$, it is an orthogonal projection and $\div_{\h} \int_{-h}^0 (\p f(\cdot,\zeta))\,\dd \zeta=0$ in $\D'(\Tor^2)$ for all $f\in L^2(\Dom;\R^2)$.
Let 
\begin{align*}
\Ls^2(\Dom)=\Big\{f\in L^2(\Dom;\R^2) \,:\, \div_{\h} \Big(\int_{-h}^0 f(\cdot,\zeta)\,\dd \zeta\Big)=0\text{ on }\Tor^2 \Big\},
\end{align*}
be endowed with the norm $\|f\|_{\Ls^2(\Dom)}\stackrel{{\rm def}}{=}\|f\|_{L^2(\Dom;\R^2)}$, and for all $k\geq 1$ we set
$$
\Hs^{k}(\Dom)\stackrel{{\rm def}}{=}H^{k}(\Dom;\R^2)\cap \Ls^{2}(\Dom), \qquad \|f\|_{\Hs^k(\Dom)}
\stackrel{{\rm def}}{=}\|f\|_{H^k(\Dom;\R^2)}.
$$
As in Subsection \ref{ss:set_up}, for $\A\in \{\Ls^2,\Hs^k\}$, we write $\A$ instead of $\A(\Dom)$, if no confusion seems likely.

Next we reformulate \eqref{eq:primitive} as a stochastic evolution equation on $\Ls^2(\Dom)\times L^2(\Dom)$ for the unknown $(v,\T)$. 
As usual in the context of primitive equations, we start by integrating the conditions \eqref{eq:primitive_4}-\eqref{eq:primitive_6}, and we obtain, a.e.\ on $\R_+\times \O$ and for all $(x_{\h},\z)\in \Tor^2\times (-h,0)=\Dom$,
\begin{align}
\label{eq:integration_vertical_direction_3_w}
w(\cdot,x)&=-\int_{-h}^{\z}\div_{\h} v(\cdot,x_{\h},\zeta)\,\dd \zeta,\\
	\label{eq:integration_vertical_direction_3_P}
	P(\cdot,x)&= p(\cdot,x_{\h}) -\int_{-h}^{\z} 
	\Big(\kone(\cdot,x_{\h},\zeta)\T(\cdot,x_{\h},\zeta) + 
	(\tp(\cdot,x_{\h},\zeta)\cdot \nabla) \T(\cdot,x_{\h},\z) \Big)\, \dd \zeta,\\
	\label{eq:integration_vertical_direction_3}
	\wt{P}_n(\cdot,x)&= \wt{p}_n(\cdot,x_{\h}) -\int_{-h}^{\z} \ktwon(\cdot,x_{\h},\zeta)\T(\cdot,x_{\h},\zeta)\, \dd \zeta.
\end{align}
To obtain \eqref{eq:integration_vertical_direction_3_w} we also used $w(\cdot,-h)=0$ by \eqref{eq:boundary_conditions_full_3}. Note that  $w(\cdot,0)=0$ is equivalent to 
\begin{equation}
\label{eq:integral_imcompressibility}
\int_{-h}^0 \div_{\h} v(\cdot,\zeta)\,\dd \zeta=0 \quad \text{ on }\Tor^2.
\end{equation}
Moreover, let us stress that the pressures $p$ and $\wt{p}_n$ are independent of the vertical direction $\z\in(-h,0)$. For this reason, in the physical literature, they are often referred to as \emph{surface pressures}. 
%%%

Hence, the system \eqref{eq:primitive_1}-\eqref{eq:primitive_2} can be equivalently rewritten as:
\begin{subequations}
\label{eq:primitive_v_T_pressure}
\begin{alignat}{2}
\label{eq:primitive_v_T_pressure_1}
		\begin{split}
		&\dd  v -\Delta v\, \dd t=\Big[  -(v\cdot \nabla_{\h})v- w(v)\partial_{3} v -\nabla_{\h} p
		+\partial_{\hp} \wt{p}_n\\
		& \qquad\qquad  +\nabla_{\h}\int_{-h}^{\cdot}  \big[\kone(\cdot,\zeta)\T(\cdot,\zeta)+ (\tp(\cdot,\zeta)
		\cdot\nabla) \T(\cdot,\zeta)\big]\,\dd \zeta+ \fv(v,\T,\nabla v,			\nabla\T) \Big]\, \dd t \\
		&\qquad\qquad +\sum_{n\geq 1}\Big[(\phi_{n}\cdot\nabla) v-\nabla_{\h} \wt{p}_n
		+\nabla_{\h}\int_{-h}^{\cdot} \big(\ktwon(\cdot,\zeta)\T(\cdot,\zeta)\big)\,\dd \zeta  +\gvn(v,\T)\Big] 						\, \dd \beta_t^n,
		\end{split}\\
\label{eq:primitive_v_T_pressure_2}
		\begin{split}
		&\dd  \T -\Delta \T\, \dd t=\Big[ -(v\cdot \nabla_{\h})\T- w(v)\partial_{3} v+ \ft(v,\T,\nabla v,\nabla \T ) 						\Big]\,\dd t \\
		&\qquad\qquad  +\sum_{n\geq 1}\Big[(\psi_{n}\cdot\nabla) \T+\gtn(v,\Temp)\Big] \, \dd \beta_t^n, 
		\end{split}
\end{alignat}
\end{subequations}
on $\Dom$, where
\begin{equation}
\label{eq:def_w}
w(v)\stackrel{{\rm def}}{=}-\int_{-h}^{\cdot}\div_{\h} v(\cdot,\zeta)\,\dd \zeta.
\end{equation} 
Next applying the hydrostatic Helmholtz projection $\p$ on \eqref{eq:primitive_v_T_pressure_1}, we obtain
\begin{align}
		\nonumber
&\dd  v -\Delta v\, \dd t=\Big(\p\Big[  -(v\cdot \nabla_{\h})v- w(v)\partial_{3} v
		+\partial_{\hp} \wt{p}_n\Big]\\
		\label{eq:primitive_v_T_pressure_1_projected}
		&\quad
		+\p\Big[\nabla_{\h}\int_{-h}^{\cdot}  \big[\kone(\cdot,\zeta)\T(\cdot,\zeta)+(\tp(\cdot,\zeta)
		\cdot\nabla) \T(\cdot,\zeta)\big]\,\dd \zeta+ \fv(v,\T,\nabla v,			\nabla\T) \Big]\Big)\, \dd t \\
		&\quad
		\nonumber
		+\sum_{n\geq 1}\p\Big[(\phi_{n}\cdot\nabla) v +\nabla_{\h}\int_{-h}^{\cdot} \big(\ktwon(\cdot,\zeta)\T(\cdot,		\zeta)\big)\,\dd \zeta  +\gvn(v,\T)\Big]\, \dd \beta_t^n.
\end{align}
In \eqref{eq:primitive_v_T_pressure_1_projected} we used that  $\p\Delta v=\Delta v$  by \eqref{eq:boundary_conditions_full_1}   and \eqref{eq:integral_imcompressibility}.
Note that (in general) in the stochastic part of the above, the operator $\p$ cannot be  removed.
%%%
In particular, we have
\begin{align*}
\nabla_{\h} \wt{p}_n=\underbrace{\q \Big[ (\phi_{n}\cdot\nabla) v +\nabla_{\h}\int_{-h}^{\cdot} \big(\ktwon(\cdot,\zeta)\T(\cdot,		\zeta)\big)\,\dd \zeta  +\gvn(v,\T)\Big]}_{\mathcal{Q}(v,\T)\stackrel{{\rm def}}{=}}.
\end{align*}
A similar relation holds for $\nabla_{\h} p$. Using the above identity and \eqref{eq:integration_vertical_direction_3}, we get
\begin{equation}
\label{eq:Lp_def}
\partial_{\hp} \wt{p}=\underbrace{\Big(\sum_{n\geq 1}\sum_{1\leq j\leq 2}\hp^{j,k}_n (\mathcal{Q}(v,\T))^j \Big)_{k=1}^2}_{\Lp(v,\T)\stackrel{{\rm def}}{=}},
\end{equation}
where $(\mathcal{Q}(v,\T))^j$ denotes the $j$-th coordinate of the vector $\mathcal{Q}(v,\T)$.
%%%%
Therefore, we have proved that  
 the system \eqref{eq:primitive_v_T_pressure} is equivalent to following system of SPDEs on $\Dom$:
% i.e.\
\begin{subequations}
\label{eq:primitive_strong}
\begin{alignat}{6}
\label{eq:primitive_strong_1}
		\begin{split}
		&\dd  v -\Delta v\, \dd t=\Big(\p\Big[  -(v\cdot \nabla_{\h})v- w(v)\partial_{3} v -\mathcal{P}_{\g}(v,\T)\\
		& \qquad \qquad  \ \ +\nabla_{\h}\int_{-h}^{\cdot}  \big[\kone(\cdot,\zeta)\T(\cdot,\zeta)
		+(\tp(\cdot,\zeta)
		\cdot\nabla) \T(\cdot,\zeta)\big]\,\dd \zeta+ \fv(v,\T,\nabla v,			\nabla\T) \Big]\Big)\, \dd t \\
		&\qquad \qquad \ \  +\sum_{n\geq 1}\p\Big[(\phi_{n}\cdot\nabla) v
		+\nabla_{\h}\int_{-h}^{\cdot} \big(\ktwon(\cdot,\zeta)\T(\cdot,\zeta)\big)\,\dd \zeta  +\gvn(v,\T)\Big] 	\,			\dd \beta_t^n,
		\end{split}\\
		\begin{split}
		&\dd  \T -\Delta \T\, \dd t=\Big[ -(v\cdot \nabla_{\h})\T- w(v)\partial_{3} v+ \ft(v,\T,\nabla v,\nabla \T ) 				\Big]\, \dd t \\
		&\qquad \qquad\ \   +\sum_{n\geq 1}\Big[(\psi_{n}\cdot\nabla) \T+\gtn(v,\Temp)\Big] \, \dd \beta_t^n.
		\end{split}
\end{alignat}
\end{subequations}
The above problem is complemented with the following boundary conditions on $\Tor^2$:
\begin{subequations}
\label{eq:primitive_strong_BC}
\begin{alignat}{3}
	\label{eq: primitive_strong_BC_1}
		\partial_{3} v (\cdot,-h)=\partial_{3} v(\cdot,0)&=0,\\
	\label{eq: primitive_strong_BC_2}
		\partial_{3} \T(\cdot,-h)= \partial_{3} \T(\cdot,0)+\alpha \T(\cdot,0)&=0,&
\end{alignat}
\end{subequations}
where $\alpha\in \R$ is a given constant. Note that \eqref{eq:primitive_strong_1} yields \eqref{eq:integral_imcompressibility} in case $\int_{-h}^0\div_{\h} v_0(\cdot,\zeta)\,\dd \zeta=0$ where $v_0$ is the initial condition of $v$, see \eqref{eq:primitive_7}.

\subsection{Main assumptions and definitions}
\label{ss:strong_strong_statement}
We begin by listing the main assumptions of this section. Below we employ the notation introduced in Subsection \ref{ss:set_up}.
 
\begin{assumption} There exist $M,\delta>0$ for which the following hold.
\label{ass:well_posedness_primitive_double_strong}
\begin{enumerate}[{\rm(1)}]
\item\label{it:well_posedness_measurability_strong_strong}
%{\em (Measurability)} 
For all $j\in \{1,2,3\}$ and $n \geq 1$, the mappings
$$
\phi_n^j,\ \psi_n^j ,  \ \kone,\ \tp^j,   \ \ktwon: \R_+\times \O\times \Dom\to \R \text{ are $\Progress\otimes \mathscr{B}$-measurable}.
$$
%%%
\item\label{it:well_posedness_primitive_parabolicity_strong_strong}
{\em (Parabolicity)}
There exists $\ellip\in (0,2)$ such that, a.s.\ for all $t\in \R_+$, $x\in \Dom$, $\xi\in \R^d$,
\begin{align*}
\sum_{n\geq 1} \Big(\sum_{1\leq j\leq 3}\phi^j_n(t,x) \xi_j\Big)^2\leq \ellip|\xi|^2,
\ \ \text{ and }\ \ 
\sum_{n\geq 1} \Big(\sum_{1\leq j\leq 3} \psi^j_n(t,x) \xi_j\Big)^2 \leq \ellip |\xi|^2.
\end{align*}
%%%%%
\item\label{it:well_posedness_primitive_phi_psi_smoothness} 
{\em (Regularity I)}
a.s.\ for all $t\in \R_+$, $j,k\in \{1,2,3\}$ and $\ell,m\in \{1,2\}$,
\begin{align*}
\Big\|\Big(\sum_{n\geq 1}| \phi^j_n(t,\cdot)|^2\Big)^{1/2} \Big\|_{L^{3+\delta}(\Dom)}+
\Big\|\Big(\sum_{n\geq 1}|\partial_k \phi^j_n(t,\cdot)|^2\Big)^{1/2} \Big\|_{L^{3+\delta}(\Dom)} \leq M,&\\
\Big\|\Big(\sum_{n\geq 1}| \psi^j_n(t,\cdot)|^2\Big)^{1/2} \Big\|_{L^{3+ \delta}(\Dom)}+
\Big\|\Big(\sum_{n\geq 1}|\partial_k \psi^j_n(t,\cdot)|^2\Big)^{1/2} \Big\|_{L^{3+ \delta}(\Dom)} \leq M,&\\
\|(\hp^{\ell,m}_n(t,\cdot))_{n\geq 1}\|_{L^{3+\delta}(\Dom;\ell^2)}\leq M.&
\end{align*}
%%%%
%%%%
\item\label{it:well_posedness_primitive_kone_smoothness_strong_strong} 
{\em (Regularity I)}
a.s.\ for all $t\in \R_+$, $i\in \{1,2\}$ and $j\in \{1,2,3\}$,
\begin{align*}
\| \kone(t,\cdot) \|_{L^{\infty}(\Tor^2;L^2(-h,0))} +\|\partial_i \kone(t,\cdot) \|_{L^{2+\delta}(\Tor^2;L^2(-h,0))}& \leq M,\\
\| \tp^j(t,\cdot) \|_{L^{\infty}(\Tor^2;L^2(-h,0))} +\|\partial_i \tp^j(t,\cdot) \|_{L^{2+\delta}(\Tor^2;L^2(-h,0))} &\leq M.
\end{align*}

\item\label{it:well_posedness_primitive_ktwo_smoothness_strong_strong} Set $\ktwo\stackrel{{\rm def}}{=} (\ktwon)_{n\geq 1}$. Then,
a.s.\ for all $t\in \R_+$ and $i,j\in \{1,2\}$,
\begin{align*}
\|\ktwo(t,\cdot)\|_{L^{\infty}(\Dom;\ell^2)}
&+ 
\|\partial_{i} \ktwo(t,\cdot)\|_{L^{2+\delta}(\Tor^2;L^{2}(-h,0;\ell^2))}\\
&+\|\partial_{i,j}^2 \ktwo(t,\cdot)\|_{L^{2+\delta}(\Tor^2;L^{2}(-h,0;\ell^2))}\leq M.
\end{align*}
%%%%
\item\label{it:nonlinearities_measurability_strong_strong}
For all $n\geq 1$, the following mappings  are $\Progress\otimes \Borel$-measurable:
\begin{align*}
\fv&: \R_+\times \O\times \Dom \times \R^6\times \R^3\times \R^2 \times \R \to \R^2,& \ft&:\R_+\times \O\times \Dom\times   \R^6\times \R^3\times \R^2 \times \R \to \R,\\
\gvn&:\R_+\times \O\times \Dom \times \R^2  \times \R \to \R^2, &\gtn&:\R_+\times \O\times \Dom \times \R^2 \times \R \to \R.&
\end{align*}
Set 
$$
\gv\stackrel{{\rm def}}{=}(\gvn)_{n\geq 1} \quad \text{ and } \quad \gt\stackrel{{\rm def}}{=}(\gtn)_{n\geq 1}.
$$
\item\label{it:nonlinearities_strong_strong}
 {\em (Global Lipschitz nonlinearities)} a.s.\ 
 \begin{align*}
\fv(\cdot,0)&\in L^2_{\loc}(\R_+\times \Dom;\R^2),& \ft(\cdot,0)&\in L^2_{\loc}(\R_+\times \Dom)\\
\gv(\cdot,0)&\in L^2_{\loc}(\R_+;H^1(\Dom;\ell^2(\N,\R^2))), &\gt(\cdot,0)&\in L^2_{\loc}(\R_+;H^1(\Dom;\ell^2)).&
\end{align*}
Moreover, there exists $K\geq 1$ such that, for all $u\in \{v,\T\}$, a.e.\ on $\R_+\times \O\times \Dom$ and for all $y,y'\in \R^2$, $Y,Y'\in \R^6$, $z,z'\in \R$ and $Z,Z'\in \R^3$,
\begin{align*}
|F_u(\cdot,y,z,Y,Z)-F_u(\cdot,y',z',Y',Z')|&\leq K(|y-y'|+|z-z'|\\
&\ \ \ + |Y-Y'|+|Z-Z'|),\\
\|\gx(\cdot,y,z)-\gx(\cdot,y',z')\|_{\ell^2}
+ \|\nabla_x \gx(\cdot,y,z)-\nabla_x \gx(\cdot,y',z')\|_{\ell^2}&\leq K(|y-y'|+|z-z'|),\\
\|\nabla_y \gx(\cdot,y,z)-\nabla_y \gx(\cdot,y',z')\|_{\ell^2}
+ \|\nabla_z \gx(\cdot,y,z)-\nabla_z  \gx(\cdot,y',z')\|_{\ell^2}&\leq K.
\end{align*}

\end{enumerate}
\end{assumption}

\begin{remark}
\label{r:boundedness}
Below we collect some observations on Assumption \ref{ass:well_posedness_primitive_double_strong}. 
\begin{itemize}
\item The Sobolev embedding $H^{1,3+\delta}(\Dom;\ell^2)\embed C^{\eta}(\Dom;\ell^2)$ for $\eta =\frac{\delta}{3+\delta}\in (0,1)$ and \eqref{it:well_posedness_primitive_phi_psi_smoothness} yield
$$
\| (\phi^j_n(t,\cdot))_{n\geq 1}\|_{C^{\eta}(\Dom;\ell^2)}+\|(\psi^j_n(t,\cdot))_{n\geq 1} \|_{ C^{\eta}(\Dom;\ell^2)} \lesssim M\ \   \text{ a.s.\ \ for all }t\in \R_+.
$$
\item 
As in \cite[Remark 3.2(c)]{Primitive1}, \eqref{it:well_posedness_primitive_parabolicity_strong_strong} 
is equivalent with the usual \emph{stochastic parabolicity} and therefore \eqref{it:well_posedness_primitive_parabolicity_strong_strong} is optimal in the parabolic setting. 
\item The global Lipschitz assumption \eqref{it:nonlinearities_strong_strong} can be weakened still keeping true the results of this manuscript. The reader is referred to Remark \ref{r:weaken_assumption_nonlinearities} for more details.
\end{itemize}
\end{remark}

Next we define $L^2$-strong solutions to \eqref{eq:primitive}-\eqref{eq:boundary_conditions_full}. 
Motivated by the reformulation of \eqref{eq:primitive} performed in Subsection \ref{ss:reformulation},
we consider the equivalent system \eqref{eq:primitive_strong} for the unknown $(v,\T)$ while the unknown $(P,\wt{P}_n,w)$ are determined uniquely by $(v,\T)$. Taking into account the boundary conditions \eqref{eq:primitive_strong_BC} and the divergence-free condition \eqref{eq:integral_imcompressibility} for the velocity $v$, we introduce the following spaces:
\begin{align}
\label{eq:def_Hn}
\Hs^{2}_{\n}(\Dom)
&\stackrel{{\rm def}}{=}\Big\{v \in H^{2}(\Dom;\R^2)\cap \Ls^2(\Dom)\,:\, \partial_{3} v(\cdot,-h)=\partial_{3} v(\cdot, 0)=0\text{ on }\Tor^2\Big\},\\
\label{eq:def_Hr}
\Hr^{2}(\Dom)
&\stackrel{{\rm def}}{=}\Big\{ \T \in H^{2}( \Dom) \,:\, \partial_{3} \T(\cdot,0)+\alpha \T(\cdot,0)= \partial_{3} \T (\cdot,-h)=0\text{ on }\Tor^2
\Big\}.
\end{align}
Note that the boundary conditions \eqref{eq:primitive_strong_BC} are included in the above spaces. Hence, the spaces $\Hs^2_{\n}$ and $\Hr^2$ serve as state spaces for the unknowns $v$ and $\T$, respectively.

Finally, we denote by $\Br_{\ell^2}$ the $\ell^2$-cylindrical Brownian motion induced by $(\beta^n)_{n\geq 1}$, i.e.\ 
\begin{equation}
\label{eq:def_Br}
\Br_{\ell^2}(f)\stackrel{{\rm def}}{=}
\sum_{n\geq 1}\int_{\R_+} f_n(t)\,\dd \beta_t^n \ \ \text{ where } \ \ f=(f_n)_{n\geq 1}\in L^2(\R_+;\ell^2).
\end{equation}

\begin{definition}[$L^2$-local, maximal and global strong solutions]
\label{def:sol_strong_strong}
Let Assumption \ref{ass:well_posedness_primitive_double_strong} be satisfied and let $\tau$ be a stopping time with values in $[0,\infty]$. Consider two stochastic processes 
$$
v:[0,\tau)\times \O \to \Hs_{\n}^2(\Dom)\qquad \text{ and } \qquad \T:[0,\tau)\times \O \to \Hr^2(\Dom).
$$
\begin{itemize}
\item\label{it:def_local} We say that $((v,\T),\tau)$ is called an \emph{$L^2$-local strong solution} to \eqref{eq:primitive}-\eqref{eq:boundary_conditions_full} if there exists a sequence of stopping times $(\tau_k)_{k\geq 1}$ for which the following hold:
\begin{itemize}
\item $\tau_k\leq \tau$ a.s.\ for all $k\geq 1$ and $\lim_{k\to \infty}\tau_k=\tau$ a.s.
%%%%
\item For all $k\geq 1$, the process $\one_{[0,\tau_k]} (v,\T)$ is progressively measurable.  
\item a.s.\ we have $(v,\T)\in L^2(0,\tau_n;\Hs_{\n}^2(\Dom)\times \Hr^2(\Dom))$ and 
\begin{equation}
\begin{aligned}
\label{eq:integrability_strong_strong}
(v\cdot \nabla_{\h}) v+ \w(v)\partial_{3} v +\fv (v,\T,\nabla v,\nabla \T)+\Lp (\cdot,v,\T)&\in L^2(0,\tau_k;L^2(\Dom;\R^2)),\\
(v\cdot \nabla_{\h}) \T+ \w(v)\partial_{3} \T +\ft (v,\T,\nabla v,\nabla \T)&\in L^2(0,\tau_k;L^2(\Dom)),\\
(\gvn(v,\T))_{n\geq 1 }&\in L^2(0,\tau_k;H^1(\Dom;\ell^2(\N;\R^2))),\\
(\gtn(v,\T))_{n\geq 1 }&\in L^2(0,\tau_k;H^1(\Dom;\ell^2)).
\end{aligned}
\end{equation}
%%%%%%%%%
\item The following equality holds a.s.\ for all $k\geq 1$ and $t\in [0,\tau_k]$:
\begin{align*} 
v(t)-v_0
&=\int_0^t \Big(\Delta v(s)+ \p\Big[\int_{-h}^{\cdot} \nabla_{\h}\big[\kone(\cdot,\zeta)\T(\cdot,\zeta)+(\tp(\cdot,\zeta)\cdot\nabla) \T(\cdot,\zeta)\big]\,\dd \zeta\\
&\qquad \qquad
- (v\cdot \nabla_{\h}) v- \w(v)\partial_{3} v  + \fv (v,\T,\nabla v,\nabla \T) +\Lp(\cdot,v,\T)\Big]\Big)\,\dd s\\
&
+\int_0^t \Big(\one_{[0,\tau_k]}\p \Big[ (\phi_{n}\cdot\nabla) v +\int_{-h}^{\cdot}\nabla_{\h}(\ktwon(\cdot,\zeta)\T(\cdot,\zeta))\,\dd \zeta  +\gvn(v,\T) \Big] \Big)_{n\geq 1}\, \dd \Br_{\ell^2}(s),\\
\T(t)-\T_0
&=
\int_0^t  \Big[\Delta \T-(v\cdot \nabla_{\h})\T -w(v)\partial_{3} \T+ \ft(v,\T,\nabla v,\nabla \T )\Big]\,\dd s\\
&+
\int_0^t \Big(\one_{[0,\tau_k]}[ (\psi_{n}\cdot\nabla) \T   +\gtn(v,\T )] \Big)_{n\geq 1}\, \dd \Br_{\ell^2}(s).
\end{align*}
\end{itemize}

In the following, we say that $(\tau_k)_{k\geq 1}$ is a \emph{localizing sequence} for $(v,\tau)$.
%%%%%%
\item An $L^2$-local strong solution $((v,\T),\tau)$ to \eqref{eq:primitive}-\eqref{eq:boundary_conditions_full} is said to be a (unique) \emph{$L^2$-maximal strong solution to} \eqref{eq:primitive}-\eqref{eq:boundary_conditions_full}
if for any other local solution $((v',\T'),\tau')$ we have 
\begin{equation*}
\tau'\leq \tau \  \text{ a.s.\ \  and } \ \ (v,\T)=(v',\T') \ \text{ a.e.\ on }[0,\tau')\times \O.
\end{equation*}
%%%%%
\item An $L^2$-maximal strong solution $((v,\T),\tau)$ to \eqref{eq:primitive}-\eqref{eq:boundary_conditions_full} is called an \emph{$L^2$-global strong solution} if $\tau=\infty$ a.s. In such a case, we write $(v,\T)$ instead of $((v,\T),\tau)$.
\end{itemize}
\end{definition}

Note that $L^2$-maximal strong solutions are \emph{unique} in the class of $L^2$-local strong solutions due to the above definition. By \eqref{eq:integrability_strong_strong}, the deterministic integrals and the 
stochastic integrals in the above definition are well-defined as %an
 $L^2$-valued Bochner and $H^1$-valued It\^{o} integrals, respectively.

\subsection{Main results}
\label{ss:statement_main_results}
To economize the notation, through this manuscript we let 
\begin{equation}
\label{eq:def_HV}
H\stackrel{{\rm def}}{=}\Hs^{1}(\Dom)\times H^1(\Dom) \quad \text{ and }\quad V\stackrel{{\rm def}}{=}\Hs^{2}_{\n}(\Dom)\times \Hr^2(\Dom).
\end{equation}
Below $H$ and $V$ play the role of the trace and regularity space for \eqref{eq:primitive}-\eqref{eq:boundary_conditions_full}, respectively.

We begin by stating a local existence result for \eqref{eq:primitive}-\eqref{eq:boundary_conditions_full}. 

\begin{theorem}[Local existence and blow-up criterion]
\label{t:local_primitive_strong_strong}
Let Assumption \ref{ass:well_posedness_primitive_double_strong} be satisfied. Let 
$
(v_0,\T_0)\in L^0_{\F_0}(\O;H)
$.
Then \eqref{eq:primitive}-\eqref{eq:boundary_conditions_full} has a (unique) $L^2$-maximal strong solution $((v,\theta) ,\tau)$ such that  
$$\tau>0 \text{ a.s.} \quad \text{ and }\quad
(v,\T)\in L^2_{\loc}([0,\tau);V)\cap C([0,\tau);H) \text{ a.s.}
$$
Finally, for all $T\in (0,\infty)$,
\begin{align}
\label{eq:blow_up_criterium}
%\P\big(\tau<T,\,  \|(v,\T)\|_{L^{\infty}(0,\tau;H)}+\|(v,\T)\|_{L^2(0,\tau;V)} <\infty\big)&=0 \\ 
\P\Big(\tau<T,\,  \sup_{t\in [0,\tau)}\big\|(v(t),\T(t))\big\|^2_{H}+\int_0^{\tau}\big\|(v(t),\T(t))\big\|_{V}^2\,\dd t <\infty\Big)=0 . 
\end{align}
\end{theorem}

The proof of Theorem \ref{t:local_primitive_strong_strong} follows as in \cite{Primitive1} where we checked the applicability of the abstract results of \cite{AV19_QSEE_1,AV19_QSEE_2}. A sketch of the proof of Theorem \ref{t:local_primitive_strong_strong} will be given in Subsection \ref{ss:proof_local}.
The statement \eqref{eq:blow_up_criterium} will be referred as \emph{blow-up criterion} as it shows that explosion $\tau=\infty$ can only happen if either $(v,\T)\not\in C([0,\tau];H)$ or $(v,\T)\not\in L^2(0,\tau;V)$ for some $\tau<\infty$. Let us note that, since $(v,\T)\in V$ a.e.\ on $[0,\tau)\times \O$, the blow-up criterion \eqref{eq:blow_up_criterium} is equivalent to 
\begin{align*}
\P\Big(\tau<T,\, & \sup_{t\in [0,\tau)} \big[\|v(t)\|^2_{H^1}+\big\|\T(t)\|^2_{H^1} \big] 
+  \int_0^{\tau} \big[\|v(t)\|^2_{H^2}+\|\T(t)\|^2_{H^2} \big]\,\dd t<\infty \Big)=0.
\end{align*}

Let us turn now our attention to the existence of global solutions to \eqref{eq:primitive}-\eqref{eq:boundary_conditions_full}. In contrast to the local existence result of Theorem \ref{t:local_primitive_strong_strong}, the global existence is much more involved. 
In particular, in addition to Assumption \ref{ass:well_posedness_primitive_double_strong} we will also need the following assumptions. 

\begin{assumption}
\label{ass:global_primitive_strong_strong} a.s.\ and for all $n\geq 1$, $x=(x_{\h},\z)\in \Tor^2\times (-h,0)= \Dom$, 
$t\in \R_+$, $j,k\in \{1,2\}$  
\begin{align*}
\phi_n^j(t,x),\ \psi^j_n (t,x), \ \hp^{j,k}_n(t,x), \ \tp^j(t,x)  \text{ and } \ \ktwon(t,x) \ \text{ are independent of $\z$}.
\end{align*}
\end{assumption}

We do not know if any of the above hypotheses can be removed in general.
Note that there are no additional assumptions on $(\phi_n^{3},\psi^3_n)$.  
However, in the case of isothermal turbulent pressure, the conditions on $\psi^{j}$ in Assumption \ref{ass:global_primitive_strong_strong} can be removed (see  \cite[Sections 3 and 6]{Primitive1} and Remark \ref{r:comparison_primitive1} below).
The physical relevance of the $x_3$-independence of $(\phi^j_n,\psi_n^j,\ktwon)$ is discussed in Remarks \ref{r:sigma_independence}, \ref{r:phi_independence} and \ref{r:twod_turbulence}. While, for $(\hp^{j,k}_n, \tp^j)$, in the physically relevant case where they are related to the Stratonovich formulation of the primitive equations (see Section \ref{s:Stratonovich}), the $x_3$-independence comes from the one of $(\phi^j_n,\psi_n^j,\ktwon)$ (cf.\ formula \eqref{eq:pi_psi_etc} in Section \ref{s:Stratonovich}).

We are ready to state the main results of this paper. For notational convenience we set 
\begin{equation}
\label{eq:def_y}
\y(t)\stackrel{{\rm def}}{=}\|\fv(t,\cdot,0)\|_{L^2}+\|\ft(t,\cdot,0)\|_{L^2}
+\|\gv(t,\cdot,0)\|_{H^1(\ell^2)}+\|\gt(t,\cdot,0)\|_{H^1(\ell^2)}.
\end{equation}
Note that $\y\in L^2(0,T)$ a.s.\ for all $T<\infty$ by Assumption \ref{ass:well_posedness_primitive_double_strong}\eqref{it:nonlinearities_strong_strong}. 

\begin{theorem}[Global existence and energy estimates]
\label{t:global_primitive_strong_strong}
Let Assumptions \ref{ass:well_posedness_primitive_double_strong} and \ref{ass:global_primitive_strong_strong} be satisfied. Let 
$
(v_0,\T_0)\in L^0_{\F_0}(\O;H)
$.
Then there exists a \emph{(unique) $L^2$-global strong solution} $((v,\T),\tau)$ to \eqref{eq:primitive}-\eqref{eq:boundary_conditions_full} such that 
$$
(v,\T)\in L^2_{{\rm loc}}([0,\infty);V)\cap C([0,\infty);H) \text{ a.s. }
$$ 
Finally, for all $T\in (0,\infty)$ there exists $C_T>0$, independent of $(v_0,\T_0)$, such that,
 for all $\g>e^e$,
\begin{align*}
\E \sup_{t\in [0,T]}\|v(t)\|^2_{L^2} +\E\int_0^T \|v(t)\|^2_{H^1}\,\dd t&\leq C_T (1+\E\|v_0\|_{L^2}^2+\E\|\T_0\|_{L^2}^2+\E \|\y\|_{L^2(0,T)}^2),\\
\E \sup_{t\in [0,T]}\|\T(t)\|^2_{L^2}+ \E\int_0^T \|\T(t)\|^2_{H^1}\,\dd t&\leq 
C_T (1+\E\|v_0\|_{L^2}^2+\E\|\T_0\|_{L^2}^2+\E \|\y\|_{L^2(0,T)}^2),\\
\P\Big(\sup_{t\in [0,T]}\|v(t)\|^2_{H^1} +\int_0^T \|v(t)\|^2_{H^2}\,\dd t\geq \g\Big)&\leq
 C_T\frac{(1 +\E\|v_0\|_{H^1}^4+\E\|\T_0\|_{H^1}^4+\E \|\y\|_{L^2(0,T)}^2)}{\log\log\log(\g)} ,\\
\P\Big(\sup_{t\in [0,T]}\|\T(t)\|^2_{H^1} +\int_0^T \|\T(t)\|^2_{H^2}\,\dd t\geq \g\Big)&\leq 
C_T\frac{(1 +\E\|v_0\|_{H^1}^4+\E\|\T_0\|_{H^1}^4+\E \|\y\|_{L^2(0,T)}^2)}{\log\log\log(\g)} .
\end{align*}
\end{theorem}

The tail estimates for the r.v.\ $\sup_{t}\|v\|^2_{H^1_x} +\|v\|^2_{L^2_t (H^2_x)}$ and  $\sup_{t}\|\T\|^2_{H^1_x} 
+\|\T\|^2_{L^2_t (H^2_x)}$
 are rather weak. However, in general, it does not seem possible to improve the estimates as they come from three applications of the Gronwall lemma. Each of them costs a $\log$ factor. The same also appears in the deterministic case where 
 one obtains estimates with exponentially increasing constants in the size of the data (see e.g.\ \cite{CT07}).
The estimates of Theorem \ref{t:global_primitive_strong_strong} can be (slightly) improved in case of \emph{isothermal} turbulent pressure, see Remark \ref{r:comparison_primitive1} below.  

Theorem \ref{t:global_primitive_strong_strong} and the following show that the problem \eqref{eq:primitive}-\eqref{eq:boundary_conditions_full} is \emph{globally well-posed}.
Recall that $\xi_n \to \xi$ in probability in $Y$ 
%provided
means that $\lim_{n\to \infty}\P(\|\xi_n-\xi\|_{Y}>\varepsilon)=0$ for all $\varepsilon>0$.
 
\begin{theorem}[Continuous dependence on the initial data]
\label{t:continuous_dependence}
Let Assumptions \ref{ass:well_posedness_primitive_double_strong} and \ref{ass:global_primitive_strong_strong} be satisfied. Suppose that $((v_{0,n},\T_{0,n}))_{n\geq 1} \subseteq L^0_{\F_0}(\O;H)$ is a sequence of initial data converging in probability in $H$ to some $(v_0,\T_0)$. Let $(v_n,\T_n)$ and $(v,\T)$ be the $L^2$-global strong solutions to \eqref{eq:primitive}-\eqref{eq:boundary_conditions_full} with initial data $(v_{0,n},\T_{0,n})$ and $(v_0,\T_0)$, respectively.
Then, for all $T\in (0,\infty)$,
$$
(v_n,\T_n)\to (v,\T)\  \text{ as $n\to \infty$ in probability in } C([0,T];H)\cap L^2(0,T;V) .
$$
\end{theorem}

The proof of Theorems \ref{t:global_primitive_strong_strong} and \ref{t:continuous_dependence} will be given in Subsections \ref{ss:proof_global} and \ref{ss:proof_continuity}, respectively. Both results essentially depend on the energy estimate of Proposition \ref{prop:energy_estimate_primitive_strong_strong}. The proof of the latter will be the major scope of our work, and Sections \ref{s:basic_estimate},  \ref{s:intermediate_estimate} and \ref{s:proof_energy_estimate_conclusion} are devoted to its proof. Finally, in Section \ref{s:Stratonovich} we 
discuss the case of Stratonovich noise.

We conclude this section with several remarks related to Theorems \ref{t:global_primitive_strong_strong} and \ref{t:continuous_dependence}. 

\begin{remark}[Feller property]
\label{r:Feller}
Let $(v_{\eta},\T_{\xi})$ be the global strong solution to \eqref{eq:primitive}-\eqref{eq:boundary_conditions_full} provided by Theorem \ref{t:global_primitive_strong_strong} with initial data $(\eta,\xi)\in H$. 
For all $t\geq 0$, set
$$
[{\mathcal{S}}_t  \varphi](\eta,\xi) \stackrel{{\rm def}}{=}\E[\varphi(v_{\eta}(t),\T_{\xi}(t)) ] \ \ \text{ for all }(\eta,\xi)\in H \ \text{ and }\ \varphi\in C(H;\R).
$$
Theorem \ref{t:continuous_dependence} in particular implies that ${\mathcal{S}}_t$ maps continuously $C(H;\R)$ into itself. This is often referred to as \emph{Feller property}. In particular, our results extend \cite[Theorem 1.5]{GHKVZ14}. In the spirit of \cite{GHKVZ14}, it would be interesting to study the existence and/or uniqueness of invariant measures. However, this goes beyond the scope of this paper.
\end{remark}

\begin{remark}[$\O$-localization of energy estimates]
The energy estimates in Theorem \ref{t:global_primitive_strong_strong} also imply tail probability estimates for non-integrable data by using localization arguments. To see this let $(v_0,\T_0)\in L^0_{\F_0}(\O;H)$. Fix $\delta>0$ and set $(v_{0}^{(\delta)},\T_{0}^{(\delta)})\stackrel{{\rm def}}{=} \one_{\{\|(v_0,\T_0)\|_H\leq \delta\}} (v_0,\T_0)$. Let $(v^{(\delta)},\T^{(\delta)})$ be the global strong solution to \eqref{eq:primitive}-\eqref{eq:boundary_conditions_full} with initial data $(v^{(\delta)}_0,\T_0^{(\delta)})$ provided by Theorem \ref{t:global_primitive_strong_strong}. Then by \cite[Theorem 4.7(4)]{AV19_QSEE_1} we have $(v^{(\delta)},\T^{(\delta)})=(v,\T)$ a.e.\ on $\R_+\times\{\|(v_0,\T_0)\|_H\leq \delta\}$. Hence, 
\begin{align*}
&\P\Big(\sup_{t\in [0,T]}\|v(t)\|^2_{H^1} +\int_0^T \|v(t)\|^2_{H^2}\,\dd t\geq \g\Big)\\
&\leq \P\Big(\sup_{t\in [0,T]}\|v^{(\delta)}(t)\|^2_{H^1} +\int_0^T \|v^{(\delta)}(t)\|^2_{H^2}\,\dd t\geq \g, \|(v_0,\T_0)\|_H\leq \delta\Big)
+ \P(\|(v_0,\T_0)\|_H> \delta)\\
&\leq C_T\frac{1+2\delta^4+\E \|\y\|_{L^2(\R_+;L^2)}^2 }{ \log\log\log (\g)} 
+ \P(\|(v_0,\T_0)\|_H> \delta) \qquad \text{ for all }  \g>e^e\text{ and }\delta>1,
\end{align*}
where in the last inequality we applied the third estimate of Theorem \ref{t:global_primitive_strong_strong}. For instance, we may choose $\delta=\log\log\log\log (\g)$, and the above estimate shows that the tail of the r.v.\  $\sup_{t}\|v\|^2_{H^1_x} +\|v(t)\|^2_{L^2_t (H^2_x)}$ converges to $0$ as $\g\to \infty$ with an explicit rate. A similar argument also holds for the other estimates in Theorem \ref{t:global_primitive_strong_strong}, where for the first two one also applies the Chebyshev inequality.

A similar argument also works if one only knows that $\y\in L^2(0,T)$ a.s.\ for all $T<\infty$. 
\end{remark}

\begin{remark}[Improved energy estimates in case of isothermal turbulent pressure]
\label{r:comparison_primitive1}
The tail estimates of Theorem \ref{t:global_primitive_strong_strong} 
are new even in the case of isothermal turbulent pressure $\ktwon\equiv 0$ and $\tp\equiv 0$, as considered in \cite{Primitive1}. 
However, following the proofs in \cite{Primitive1} and the one presented here, one sees that the tail estimates of  Theorem \ref{t:global_primitive_strong_strong} can be improved in the setting considered in \cite{Primitive1}. Indeed, as in \cite{Primitive1}, the tail estimate for $\sup_t\|\T\|_{L^4}^4+\||\T||\nabla \T|\|_{L^2_{t}(L^2_x)}^2$ of Lemma \ref{l:basic_estimates} are not needed as a starting point. Hence, following the arguments in \cite{Primitive1} and using the stochastic Grownall lemma of \cite[Lemma A.1]{AV_variational} as in the present paper, one sees that the $\log\log\log(\g)$ decay in Theorem \ref{t:global_primitive_strong_strong} can be improved to a $\log\log(\g)$-one (cf.\ \cite[Lemma 6.1]{A23_primitive}).
\end{remark}

\begin{remark}[Non homogeneous viscosity/conductivity]
\label{r:inhomogeneous_viscosity_conductivity}
Arguing as in \cite[Section 7]{Primitive1}, one can check that Theorems \ref{t:global_primitive_strong_strong} and \ref{t:continuous_dependence} extends to the case of inhomogeneous viscosity and/or conductivity. More precisely, we may replace the terms $\Delta v$ and $\Delta \T$ in \eqref{eq:primitive_1}-\eqref{eq:primitive_2} by 
\begin{equation}
\label{eq:differential_operators_inhomogeneous}
\p\Big[\sum_{1\leq i,j\leq 3} a^{i,j}_v \partial_{i,j}^2 v + \sum_{1\leq k\leq 3} b_v^k \partial_k v \Big] \quad \text{ and }\quad \sum_{1\leq i,j\leq 3} a^{i,j}_{\T} \partial_{i,j}^2 \T + \sum_{1\leq k\leq 3} b_{\T}^k \partial_k \T,
\end{equation} 
respectively.
The above situation arises in the case of noise in the Stratonovich formulation of \eqref{eq:primitive}-\eqref{eq:boundary_conditions_full}, see Section \ref{s:Stratonovich}. 
We may also consider $0$-th order terms in \eqref{eq:differential_operators_inhomogeneous}. 
However, as they are not needed in  Section \ref{s:Stratonovich}, we do not consider such terms here. 
We leave the details to the interested reader.

The local existence result of Theorem \ref{t:local_primitive_strong_strong} extends to such situation 
under suitable assumptions on $(a_v,b_v,a_{\T},b_{\T})$. More precisely, in addition to Assumption \ref{ass:well_posedness_primitive_double_strong}\eqref{it:well_posedness_measurability_strong_strong}, \eqref{it:well_posedness_primitive_phi_psi_smoothness}-\eqref{it:nonlinearities_strong_strong} and Assumption \ref{ass:global_primitive_strong_strong}, one assumes that:
\begin{itemize}
\item(Measurability) $a^{i,j}_v,b_v^k,a^{i,j}_{\T},b_{\T}^k:\R_+\times \O\times \Dom \to \R$ are $\Progress\otimes \Borel(\Dom)$-measurable.
\item (Parabolicity) There exists $\ellip>0$ such that, a.e.\ on $\R_+\times \O\times \Dom$ and all $\xi\in \R^3$, 
$$
\sum_{1\leq i,j\leq 3}\Big( a^{i,j}_{v}- \frac{1}{2}
\sum_{n\geq 1} \phi_n^i\phi_n^j \Big)\xi_i\xi_j\geq \ellip|\xi|^2
 \ \  \text{ and }\ \ 
\sum_{1\leq i,j\leq 3}\Big( a^{i,j}_{\T}- \frac{1}{2}
\sum_{n\geq 1} \psi_n^i\psi_n^j \Big)\xi_i\xi_j\geq \ellip|\xi|^2.
$$
\item (Regularity) There exist $M,\delta>0$ such that, a.e.\ on $\R_+\times \O$,
$$
\|a_v\|_{H^{1,3+\delta}(\Dom;\R^{d\times d})}
+
\|a_{\T}\|_{H^{1,3+\delta}(\Dom;\R^{d\times d})}
+ 
\|b_{v}\|_{L^{3+\delta}(\Dom;\R^{d})}
+
\|b_{\T}\|_{L^{3+\delta}(\Dom;\R^{d})}\leq M.
$$
\item (Boundary regularity) a.s.\ for all $t\in \R_+$, $x_{\h}\in\Tor^2$ and $R\in \{v,\T\}$,
$$
\| a_{R}^{3,j}(t,\cdot,0)\|_{H^{\frac{1}{2}+\delta}(\Tor^2)}=\| a_{R}^{3,j}(t,\cdot,-h)\|_{H^{\frac{1}{2}+\delta}(\Tor^2)}\leq M.
$$
\end{itemize} 
The reader is referred to \cite[Assumption 7.4 and Remark 7.6]{Primitive1} for a discussion on the above conditions.

Similarly, as in \cite[Section 7]{Primitive1},
the global well-posedness result of Theorems \ref{t:global_primitive_strong_strong} and \ref{t:continuous_dependence} also extend to the case of inhomogeneous viscosity and/or conductivity by also assuming that:
\begin{itemize}
\item For all $i,j\in\{1,2\}$, the maps $(a_v^{i,j},a_{\T}^{i,j},b_v^j,b_{\T}^j)$ are independent of $\z$. Moreover,
 a.s.\ for all $t\in \R_+$, $x_{\h}\in\Tor^2$ and $R\in \{v,\T\}$,
$$
a_{R}^{3,j}(t,x_{\h},0)= a_{R}^{3,j}(t,x_{\h},-h)=0.
$$
\end{itemize}
Note that the condition on $a_{R}^{3,j}|_{\Tor^2\times\{-h,0\}}$ is stronger than the one needed for local existence.
\end{remark}

\begin{remark}[Weakening the assumptions on the nonlinearities]
\label{r:weaken_assumption_nonlinearities}
Assumption \ref{ass:well_posedness_primitive_double_strong}\eqref{it:nonlinearities_measurability_strong_strong}-\eqref{it:nonlinearities_strong_strong} can be generalized still keeping true (a subset of) Theorems \ref{t:local_primitive_strong_strong} and \ref{t:global_primitive_strong_strong}-\ref{t:continuous_dependence}.
More precisely:
\begin{enumerate}[{\rm(a)}]
\item\label{it:local_lip} Theorem \ref{t:local_primitive_strong_strong} holds if Assumptions \ref{ass:well_posedness_primitive_double_strong}\eqref{it:nonlinearities_measurability_strong_strong}-\eqref{it:nonlinearities_strong_strong} are replaced by \cite[(HF)-(HG)]{AV19_QSEE_2} with $X_0=L^2(\Dom)\times \Ls^2(\Dom)$ and $X_1=\Hs_{\n}^{2}(\Dom)\times \Hr^2(\Dom)$.
In particular, instead of the global Lipschitz condition we may require the local Lipschitz condition \eqref{eq:estimate_nonlinearities_continuity_proof} below.
\item\label{it:global_sublinear} Theorems \ref{t:global_primitive_strong_strong}-\ref{t:continuous_dependence} still hold if Assumptions \ref{ass:well_posedness_primitive_double_strong}\eqref{it:nonlinearities_measurability_strong_strong}-\eqref{it:nonlinearities_strong_strong} are replaced by the conditions in \eqref{it:local_lip} and a (sub-linear) condition: There exists $\y\in L^0((0,T)\times \O)$ for all $T<\infty$ such that, for all $(v',\T')\in V$ and  a.e.\ on $\R_+\times \O$,
\begin{align*}
\|F_u(\cdot,v',\T',\nabla v',\nabla \T')\|_{L^2}
&+ \|G_u(\cdot,v',\T')\|_{H^1(\ell^2)}\\
&\ \ \lesssim\y+\|v'\|_{H^1}+\|\T'\|_{H^1} ,\ \  \text{ for } u\in \{v,\T\}.
\end{align*}
\end{enumerate}
\end{remark}

\begin{remark}[Periodic boundary conditions in all directions]
\label{r:periodic_BC}
The contents of Theorems \ref{t:global_primitive_strong_strong} and \ref{t:continuous_dependence} also hold in case the boundary conditions \eqref{eq:boundary_conditions_full} are replaced by periodic ones. The proofs remain essentially unchanged, as it is enough to neglect boundary contributions. 
\end{remark}

\section{Proof of Theorems \ref{t:local_primitive_strong_strong}, \ref{t:global_primitive_strong_strong} and \ref{t:continuous_dependence}}
\label{s:proof_main_results}
Recall that, in \eqref{eq:def_HV}, we set $H=\Hs^{1}(\Dom)\times H^1(\Dom)$ and $V=\Hs^{2}_{\n}(\Dom)\times \Hr^2(\Dom)$.

\subsection{Proof of Theorem \ref{t:local_primitive_strong_strong}}
\label{ss:proof_local}
The proof of Theorem \ref{t:local_primitive_strong_strong} follows as in \cite[Section 6.4]{Primitive1} by using the results of \cite{AV19_QSEE_1,AV19_QSEE_2}  (see \cite[Section 5.1]{Primitive1} for a similar situation). 

We begin by reformulating \eqref{eq:primitive}-\eqref{eq:boundary_conditions_full} as a stochastic evolution equation on the Banach space $V_0\stackrel{{\rm def}}{=}\Ls^2(\Dom)\times L^2(\Dom)$ for the unknown $U\stackrel{{\rm def}}{=}(v,\T)$:
\begin{equation}
\label{eq:abstract_formulation_strong_strong}
\left\{
\begin{aligned}
\dd  U+ A (\cdot)U\,\dd t
&=F(\cdot,U)\,\dd t + [(B_n (\cdot)U +G_n(\cdot,U))_{n\geq 1}]\,\dd \Br_{\ell^2}(t),\\
U(0)&=(v_0,\T_0),
\end{aligned}\right.
\end{equation}
where $(A,B,F,G)$ are given below and $\Br_{\ell^2}$ is as in \eqref{eq:def_Br}.
Before describing $(A,B,F,G)$, we introduce some more notation. 
Firstly, for a weakly differentiable map $f$, we set 
$$
[\op f](x)\stackrel{{\rm def}}{=}\nabla_{\h}
\int_{-h}^{\z}  f(x_{\h},\zeta)\,\dd \zeta, \quad \text{ for }\ \ x=(x_{\h},\z)\in \Tor^2\times (-h,0).
$$
Moreover, set 
\begin{align*}
\Lpp(v,\T) &\stackrel{{\rm def}}{=}\Big(\sum_{n\geq 1}\hp^{j,k}_n\q \Big[ (\phi_{n}\cdot\nabla) v +\nabla_{\h}\int_{-h}^{\cdot} \big(\ktwon(\cdot,\zeta)\T(\cdot,	\zeta)\big)\,\dd \zeta\Big]^j \Big)_{k=1}^2,\\
\Lpg(v,\T) &\stackrel{{\rm def}}{=}
\Big(\sum_{n\geq 1}\sum_{1\leq j\leq 2} \hp^{j,k}_n\q \big[ \gvn(v,\T)\big]^j 
\Big)_{k=1}^2.
\end{align*}
Note that $\Lpp(v,\T)+\Lpg(v,\T)=\Lp(v,\T)$, where $\Lp$ is as in \eqref{eq:Lp_def}. 

We can now make explicit the terms in \eqref{eq:abstract_formulation_strong_strong}:
\begin{align}
\label{eq:def_A}
A(\cdot,U)
&\stackrel{{\rm def}}{=}\begin{bmatrix}\vspace{0.05cm}
\Delta v -\p \big[\op(\kone \T+ (\tp\cdot\nabla)\T)+\Lpp(\cdot,v,\T)\big]\\
\Delta \T
\end{bmatrix},\\
\label{eq:def_B}
B_n(\cdot,U)
&\stackrel{{\rm def}}{=}\begin{bmatrix}\vspace{0.05cm}
 \p\big[(\phi_n\cdot\nabla) v+ \op(\ktwon \T)\big]\\
 (\psi_n\cdot\nabla) \T
\end{bmatrix}, \qquad  B(\cdot,U)=(B_n(\cdot,U))_{n\geq 1},\\
\label{eq:def_F}
F(\cdot,U)
&\stackrel{{\rm def}}{=}\begin{bmatrix}\vspace{0.05cm}
\p[ (v\cdot\nabla_{\h}) v + w(v)\cdot \partial_{3} v +\fv(\cdot,v,\T,\nabla v,\nabla\T)+\Lpg(\cdot,v,\T)]\\
 (v\cdot\nabla_{\h})\T+w(v)\partial_{3} \T + \ft(\cdot,v,\T)
\end{bmatrix},\\
\label{eq:def_G}
G(\cdot,U)
&\stackrel{{\rm def}}{=}\begin{bmatrix}\vspace{0.05cm}
 \p[\gvn(\cdot,v,\T)]\\
  \gtn(\cdot,v,\T)
\end{bmatrix},\qquad\qquad\qquad  G(\cdot,U)=(G_n(\cdot,U))_{n\geq 1},
\end{align}
where $w(v)$ is as in \eqref{eq:def_w}.

By virtue of Definition \ref{def:sol_strong_strong}, one can see  that $((v,\T),\tau)$ is a $L^2$-maximal (resp.\ $L^2$-local) strong solution to \eqref{eq:primitive}-\eqref{eq:boundary_conditions_full} if and only if $(U,\tau)$, where $U\stackrel{{\rm def}}{=}(v,\T)$, is an $L^2$-maximal (resp.\ $L^2$-local) solution to \eqref{eq:abstract_formulation_strong_strong} in the sense of \cite[Definition 4.4]{AV19_QSEE_1} (see also \cite[Remark 5.6]{AV19_QSEE_2} and Lemma \ref{l:smr} below).

Now Theorem \ref{t:local_primitive_strong_strong} can be proved as \cite[Theorem 6.4]{Primitive1}. To avoid repetitions, below we only give a sketch of the proof of the maximal $L^2$-regularity for the linearized system of \eqref{eq:primitive}-\eqref{eq:boundary_conditions_full} which is the main ingredient to apply the results of \cite{AV19_QSEE_1,AV19_QSEE_2} (see \cite[Proposition 6.8]{Primitive1} for the case of isothermal turbulent pressure). 
Below we set $H(\ell^2)\stackrel{{\rm def}}{=}\calL_2(\ell^2,H)$ where $\calL_2$ denotes the class of Hilbert-Schmidt operators.
 
\begin{lemma}[Stochastic maximal $L^2$-regularity]
\label{l:smr}
Let Assumption \ref{ass:well_posedness_primitive_double_strong}\eqref{it:well_posedness_measurability_strong_strong}-\eqref{it:well_posedness_primitive_kone_smoothness_strong_strong} be satisfied. Fix $T\in (0,\infty)$ and let $\tau$ be a stopping time with values in $[0,T]$. 
Let 
$$
f\in L^2_{\Progress}((0,\tau)\times \O;\Ls^2\times L^2) \quad \text{ and } \quad g\in  
L^2_{\Progress}((0,\tau)\times \O;H(\ell^2)).
$$
Then there exists a unique $U\in L^2_{\Progress}((0,\tau)\times \O;V)$ such that, a.s.\ for all $t\in [0,T]$, 
\begin{equation}
\label{eq:U_strong_solution}
U(t) =\int_{0}^t (A(s)U(s)+ f(s))\,\dd s +\int_{0}^t (B_n(s)U(s)+ g_n(s))_n\,\dd \Br_{\ell^2}(s) .
\end{equation}
Moreover, there exists  $C>0$, independent of $(f,g)$, such that for all $U\in L^2_{\Progress}((0,\tau)\times \O;V)$ satisfying \eqref{eq:U_strong_solution} one has  $U\in C([0,\tau];H)$ a.s.\ and 
$$
\E\|U\|_{C([0,\tau];H)}^2+
\E\|U\|_{L^2(0,\tau;V)}^2\leq C (
\E\|f\|_{L^2(0,\tau;\Ls^2\times L^2)}^2+
\E\|U\|_{L^2(0,\tau;H(\ell^2))}^2).
$$
\end{lemma}

Combining the above result with \cite[Proposition 3.9]{AV19_QSEE_1}, in \eqref{eq:U_strong_solution} we can also allow non-trivial initial data from the space $L^2_{\F_0}(\O;H)$.

\begin{proof}[Proof of Lemma \ref{l:smr} - Sketch]
The proof is similar to the one of \cite[Proposition 6.8]{Primitive1} and therefore we only give a sketch of its proof. As in \cite{Primitive1} we consider only the case $\g\equiv 0$ as one can check that the term $\Lpp(v,\T)$ is of lower order, and therefore \cite[Theorem 3.2]{AV_torus} applies. 

As in \cite{Primitive1}, we used the method of continuity of \cite[Proposition 3.13 and Remark 3.14]{AV19_QSEE_2}. Hence, for $\lambda\in [0,1]$, consider, on $\Dom$,
\begin{subequations}
	\label{eq:linear}
\begin{alignat}{3}
\label{eq:linear_1}
\begin{split}
\dd v -\Delta v\,\dd t&=\Big[ f_{v} +\lambda\op(\kone \T+ (\tp\cdot\nabla) \T)\Big]\,\dd t\\
&\ \  + \sum_{n\geq 1} \Big[\lambda\p[(\phi_n\cdot\nabla)v ]+\op(\ktwon \T)+ g_{v,n}\Big]\,\dd \beta^n_t,
\end{split}
\\
\label{eq:linear_2}
\dd \T -\Delta \T\,\dd t&= f_{\T} \,\dd t + \sum_{n\geq 1} \Big[(\psi_n\cdot\nabla) \T+ g_{\T,n}\Big]\,\dd \beta^n_t,\\
\label{eq:linear_3}
 \int_{-h}^{0} \div_{\h}v(\cdot,\zeta)\,\dd \zeta&=0,\\
\label{eq:linear_4}
 v(0,\cdot)&=0,\ \ \text{ and } \ \ \  \T(0,\cdot)=0. 
\end{alignat}
\end{subequations}
The above problem is complemented with the boundary conditions \eqref{eq:primitive_strong_BC}.

The above \emph{linear} problem \eqref{eq:linear} coincides with \eqref{eq:U_strong_solution} in case $\lambda=1$ (recall that we are assuming $\g\equiv 0$). By the above-mentioned method of continuity, it is enough to show a-priori estimates for $L^2$-strong solutions of \eqref{eq:linear} (i.e.\ $(v,\T)\in L^2((0,\tau)\times \O;V)\cap L^2(\O;C([0,\tau];H))$) with constants \emph{independent} of $\lambda$.
More precisely, by \cite[Proposition 3.13 and Remark 3.14]{AV19_QSEE_2}, it is enough to show that, for all $L^2$-strong solutions $(v,\T)$ to \eqref{eq:linear},
\begin{align}
\label{eq:a_priori_estimate_lambda}
\E\|\T\|_{L^2(0,\tau;H^2)}^2+ 
\E\|v\|_{L^2(0,\tau;H^2)}^2
&\lesssim \E\|f_{\T}\|_{L^2(0,\tau; L^2)}^2+
\E\|g_{\T}\|_{L^2(0,\tau;H^1(\ell^2))}^2\\
\nonumber
&\ +
\E\|f_{v}\|_{L^2(0,\tau; L^2)}^2+
\E\|g_{v}\|_{L^2(0,\tau;H^1(\ell^2))}^2
\end{align}
with an the implicit constant that is independent of $\lambda$. 
To this end, as in the proof of \cite[Proposition 6.8]{Primitive1}, the idea is to use the triangular structure of the system \eqref{eq:linear}, i.e.\ the velocity $v$ does not appear in the equation for the temperature \eqref{eq:linear_2}. Therefore, one can first obtain an estimate for $\T$ and then use it in estimating $v$.

We begin by estimating $\T$. As in \cite[Proposition 6.8]{Primitive1} (see also Step 1 in the proof of \cite[Proposition 4.1]{Primitive1}) an application of the It\^{o} formula to $\T\mapsto \|\nabla \T\|_{L^2}^2$, an integration by part and Assumption \ref{ass:well_posedness_primitive_double_strong}\eqref{it:well_posedness_primitive_parabolicity_strong_strong} yield
\begin{equation}
\label{eq:estimate_T_smr_strong}
\E\|\T\|_{L^2(0,\tau;H^2)}^2\lesssim
\E\|f_{\T}\|_{L^2(0,\tau; L^2)}^2+
\E\|g_{\T}\|_{L^2(0,\tau;H^1(\ell^2))}^2,
\end{equation}
where the implicit constant is independent of $\lambda\in [0,1]$ and we set 
$H^1(\ell^2)\stackrel{{\rm def}}{=}H^1(\Dom;\ell^2)$. 

The same argument also applies to $v$. Since $v$ solves \eqref{eq:linear_1},  we have
\begin{align}
\label{eq:estimate_v_smr_strong}
\E\|v\|_{L^2(0,\tau;H^2)}^2
&\lesssim 
\E\|f_{v}\|_{L^2(0,\tau; L^2)}^2+
\E\|g_{v}\|_{L^2(0,\tau;H^1(\ell^2))}^2\\
\nonumber
& +\E \|(\op(\ktwo_n \T))_{n\geq 1}\|_{L^2(0,\tau;H^1(\ell^2))}^2\\
\nonumber
& + \E\|\op(\kone \T+ (\tp\cdot\nabla)\T)\|_{L^2(0,\tau;L^2)}^2,
\end{align}
where the implicit constant is independent of $\lambda\in [0,1]$.

By \eqref{eq:estimate_T_smr_strong}-\eqref{eq:estimate_v_smr_strong}, to obtain \eqref{eq:a_priori_estimate_lambda}, it remains to show that, for all $\varphi\in H^2$,
\begin{equation}
\label{eq:estimate_op_ktwon_etc}
\|(\op(\ktwo_n \varphi))_{n\geq 1}\|_{H^1(\ell^2)}
 + \|\op(\kone \varphi)\|_{L^2}
 +  \|\op((\tp\cdot\nabla)\varphi)\|_{L^2} \lesssim_M \|\varphi\|_{H^2},
\end{equation} 
where $M$ is as in Assumption \ref{ass:well_posedness_primitive_double_strong}. 
For brevity, we only provide some details for the estimate of $
\|(\op(\ktwo_n \varphi))_{n\geq 1}\|_{H^1(\ell^2)}$. The other follows similarly by using 
Assumption \ref{ass:well_posedness_primitive_double_strong}\eqref{it:well_posedness_primitive_kone_smoothness_strong_strong} instead of 
 \ref{ass:well_posedness_primitive_double_strong}\eqref{it:well_posedness_primitive_ktwo_smoothness_strong_strong}.

Let $r\in (1,\infty)$ be such that $\frac{1}{r}
+\frac{1}{2+\delta}=\frac{1}{2}$, where $\delta>0$ is as in Assumption \ref{ass:well_posedness_primitive_double_strong}. 
To estimate $
\|(\op(\ktwo_n \varphi))_{n\geq 1}\|_{H^1(\ell^2)}$, firstly, note that,
\begin{align*}
\|(\op(\ktwon \varphi))_{n\geq 1}\|_{L^2(\ell^2)}
&\lesssim 
\Big\|\Big(\int_{-h}^0 \|\ktwo\|_{\ell^2}^2\,\dd \zeta\Big)^{1/2}\Big(\int_{-h}^0 |\nabla_{\h} \varphi|^2\,\dd \zeta\Big)^{1/2}\Big\|_{L^2(\Tor^2)}\\
&+
\Big\|\Big(\int_{-h}^0 \|\nabla_{\h} \ktwo\|_{\ell^2}^2\,\dd \zeta\Big)^{1/2}\Big(\int_{-h}^0 | \varphi|^2\,\dd \zeta\Big)^{1/2}\Big\|_{L^2(\Tor^2)}\\
&\leq \|\sigma\|_{L^{\infty}(\Tor^2;L^2(-h,0;\ell^2))} \|\varphi\|_{H^1}\\
&+ 
\|\sigma\|_{H^{1,2+\delta}(\Tor^2;L^2(-h,0;\ell^2))}\|\varphi\|_{L^{r}(\Tor^2;L^2(-h,0))}\\
&\stackrel{(i)}{\lesssim}_M \|\varphi\|_{H^{1}(\Tor^2;L^2(-h,0))}\lesssim \|\varphi\|_{H^1},
\end{align*}
where in $(i)$ we used $H^1(\Tor^2;L^2(-h,0))\embed L^r(\Tor^2;L^2(-h,0))$ and Assumption \ref{ass:well_posedness_primitive_double_strong}\eqref{it:well_posedness_primitive_ktwo_smoothness_strong_strong}.
The estimate of $\|(\nabla\op(\ktwon \varphi))_{n\geq 1}\|_{L^2(\ell^2)}$ is similar, where one also uses that $\partial_{3} \op(\ktwon \varphi)= \nabla_{\h}(\ktwon \varphi)$,
$$
H^2\embed L^{\infty}\quad \text{ and }\quad 
H^2\embed L^2(\Tor^2;H^2(-h,0))\embed L^2(\Tor^2;L^{\infty}(-h,0)),
$$ 
by Sobolev embeddings.
This completes the proof of \eqref{eq:estimate_op_ktwon_etc} and the claim of Lemma \ref{l:smr} follows.
\end{proof}

\subsection{Proof of Theorem \ref{t:global_primitive_strong_strong}}
\label{ss:proof_global}
As commented below the statements of Theorems \ref{t:global_primitive_strong_strong} and \ref{t:continuous_dependence}, the following result is the key ingredient in their proofs.
Recall that $\y$ is defined in \eqref{eq:def_y}.

\begin{proposition}[Energy estimate for maximal solutions]
\label{prop:energy_estimate_primitive_strong_strong}
Let Assumptions \ref{ass:well_posedness_primitive_double_strong} and \ref{ass:global_primitive_strong_strong} be satisfied. Let $T\in (0,\infty)$.
Assume that $(v_0,\T_0)\in L^{4}_{\F_0}(\O;\Hs^1\times H^1)$. Let $((v,\T),\tau)$ be the $L^2$-maximal strong solution to \eqref{eq:primitive}-\eqref{eq:boundary_conditions_full} provided by Theorem \ref{t:local_primitive_strong_strong}.
Then
\begin{equation}
\label{eq:v_T_almost_surely_H_1_H_2}
\sup_{s\in [0,\tau\wedge T)}\big[\|v(s)\|^2_{H^1}+\|\T(s)\|^2_{H^1}] 
+\int_0^{\tau\wedge T}\big[ \|v(s)\|^2_{H^2}+ \|\T(s)\|^2_{H^2}\big]\,\dd s<\infty \ \text{ a.s.\ }
\end{equation}
Moreover,
there exists $C_T>0$, independent of $(v_0,\T_0)$, such that, for all $\g>e^e$,
\begin{align*}
\E \sup_{t\in [0,\tau\wedge T)}\|v(t)\|^2_{L^2} +\E\int_0^{\tau\wedge T} \|v(t)\|^2_{H^1}\,\dd t
&\leq C_T (1+\E \|\y\|_{L^2(0,T)}^2+\E\|v_0\|_{L^2}^2+\E\|\T_0\|_{L^2}^2),\\
\E \sup_{t\in [0,\tau\wedge T)}\|\T(t)\|^2_{L^2}+ \E\int_0^{\tau\wedge T} \|\T(t)\|^2_{H^1}\,\dd t
&\leq C_T(1+\E \|\y\|_{L^2(0,T)}^2+\E\|v_0\|_{L^2}^2+\E\|\T_0\|_{L^2}^2),\\
\P\Big(\sup_{s\in [0,\tau\wedge T)}\|v(t)\|^2_{H^1} +\int_0^{\tau\wedge T} \|v(t)\|^2_{H^2}\,\dd t\geq \g\Big)
&\leq C_T\frac{(1+\E \|\y\|_{L^2(0,T)}^2+\E\|v_0\|_{H^1}^4+\E\|\T_0\|_{H^1}^4)}{ \log\log\log(\g)} ,\\
\P\Big(\sup_{t\in [0,\tau\wedge T)}\|\T(t)\|^2_{H^1} +\int_0^{\tau\wedge T} \|\T(t)\|^2_{H^2}\,\dd t\geq \g\Big)
&\leq C_T\frac{ (1+\E \|\y\|_{L^2(0,T)}^2+\E\|v_0\|_{H^1}^4+\E\|\T_0\|_{H^1}^4)}{ \log\log\log(\g)}.
\end{align*}
\end{proposition}

The proof of Proposition \ref{prop:energy_estimate_primitive_strong_strong} is postponed to Section \ref{s:proof_energy_estimate_conclusion} and Sections \ref{s:basic_estimate}-\ref{s:intermediate_estimate} are preparatory to its proof. In this section, we show that Theorems \ref{t:global_primitive_strong_strong} and \ref{t:continuous_dependence} follow from Proposition \ref{prop:energy_estimate_primitive_strong_strong}.
More precisely, Theorem \ref{t:global_primitive_strong_strong}
follows from the blow-up criteria of Theorem \ref{t:local_primitive_strong_strong} and \eqref{eq:v_T_almost_surely_H_1_H_2}, see e.g.\ the proof of \cite[Theorem 3.7]{Primitive1} for a similar situation. For the reader's convenience, we provide some details.
The estimates of Proposition \ref{prop:energy_estimate_primitive_strong_strong} will be used to prove Theorem \ref{t:continuous_dependence}.

\begin{proof}[Proof of Theorem \ref{t:global_primitive_strong_strong}]
By localization of solutions to stochastic evolution equations (see \cite[Proposition 4.13]{AV19_QSEE_2}), it is enough to consider $(v_0,\T_0)\in L^{\infty}(\O;\Hs^1\times H^1)$. Hence, for all $T\in (0,\infty)$,
$$
\P(\tau<T)\stackrel{\eqref{eq:v_T_almost_surely_H_1_H_2}}{=}
\P\Big(\tau <T,\,\sup_{t\in [0,\tau)}\big\|(v(t),\T(t))\big\|^2_{H}+\int_0^{\tau}\big\|(v(t),\T(t))\big\|_{V}^2\,\dd t<\infty \Big)\stackrel{(i)}{=}0,
$$
where in $(i)$ we used Theorem \ref{t:local_primitive_strong_strong}. Since $T\in (0,\infty)$ is arbitrary, the above yields $\tau=\infty$ a.s.

The estimates in Theorem \ref{t:global_primitive_strong_strong} follow from the one in Proposition \ref{prop:energy_estimate_primitive_strong_strong} with $\tau=\infty$.
\end{proof}

\subsection{Proof of Theorem \ref{t:continuous_dependence}}
\label{ss:proof_continuity}
To prove Theorem \ref{t:continuous_dependence} we argue as in \cite{AV_variational}. As in the proof of Theorem \ref{t:continuous_dependence} readily follows from the following result.

\begin{proposition}
\label{prop:Lip_estimate_continuity}
Let Assumptions \ref{ass:well_posedness_primitive_double_strong} and \ref{ass:global_primitive_strong_strong} be satisfied. Fix $T\in (0,\infty)$ and $(v_0,\T_0),(v_0',\T_0')\in L^{4}_{\F_0}(\O;H)$. 
Let $(v,\T)$ and $(v',\T')$ be the $L^2$-global strong solution to \eqref{eq:primitive}-\eqref{eq:boundary_conditions_full} provided provided by Theorem \ref{t:global_primitive_strong_strong} with initial data $(v_0,\T_0)$ and $(v_0',\T_0')$, respectively. 
Then there exist mappings $\psi,N :[0,\infty)\to [0,\infty)$, independent of $(v_0,\T_0), (v_0',\T_0')$, such that,
 for all $R,\varepsilon>0$,%
\begin{align*}
\P\Big(\|(v,\T)-(v',\T')\|_{C([0,T];H)\cap L^2(0,T;V)} >\varepsilon \Big)
 \leq \frac{\psi(R)}{\varepsilon^2}  \E\|(v_0,\T_0)-(v_0',\T_0')\|_{H}^2&\\
 + N(R)\Big( 1+\E\|\Xi\|_{L^2(0,T)}^2+\E\|(v_0,\T_0)\|_{H}^4+ \E\|(v_0',\T_0')\|_{H}^4\Big),&
\end{align*}
and $\displaystyle{\lim_{R \to \infty} N(R) =0}$.
\end{proposition}

\begin{proof}
To economize the notation, here we adopt the one used in Subsection \ref{ss:proof_local} for the proof of Theorem \ref{t:local_primitive_strong_strong}. In particular $(A,B,F,G)$ are as in \eqref{eq:def_A}-\eqref{eq:def_G} and $U=(v,\T)$. Similarly $U=(v,\T)$, $U'=(v',\T')$, $V_0= \Ls^2\times L^2$, $H(\ell^2)\stackrel{{\rm def}}{=} \calL_2(\ell^2,H)$ etc. Moreover, for notational convenience, we set
$$
V_{\theta}\stackrel{{\rm def}}{=}[V_0,V]_{\theta}\  \text{ for $\theta\in (0,1)$ \ \ (complex interpolation).}
$$
Note that $V_{1/2}=H$.
Since $V\embed H^2\times H^2$ and $V_0\embed L^2\times L^2$, 
we have $
V_{\theta}\embed H^{\theta}\times H^{\theta}.
$

Next note that the difference $\v\stackrel{{\rm def}}{=}U-U'$ solves
\begin{equation*}
\left\{
\begin{aligned}
\dd \v-A \v\,\dd t&= (F(U)-F(U'))\,\dd t + [B \v + (G(U)-G(U'))]\,\dd \Br_{\ell^2},\\
\v(0)&=U_0-U_0'.
\end{aligned}\right.
\end{equation*}
Fix $T\in (0,\infty)$. By Lemma \ref{l:smr} and \cite[Proposition 3.9 and 3.12]{AV19_QSEE_1} there exists $C_0>0$,  independent of $U_0,U_0'$, such that for all stopping times $(\eta,\xi)$ such that $0\leq \eta\leq \xi\leq T$ a.s.\ 
\begin{align}
\label{eq:smr_estimate_difference}
&\E\sup_{s\in [\eta,\xi]}\|\v(s)\|_{H}^2
+\E \|\v(s)\|^2_{L^2(\eta,\xi;V)} \leq C_0\E\|\v(\eta)\|^2_{H}\\
\nonumber
&\qquad\qquad
+C_0 \E\|F(U)-F(U')\|_{L^2(0,T;V_0)}^2
+C_0  \E\|G(U)-G(U')\|_{L^2(0,T;H(\ell^2))}^2.
\end{align}
Next, we estimate the nonlinearities $(F,G)$. The arguments in \cite[Theorem 3.4]{Primitive1} show the existence of $m\geq 1$, $(\rho_j)_{j=1}^m$ such that, for all $U,U'\in X_1$,
\begin{equation}
\label{eq:estimate_nonlinearities_continuity_proof}
\begin{aligned}
\|F(U)-F(U')\|_{V_0}
+ \|G(U)-G(U')\|_{H(\ell^2)}
\lesssim \sum_{1\leq j\leq m}(1+\|U\|_{V_{\beta_j}}^{\rho_j}+\|U'\|_{V_{\beta_j}}^{\rho_j})\|U-U'\|_{V_{\beta_j}},
\end{aligned}
\end{equation}
where $\beta_j=\frac{2+\rho_j}{2(1+\rho_j)}\in (\frac{1}{2},1)$ and the implicit constant is independent of $U,U'$.
Note that, by standard interpolation arguments, for $\theta_j=2\beta_j-1\in (0,1)$, we have
$$
\|x\|_{V_{\beta_j}}\lesssim \|x\|_{H}^{1-\theta_j}\|x\|_{V}^{\theta_j}, \quad \text{ for all }x\in V.
$$
Hence, for all $x,x'\in V$ and $\eta>0$,
\begin{align*}
\|x\|_{V_{\beta_j}}^{\rho_j}\|x'\|_{V_{\beta_j}}
&\lesssim 
\|x\|_{H}^{\rho_j(1-\theta_j)}\|x\|_{V}^{\rho_j\theta_j}\|x'\|_{H}^{1-\theta_j}\|x'\|_{V}^{\theta_j}\\
&\stackrel{(i)}{\leq} C_{\eta}
\|x\|_{H}^{\rho_j}\|x\|_{V} \|x'\|_{H}+ \eta
\|x'\|_{V},
\end{align*}
where in $(i)$ we used the Young's inequality with exponents $(\frac{1}{1-\theta_j},\frac{1}{\theta_j})$ and that 
$
\frac{\rho_j\theta_j}{1-\theta_j}=1
$
since $\beta_j=\frac{2+\rho_j}{2(1+\rho_j)}$. Combining the above estimate with \eqref{eq:estimate_nonlinearities_continuity_proof} we have, for all $\eta>0$ and $x,x'\in V$,
\begin{align*}
\|F(\cdot,x)-F(\cdot,x')\|_{V_0}
&+ \|G(\cdot,x)-G(\cdot,x')\|_{H(\ell^2)}\\
& \leq \Big(\sum_{1\leq j\leq m} (1+ \|x\|_{H}^{\rho_j}\|x\|_{V}+\|x'\|_{H}^{\rho_j}\|x'\|_{V}) \|x-x'\|_{H}\Big)
  + m \eta \|x-x'\|_{V}.
\end{align*}
Choosing $\eta=\frac{1}{4C_0m}$, the above inequality and \eqref{eq:smr_estimate_difference} yield, for some $c_0>0$ independent of $U_0,U_0'$,
\begin{align} 
\label{eq:estimate_grownall_prepared_U_star}
\E\sup_{s\in [\eta,\xi]}\|\v(s)\|_{H}^2
&+\E\int_{\eta}^{\xi} \|\v\|_{V}^2 \,\dd s
\leq c_0 \E\|\v(\eta)\|_{H}^2\\
\nonumber
%&\qquad\quad
&+c_0   
\E \int_{\eta}^{\xi} \Big[\underbrace{\sum_{1\leq j\leq m}\big(1+\|U\|_{H}^{2\rho_j}\|U\|_{V}^2 +\|U'\|_{H}^{2\rho_j}\|U'\|_{V}^2 \big)}_{M\stackrel{{\rm def}}{=}}\Big]\|\v\|^2_{H}\,\dd s.
\end{align}
Note that $M\in L^1(0,T)$ a.s.\ since $U,U'\in C([0,T];H)\cap L^2(0,T;V)$ a.s.
By the tail estimates of Theorem \ref{t:global_primitive_strong_strong}, there exists a mapping $N:[0,\infty)\to [0,\infty)$, independent of $U_0,U_0'$, such that $\lim_{R\to \infty} N(R)=0$ and for all $R>1$
\begin{equation}
\label{eq:tail_probability_M_U_star}
\P \Big(\int_0^{T} M_s\,\dd s \geq R\Big)\leq N(R) ( 1+\E\|\Xi\|_{L^2(0,T)}^2+\E\|(v_0,\T_0)\|_{H}^4+ \E\|(v_0',\T_0')\|_{H}^4).
\end{equation}
The conclusion follows from \eqref{eq:estimate_grownall_prepared_U_star}-\eqref{eq:tail_probability_M_U_star} and 
the Gronwall lemma in \cite[Lemma A.1]{AV_variational}.
\end{proof}

\begin{proof}[Proof of Theorem \ref{t:continuous_dependence}]
Due to Proposition \ref{prop:Lip_estimate_continuity}, the proof of Theorem \ref{t:continuous_dependence} follows verbatim from the one of \cite[Theorem 3.8]{AV_variational}.
\end{proof}

\section{Basic estimates}
\label{s:basic_estimate}
The aim of this section is to prove the following result. Recall that $H$ is defined in \eqref{eq:def_HV}.

\begin{lemma}
\label{l:basic_estimates}
Let Assumptions \ref{ass:well_posedness_primitive_double_strong} and \ref{ass:global_primitive_strong_strong} be satisfied. Let $T\in (0,\infty)$.
Assume that $(v_0,\T_0)\in L^{4}_{\F_0}(\O;H)$. Let $((v,\T),\tau)$ be the $L^2$-maximal strong solution to  \eqref{eq:primitive}-\eqref{eq:boundary_conditions_full} provided by Theorem \ref{t:local_primitive_strong_strong}.
%%%%%%%%
 Then, for all $\g>1$,
\begin{align*}
\E \sup_{s\in [0,\tau\wedge T)}\|v(s)\|^2_{L^2} +\E\int_0^{\tau\wedge T}\|v(s)\|^2_{H^1}\,\dd s
& \lesssim_T 1+\E\|v_0\|_{L^2}^2+\E\|\T_0\|_{L^2}^2+\E\|\Xi\|_{L^2(0,T)}^2,\\
\E \sup_{s\in [0,\tau\wedge T)}\|\T(s)\|^2_{L^2}
+ \E\int_0^{\tau\wedge T} \|\T(s)\|^2_{H^1}\,\dd s 
& \lesssim_T 1+\E\|v_0\|_{L^2}^2+\E\|\T_0\|_{L^2}^2+\E\|\Xi\|_{L^2(0,T)}^2,\\
\P\Big(\sup_{s\in [0,\tau\wedge T)} \|\T(s)\|_{L^4}^4\geq  \g\Big)
&\lesssim_T \frac{1+\E\|v_0\|_{L^2}^2+\E\|\T_0\|_{L^4}^4+\E\|\Xi\|_{L^2(0,T)}^2}{\log (\g)},\\
\P\Big(  \int_0^{\tau\wedge T}\int_{\Dom} |\T|^2|\nabla \T|^2 \,\dd x\dd s \geq \g\Big)
& \lesssim_T \frac{1+\E\|v_0\|_{L^2}^2+\E\|\T_0\|_{L^4}^4+\E\|\Xi\|_{L^2(0,T)}^2}{\log (\g)},
\end{align*}
where the implicit constants in the above estimates are independent of $(v_0,\T_0)$.
\end{lemma}
 
The first two inequalities are standard energy estimates and coincide with the first two estimates in Proposition \ref{prop:energy_estimate_primitive_strong_strong}. The last two estimates are rather weak and do not give any information on moments of the r.v.\ $\sup_{s\in [0,\tau\wedge T)} \|\T(s)\|_{L^4}^4$ and $\int_0^{\tau\wedge T}\int_{\Dom} |\T|^2|\nabla \T|^2 \,\dd x\dd s$. However, it does not seem possible to improve them in general. 
%%%%
Note that  
$
\int_{\Dom} \big|\nabla |\T|^2\big|^2\,\dd x \eqsim 
\int_{\Dom} |\T|^2 |\nabla 	\T|^2\,\dd x.
$
Combining the Sobolev embedding $H^1(\Dom)\embed L^{6}(\Dom)$, well-known interpolation inequalities and the last two estimates of Lemma \ref{l:basic_estimates} we get, for all $\eta\in [0,1]$ and $p\stackrel{{\rm def}}{=}\frac{4}{1-\eta}$, $q\stackrel{{\rm def}}{=}\frac{12}{1+2\eta}$,
\begin{equation}
\label{eq:P_T_eta_tail_estimate}
\P\big(\|\T\|_{L^{p}(0,\tau\wedge T;L^{q}(\Dom))}\geq \g\big)
\lesssim_{T,\eta} \frac{1+\E\|v_0\|_{L^2}^2+\E\|\T_0\|_{L^4}^4+\E\|\Xi\|_{L^2(0,T)}^2}{\log (\g)}.
\end{equation}
As in Lemma \ref{l:basic_estimates}, the implicit constant in \eqref{eq:P_T_eta_tail_estimate} is independent of $(v_0,\T_0)$.

The energy estimates of Lemma \ref{l:basic_estimates} and of Proposition \ref{prop:energy_estimate_primitive_strong_strong} are based on certain cancellations of the nonlinearities in \eqref{eq:primitive}. We formulate the one needed in the current work in the following.

\begin{lemma}[Cancellations]
\label{l:cancellation}
Assume that 
$\vv=(\vv^{k})_{k=1}^3\in C^{\infty}(\overline{\Dom};\R^3)$ satisfies 
$$
\vv^{3}(\cdot,-h)=
\vv^{3}(\cdot,0)=0\text{ on }\Tor^2 \quad \text{ and }\quad 
\div \,\vv=0\text{ on }\Dom.
$$
Then, for all integers $r\geq 2$ and all $f\in C^{\infty}(\overline{\Dom};\R^3)$, $g\in C^{\infty}(\Dom)$,
$$
\int_{\Dom} |f|^{r} g^{r-1} [(\vv\cdot \nabla) g]\,\dd x
+\int_{\Dom} g^{r} |f|^{r-2} f\cdot [(\vv\cdot \nabla) f]\,\dd x=0.
$$ 
\end{lemma}

To prove Lemma \ref{l:basic_estimates} we use the above with $g\equiv 1$. However, in the proof of Proposition \ref{prop:energy_estimate_primitive_strong_strong} we also need the case $f, g\neq 1$. 
To check the smoothness assumptions, we will use that $(u,f,g)$ are Sobolev maps and the density of smooth functions in Sobolev spaces.

\subsection{Proof of Lemmas \ref{l:basic_estimates}-\ref{l:cancellation} }
We first prove Lemma \ref{l:cancellation} and afterwards Lemma \ref{l:basic_estimates}.

\begin{proof}[Proof of Lemma \ref{l:cancellation}]
Integrating by parts, we have
\begin{align*}
\int_{\Dom} |f|^{r} g^{r-1} [(\vv\cdot \nabla) g]\,\dd x
&= \frac{1}{r} 
\int_{\Dom} |f|^{r}  (\vv\cdot \nabla) \big[g^{r}\big]\,\dd x\\
&\stackrel{(i)}{=} \frac{1}{r} 
\int_{\Dom} |f|^{r}  \div(\vv \, g^{r})\,\dd x\stackrel{(ii)}{=}-
\int_{\Dom} g^{r} |f|^{r-2} f\cdot [(\vv\cdot \nabla) f]\,\dd x,
\end{align*}
where in $(i)$ we used $\div\, \vv=0$ and in $(ii)$ that 
$\vv^{3}(\cdot,-h)=
\vv^{3}(\cdot,0)=0$ on $\Tor^2$.
\end{proof}

Before going into the proof of Lemma \ref{l:basic_estimates}, let us recall the boundedness of the trace operator on the boundary $\partial\Dom=\Tor^2 \times \{h,0\}$ (see e.g.\ \cite[Proposition 1.6, Chapter 4]{TayPDE1}),
\begin{equation}
\label{eq:boundedness_trace}
H^{1/2+r}(\Dom)\ni  f\mapsto f|_{\Tor^2 \times \{h,0\}} \in H^{r}(\Tor^2 \times \{-h,0\}) \
 \text{ for all }\ r>0.
\end{equation}

\begin{proof}[Proof of Lemma \ref{l:basic_estimates}]
The first two estimates of 
Lemma \ref{l:basic_estimates} can be proven as in \cite[Lemma 5.2]{Primitive1} with minor modifications.
To avoid repetitions we omit the details.
To prove the third estimate in 
Lemma \ref{l:basic_estimates} we employ the stochastic Grownall lemma \cite[Lemma A.1]{AV_variational}.  To this end, we need a localization argument. Throughout this proof, we fix $T\in (0,\infty)$. For each $j\geq 1$, let 
\begin{equation}
\label{eq:def_tau_j}
\begin{aligned}
\tau_j\stackrel{{\rm def}}{=}\inf\big\{t\in [0,\tau)&\,:\,\|v(t)\|_{H^1}+\| v\|_{L^2(0,t;H^2)} \\
&+\|\T(t)\|_{L^2}+\| \T\|_{L^2(0,t;H^1)} +\|\y\|_{L^2(0,t:L^2)}\geq j\big\}\wedge T, 
\end{aligned}
\end{equation}
where $\inf\emptyset\stackrel{{\rm def}}{=}\tau$ and $\y$ is as \eqref{eq:def_y}.
Note that $(\tau_j)_{j\geq 1}$ is a localizing sequence for $(v,\tau\wedge T)$. 
In particular $\lim_{j\to \infty}\tau_j =\tau\wedge T$. 
Moreover, by Definition \ref{def:sol_strong_strong} and \eqref{eq:def_tau_j} we have, 
uniformly in $\O$ and for all $j\geq 1$ (recall that $(H,V)$ are as in \eqref{eq:def_HV}), that
\begin{equation}
\label{eq:T_v_stopped_are_good}
\begin{aligned}
(v,\T)\in C([0,\tau_j];H)\cap L^2(0,\tau_j;V)\ \text{ a.s.\ }
\end{aligned}
\end{equation}

Fix $j\geq 1$ and let $(\eta,\xi)$ be stopping times such that $0\leq \eta\leq \xi\leq \tau_j$ a.s. 
We claim that there exists $c_0\geq 1$ independent of $(j,\eta,\xi,v_0,\T_0)$ such that 
\begin{equation}
\begin{aligned}
\label{eq:grownall_L_6}
\E\Big[\sup_{t\in [\eta,\xi]} \|\T(t)\|_{L^4}^4 \Big]
&+
\E\int_{\eta}^{\xi}\int_{\Dom} |\T|^2|\nabla \T|^2\,\dd x\dd s
\\
&\leq c_0(1+ \E\|\T(\eta)\|_{L^4}^4)
+c_0 \E\int_{\eta}^{\xi}N(s)(1+\|\T(s)\|_{L^4}^4)\,\dd s
\end{aligned}
\end{equation}
where
$
N_{v,\T}(t)\stackrel{{\rm def}}{=}1+\|v(t)\|_{H^1}^2+\|\T(t)\|_{H^1}^2 + \Xi(t).
$
To economize the notation, for all $j\geq 1$, we set
$$
E_j \stackrel{{\rm def}}{=}
\sup_{t\in [0,\tau_j]} \|\T(t)\|_{L^4}^4 +\int_{0}^{\tau_j}\int_{\Dom} |\T|^2|\nabla \T|^2\,\dd x.
$$
Suppose for a moment that \eqref{eq:grownall_L_6} holds. Then, by \cite[Lemma A.1]{AV_variational}, we have for all $R,\g>1$,
\begin{align*}
\P(E_j \geq \g)
&\leq \frac{8c_0}{\g}e^{8c_0 R} (1+\E\|\T_0\|_{L^4}^4)+ \P\Big(\int_0^{\tau_j} N(s)\,\dd s \geq \frac{R}{2c_0}\Big)\\
&\leq \frac{8c_0}{\g}e^{8c_0 R} (1+\E\|\T_0\|_{L^4}^4)+\frac{2c_0}{R} (1+\E\|v_0\|_{L^2}^2+\E\|\T_0\|_{L^2}^2+\E\|\y\|_{L^2(0,T)}^2)\\
&\leq \Big( \frac{8c_0}{\g}e^{8c_0 R} + \frac{2c_0}{R}\Big) (1+\E\|v_0\|_{L^4}^4+\E\|\T_0\|_{L^4}^4+\E\|\y\|_{L^2(0,T)}^2).
\end{align*}
Now choosing $R = \frac{1}{8c_0} \log(\frac{\g^2}{ \log(\g)}) \geq \frac{1}{8c_0} \log(\g)$ for $\g$ large, we have for some $C_0>0$ (depending only on $c_0$),
$$
\P(E_j\geq \g)\leq C_0\frac{1+\E\|v_0\|_{L^4}^4+\E\|\T_0\|_{L^4}^4+\E\|\y\|_{L^2(0,T)}^2}{\log(\g)} .
$$
Since $C_0(c_0)$ is independent of $(j,v_0,\T_0)$, the last estimate in Lemma \ref{l:basic_estimates} follows by letting $j\to \infty$ in the previous estimate.

Hence it remains to prove \eqref{eq:grownall_L_6}. To this end, we set
\begin{align*}
f_{\T}\stackrel{{\rm def}}{=} \one_{[0,\tau)\times\O}F_{\T}(\cdot,v,\T,\nabla v,\nabla \T) \quad \text{ and }\quad 
g_{\T,n}\stackrel{{\rm def}}{=}  \one_{[0,\tau)\times\O} G_{\T,n}(\cdot,v,\T,\nabla v,\nabla \T).
\end{align*}
Note that, by Assumption \ref{ass:well_posedness_primitive_double_strong}\eqref{it:nonlinearities_strong_strong} and \eqref{eq:def_y}, a.s.\ for all $t\in \R_+$,
\begin{equation}
\|f_{\T}(t)\|_{L^2}+
\|g_{\T}(t)\|_{H^1(\ell^2)}\lesssim 1+ \y+ \|v(t)\|_{H^1}+\|\T(t)\|_{H^1}.
\end{equation}
Applying the It\^{o}'s formula to $\T\mapsto \|\T\|^4_{L^4}$ (using  a standard approximation argument, see e.g.\ \cite[Step 3 of Lemma 5.3]{Primitive1}) we have, a.s.\ for all $t\in [0,T]$,
\begin{align}
\label{eq:Ito_L_6_estimate}
\|\T^{\eta,\xi}(t)\|_{L^4}^4
&+12\int_{0}^t \int_{\Dom}\one_{[\eta,\xi]} \T^2 |\nabla\T|^2\,\dd x\dd s  \\
\nonumber
&=
\|\T(\eta)\|_{L^4}^4+\sum_{1\leq j\leq 3} \int_0^t \one_{[\eta,\xi]}
I_{\T,j}(s) \,\dd s +\MT(t),
\end{align}
where $I_{\T,1}(t)\stackrel{{\rm def}}{=} -\alpha \int_{\Tor^2}|\T(\cdot,0)|^4 \,\dd x_{\h}$,
\begin{align*}
I_{\T,2}(t)&\stackrel{{\rm def}}{=}4\int_{\Dom} \T^3 f_{\T}\,\dd x, \qquad 
I_{\T,3}(t)\stackrel{{\rm def}}{=}12\sum_{n\geq 1} \int_{\Dom} \T^2 |(\psi_n\cdot\nabla)\T+g_{\T,n}|^2\,\dd x,\\
\MT(t)&\stackrel{{\rm def}}{=}
4\sum_{n\geq 1}\int_{0}^t \one_{[\eta,\xi]}\int_{\Dom}\T^3 ((\psi_n\cdot\nabla)\T+g_{\T,n})\,\dd x\dd \beta_s^n,
\end{align*}
and we used the cancellation  
$$
\int_{\Dom} \T^3\, [(v\cdot\nabla_{\h} )\T+ w(v)\partial_{3} \T] \,\dd x=0
$$ 
which follows from Lemma \ref{l:cancellation} with $g=1$, $f=\T$, $\vv=(v,w(v))$ and a standard density argument. 
For the convenience of the exposition, the remaining part of the proof is split into several steps.

\emph{Step 1: There exists $c_1$ independent of $(j,\eta,\xi,v_0,\T_0)$ such that}
\begin{align*}
\E\int_{\eta}^{\xi} |\T^2| |\nabla \T|^2\,\dd x\dd s 
\leq c_1 (1+\E\|\T_0\|_{L^4}^4)
+c_1\E \int_{\eta}^{\xi} N(s) (1+\|\T(s)\|_{L^4}^4) \,\dd s,
\end{align*}
\emph{where $N$ is as below \eqref{eq:grownall_L_6}.}
%%%%
We begin by estimating $I_{\T,1}$. Let $\varepsilon>0$ be 
%decided 
fixed later. Note that, by \eqref{eq:boundedness_trace} and interpolation, we have, a.e.\ on $[0,\tau)\times \O$,
\begin{align*}
I_{\T,1}=\||\T(\cdot,0)|^2\|_{L^2(\Tor^2)}^2
&\leq \varepsilon \|\nabla |\T|^2\|_{L^2}^2 + C_{\varepsilon}\|\T\|_{L^4}^4\\
&\leq \varepsilon \int_{\Dom}|\T|^2|\nabla \T|^2\,\dd x + C_{\varepsilon}\|\T\|_{L^4}^4.
\end{align*}
Next we estimate $I_{\T,2}$:
\begin{align*}
|I_{\T,2}|
&\leq \big\||\T|^3\big\|_{L^2}\|f_{\T}\|_{L^2}
= \big\||\T|^{2}\big\|_{L^{3}}^{3/2}\|f_{\T}\|_{L^2}\\
&\stackrel{(i)}{\lesssim }
\big\||\T|^{2}\big\|_{L^{2}}^{3/4}\Big(\big\||\T|^{2}\big\|_{L^{2}}+\big\|\nabla [|\T|^{2}]\big\|_{L^{2}}\Big)^{3/4}\|f_{\T}\|_{L^2}\\
&\lesssim 
\|\T\|_{L^4}^{3}\|f_{\T}\|_{L^2} + \|\T\|_{L^4}^{3/2} \big\||\T| |\nabla \T|\big\|_{L^2}^{3/4}\|f_{\T}\|_{L^2}\\
&\stackrel{(ii)}{\leq} \|\T\|_{L^4}^{3}\|f_{\T}\|_{L^2} + \varepsilon \big\||\T| |\nabla \T|\big\|_{L^2}^{2}+ 
C_{\varepsilon}\|\T\|_{L^4}^{12/5}\|f_{\T}\|_{L^2}^{8/5}\\
&\leq 
\varepsilon \big\||\T|^2 |\nabla \T|\big\|_{L^2}^2+  
C_{\varepsilon}(1+\|f_{\T}\|_{L^2}^2) (1+\|\T\|_{L^4}^{4}),
\end{align*}
where in $(i)$ we used the interpolation inequality $\|\zeta\|_{L^{3}}\lesssim \|\zeta\|_{L^2}^{1/2}\|\zeta\|_{H^1}^{1/2}$ for $\zeta\in H^1(\Dom)$ and in $(ii)$ the Young's inequality with exponents $(\frac{8}{3},\frac{8}{5})$.

It remains to estimate $I_{\T,3}$. By the Cauchy-Schwartz inequality we have, a.e.\ on $[0,\tau)\times \O$,
\begin{align*}
|I_{\T,3}|
&\leq 6(1+\varepsilon)
\sum_{n\geq 1}\Big[\int_{\Dom} \T^2 |(\psi_n\cdot\nabla)\T|^2\,\dd x
+C_{\varepsilon}
\int_{\Dom} \T^4 |g_{\T,n}|^2\,\dd x\Big]\\
&\leq 
6(1+\varepsilon)\ellip
\int_{\Dom} \T^2 |\nabla \T|^2\,\dd x
+C_{\varepsilon} \sum_{n\geq 1}
 \int_{\Dom} \T^2 |g_{\T}|^2\,\dd x,
\end{align*}
where in the last step we used Assumption \ref{ass:well_posedness_primitive_double_strong}\eqref{it:well_posedness_primitive_parabolicity_strong_strong}.
We now estimate the last term appearing in the above estimate:
\begin{equation}
\label{eq:T_6_g_estimate}
\begin{aligned}
\Big|
\sum_{n\geq 1} \int_{\Dom} \T^2 |g_{\T,n}|^2\,\dd x
\Big|
\leq \|\T\|_{L^{4}}^2 \|(g_{\T,n})_{n\geq 1} \big\|_{L^4(\ell^2)}^2
&\lesssim
 \|\T\|_{L^{4}}^2\|g_{\T}\|_{H^1(\ell^2)}^2,
\end{aligned}
\end{equation}
where in the last estimate we used the Sobolev embedding $H^{1}(\Dom;\ell^2)\embed L^6(\Dom;\ell^2)$.

Taking $t=T$ in \eqref{eq:Ito_L_6_estimate} and afterwards the expected values,  the previous estimates show that 
\begin{align*}
12\E\int_{\eta}^{\xi}|\T|^2|\nabla \T|^2\,\dd x\dd s \leq 
(6\ellip(1+\varepsilon)+2 \varepsilon) 
\E\int_{\eta}^{\xi}|\T|^2|\nabla \T|^2\,\dd x\dd s + C_{\varepsilon}
\E\int_{\eta}^{\xi} N(s)(1+\|\T\|_{L^4}^4)\,\dd s.
\end{align*}
Here we have also used that $\E[\MT(T)]=\E[\MT(0)]=0$. Recall that $\ellip<2$. Thus the claim of this step follows by choosing $\varepsilon$ so that $(6\ellip(1+\varepsilon)+2 \varepsilon) < 12$ in the above estimate.

\emph{Step 2: There exists $c_2>0$, independent of $(j,\eta,\xi,v_0,\T_0)$, such that}
\begin{align*}
\E\Big[\sup_{s\in [0,T]}|\MT(s)|\Big]
&\leq \frac{1}{2}
\E\Big[\sup_{s\in [\eta,\xi]}\|\T(s)\|_{L^4}^4\Big]
+c_2 (1+\E\|\T_0\|_{L^4}^4)\\
&+ c_2 \E\int_{\eta}^{\xi} N(s) (1+\|\T(s)\|_{L^4}^4) \,\dd s.
\end{align*}

The Burkholder-Davis-Gundy inequality yields:
\begin{align*}
\E \Big[\sup_{s\in [0,T]}|\MT(s)| \Big]
\lesssim
\E\Big[\int_{\eta}^{\xi}\sum_{n\geq  1} \Big(\int_{\Dom}|\T|^3 |(\psi_n\cdot\nabla)\T+g_{\T,n}|\,\dd x\Big)^{2}\,\dd s\Big]^{1/2}.
\end{align*}
The Cauchy-Schwartz inequality, a.e.\ on $[0,\tau)\times \O$,
\begin{align*}
\sum_{n\geq 1} \Big(\int_{\Dom}|\T|^3 |(\psi_n\cdot\nabla)\T+g_{\T,n}|\,\dd x\Big)^{2}
& \leq \|\T\|_{L^4}^4 \Big[ \int_{\Dom} |\T|^2  \sum_{n\geq 1}  |(\psi_n\cdot\nabla)\T + g_{\T,n}|^2\,\dd x\Big]\\
 &\lesssim \|\T\|_{L^4}^4 \Big[ \int_{\Dom} |\T|^2 ( |\nabla\T|^2 +\|g_{\T}\|_{\ell^2}^2)\,\dd x\Big],
\end{align*}
where in the last estimate we used boundedness of $(\psi_{n})_{n\geq 1}$, cf.\ Remark \ref{r:boundedness}. 

Hence, the Young inequality yields
\begin{align*}
\E \Big[\sup_{s\in [0,T]}|\MT(s)| \Big]
&\lesssim \E \Big[\Big(\sup_{s\in [\eta,\xi]}\|\T\|_{L^4}^4\Big)^{1/2} 
\Big(\int_{\eta}^{\xi} \int_{\Dom} |\T|^2 ( |\nabla\T|^2 +\|g_{\T}\|_{\ell^2}^2)\,\dd x\dd s\Big)^{1/2}\Big]\\
& \leq \frac{1}{2} \E \Big[\sup_{s\in [\eta,\xi]}\|\T(s)\|_{L^4}^4\Big] +  
C \E \int_{\eta}^{\xi}\int_{\Dom} |\T|^2\big( |\nabla\T|^2+ \|g_{\T}\|_{\ell^2}^2\big)\,\dd x\dd s.
\end{align*}
The claim of Step 2 follows by combining the previous estimate with Step 1 and \eqref{eq:T_6_g_estimate}.

\emph{Step 3: Proof of \eqref{eq:grownall_L_6}}.
Taking $\E\big[\sup_{t\in [0,T]}|\cdot|\big]$ on both sides of \eqref{eq:Ito_L_6_estimate} the claim follows by repeating the estimates for $(I_{\T,j})_{j=1}^3$ performed in Step 1 and using the estimate for $\mathcal{M}$ of Step 2. Note that the term $\frac{1}{2}
\E\big[\sup_{s\in [\eta,\xi]}\|\T(s)\|_{L^4}^4\big]$ can be absorbed on the left-hand side of the corresponding estimate since $\T^{\eta,\xi}=\T((\cdot\vee\eta)\wedge \xi)$.
\end{proof}

\section{The main intermediate estimate}
\label{s:intermediate_estimate}
The aim of this section is to obtain the following key estimate for the $L^2$-maximal strong solution to \eqref{eq:primitive}-\eqref{eq:boundary_conditions_full}, which is the main ingredient in the proof of Proposition \ref{prop:energy_estimate_primitive_strong_strong}.
%%%%%%
As in \cite[Subsection 5.2]{Primitive1}, inspired by the seminal work of C.\ Cao and E.S.\ Titi \cite{CT07}, the main estimate involves the barotropic and baroclinic modes, i.e.\ 
\begin{equation}
\label{eq:splitting_v}
\overline{v}\stackrel{{\rm def}}{=}\fint_{-h}^0 v(\cdot,\zeta)\,\dd \zeta
\qquad
\text{ and }
\qquad
\wt{v}\stackrel{{\rm def}}{=}v-\overline{v}.
\end{equation}

\begin{lemma}[Main intermediate estimate]
\label{l:main_estimate}
Let Assumptions \ref{ass:well_posedness_primitive_double_strong} and \ref{ass:global_primitive_strong_strong} be satisfied. Fix $T\in (0,\infty)$.
Assume that $(v_0,\T_0)\in L^{4}_{\F_0}(\O;\Hs^1\times H^1)$.
Let $((v,\T),\tau)$ be the $L^2$-maximal strong solution to \eqref{eq:primitive}-\eqref{eq:boundary_conditions_full} provided by Theorem \ref{t:local_primitive_strong_strong}. For all $s\in [0,\tau)$, set 
\begin{align*}
X_s&\stackrel{{\rm def}}{=} 
\|\wt{v}(s)\|_{L^4(\Dom)}^4
+
\|\overline{v}(s)\|_{H^1(\Tor^2)}^2
+
\| \partial_{3} v(s)\|_{L^2(\Dom)}^2
+
\|\partial_{3} \T(s)\|_{L^2(\Dom)}^2,\\
Y_s
&\stackrel{{\rm def}}{=}
\Big\||\wt{v}(s)||\nabla \wt{v}(s)| \Big\|_{L^2(\Dom)}^2
+
\|\overline{v}(s)\|_{H^2(\Tor^2)}^2
+
\| \partial_{3} v(s)\|_{H^1(\Dom)}^2\\
& \ +
\big\|\partial_{3} \T(s)\big\|_{H^1(\Dom)}^2+ 
\Big\||\wt{v}(s)||\nabla \T(s)| \Big\|_{L^2(\Dom)}^2.
\end{align*}
Then there exists $C_T>0$, independent of $(v_0,\T_0)$, such that, for all $\g>e$,
\begin{equation}
\label{eq:tail_probability_main_estimate}
\P\Big(\sup_{s\in [0,\tau\wedge T)} X_s +  \int_0^{\tau\wedge T} Y_s\,\dd s \geq \g\Big)
 \lesssim_T\frac{1+\E \|\y\|_{L^2(0,T)}^2+ \E \|v_0\|_{H^1}^4+ \E \|\T_0\|_{H^1}^4}{\log\log(\g)}.
\end{equation}
\end{lemma}

The proof of the above result requires several steps which are spread over this section. 
The proof of Lemma \ref{l:main_estimate} will be given in Subsection \ref{ss:proof_lemma_main_estimate}.
%We remark that $\E[X_0]<\infty$ by virtue of the Sobolev embedding $H^1(\Dom)\embed L^6(\Dom)$. 

Lemma \ref{l:main_estimate} can be seen as an extension of \cite[Lemma 5.3]{Primitive1} to the case of non-isothermal turbulent pressure. Note that the estimate of the tail probability \eqref{eq:tail_probability_main_estimate} was not given in \cite{Primitive1}. In case of isothermal turbulent pressure (i.e.\ $\ktwon\equiv 0$ and $\tp\equiv 0$), the decay factor $(\log\log(\g))^{-1}$ on the right hand side of \eqref{eq:tail_probability_main_estimate} can be replaced by $(\log(\g))^{-1}$, cf.\  Remark \ref{r:comparison_primitive1} for a similar situation.

As in \cite{Primitive1}, to prove the above main estimate we follow the approach of the second author and T.\ Kashiwabara in \cite{HK16}. There the main idea was to prove three estimates separately for the variables 
$\overline{v},\wt{v}$ and $\partial_{3} v$. Afterwards, one multiplies these estimates with suitable constants and by summing them up, one obtains a closed estimate (cf.\ \cite[Lemma 5.3]{Primitive1}).  Since in  \cite{Primitive1} we were concerned with the case of \emph{isothermal} turbulent pressure, we were able to follow the strategy of \cite{HK16} as the temperature $\T$ played only a minor role in the estimates (see the discussion in Subsection \ref{ss:novelty}).
Indeed, in the case of isothermal turbulent pressure, the energy bound for $\sup_{t}\|\T\|_{L^2_x}^2 + \|\T\|_{L^2_t(L^2_x)}^2$ in Lemma \ref{l:basic_estimates} already gives enough information on $\T$ to obtain global well-posedness, cf.\ the proof of  \cite[Theorem 3.7]{Primitive1}. 
However, this is \emph{not} true in the case of non-isothermal turbulent pressure. Indeed, if $\ktwon\neq 0$, then the term
\begin{equation}
\label{eq:extra_term_main_estimate}
\p \Big[\int_{-h}^{\cdot} \nabla_{\h} (\ktwon(\cdot,\zeta)\T(\cdot,\zeta))\,\dd \zeta\Big] \, \dd \beta_t^n 
\end{equation}
appearing in \eqref{eq:primitive_v_T_pressure_1} cannot be controlled via Lemma \ref{l:basic_estimates} in the strong setting, cf.\ Lemma \ref{l:smr}. In other words, the action of $\T$ through the term \eqref{eq:extra_term_main_estimate} in the $v$-equation is \emph{not} lower order. 
Hence, in contrast to \cite{Primitive1}, we need to consider the equations for $v$ and $\T$ \emph{jointly}. 
This gives rise to some new terms in the equations for $\overline{v}$ and $\wt{v}$ which we are going to describe. 
To explain the new quantities arising in the estimates, let us follow the argument in \cite[Lemma 5.3]{Primitive1} and therefore we first look at the estimate for $\overline{v}$. Taking the averaging operator $\overline{\cdot}=\fint_{-h}^0 \cdot\,\dd \zeta$ in \eqref{eq:primitive_v_T_pressure_1}, one sees that the the following term appears
\begin{equation}
\label{eq:avarage_whT_appears}
\overline{\int_{-h}^{\cdot} \nabla_{\h}(\ktwon(\cdot) \T(\cdot,\zeta))\,\dd \zeta}
= \nabla_{\h}\Big(\ktwon(\cdot)  \overline{\int_{-h}^{\cdot}  \T(\cdot,\zeta)\,\dd \zeta} \Big)
=-\nabla_{\h}(\ktwon(\cdot)\wh{\T}(\cdot)),
\end{equation}
where we used that $\ktwon$'s are $\z$-independent by Assumption \ref{ass:global_primitive_strong_strong} and we set
\begin{equation}
\label{eq:hat_T}
\wh{\T}\stackrel{{\rm def}}{=}\fint_{-h}^0  \T(\cdot,\zeta)\zeta\,\dd \zeta.
\end{equation}

\begin{remark}[Physical interpretation of $\wh{\T}$]
Recall that  $\T$ is proportional to $ \rho$, cf.\ \eqref{eq:density_law_linear_det} and \eqref{eq:simplification_in_coeffcients}. Hence the ratio $\wh{\T}/\overline{\T}$ is equal to the center of gravity in the vertical direction. 
\end{remark}

To repeat the argument of \cite[Step 1 of Lemma 5.3]{Primitive1}, by stochastic maximal $L^2$-regularity (cf.\ Lemma \ref{l:smr}), in order to obtain $L^{\infty}_t(H^1_x)\cap L^2_t(H^2_x)$-estimates for $\overline{v}$, we need $L^{2}_t(H^1_x)$-estimates for $\wh{\T}$. However, the latter estimate does not follow from Lemma \ref{l:basic_estimates}. Thus we need an additional argument to obtain the required $L^{2}_t(H^1_x)$-estimates for $\wh{\T}$.
This this end, we apply the weighted average operator $\wh{\cdot}=\fint_{-h}^0 \cdot\, \zeta\,\dd \zeta$ in \eqref{eq:primitive_v_T_pressure_2}, and then the following term appears
\begin{align}
\label{eq:wh_partial_3_T_equality}
\reallywidehat{w(v)\partial_{3} \T}
&=
\fint_{-h}^0 w(v)\partial_{3} \T \, \zeta\,\dd \zeta\\
\nonumber
&\stackrel{(i)}{=}-\fint_{-h}^{0} \Big([w(v)](\cdot,\zeta)\T(\cdot,\zeta)
-\div_{\h} v(\cdot,\zeta)\T(\cdot,\zeta)\zeta \Big)\,\dd \zeta\\
\nonumber
&\stackrel{(ii)}{=}
-\fint_{-h}^{0} \Big[\Big(\int_{\zeta}^0 \div_{\h} v(\cdot,\xi)\,d\xi\Big)\T
- \div_{\h} v(\cdot,\zeta) \T(\cdot,\zeta) \zeta \Big]\,\dd \zeta\\
\nonumber
&=
\fint_{-h}^{0} \Big[-
\div_{\h}v (\cdot,\zeta)\Big(\int_{-h}^{\zeta} \T(\cdot,\xi)\,d\xi \Big)+\div_{\h} v(\cdot,\zeta) \T(\cdot,\zeta) \zeta 
\Big]\,\dd \zeta,
\end{align}
where in $(i)$ we use an integration by parts and $[w(v)](\cdot,-h)=[w(v)](\cdot,0)=0$, in $(ii)$  \eqref{eq:def_w} and $\int_{-h}^0 \div_{\h} v\,\dd \zeta=0 $. 
Therefore, to obtain $L^2(H^1)$-estimates for $\wh{\T}$ we need to bound the products 
\begin{equation}
\label{eq:additional_to_bound_term_mixed}
\big\||\T|\, |\nabla \wt{v}|\big\|_{L^2((0,\tau)\times\Dom)}^2 \qquad \text{ and } \qquad 
\Big\||\int_{-h}^{\cdot} \T(\cdot,\zeta)\,\dd \zeta|\, |\nabla \wt{v}|\Big\|_{L^2((0,\tau)\times\Dom)}^2.
\end{equation}
Such quantities can be estimated by applying the It\^o formula to the functionals
\begin{equation*}
(\wt{v},\T)\mapsto \big\||\T|\, | \wt{v}|\big\|_{L^2(\Dom)}^2 \qquad \text{ and } \qquad 
(\wt{v},\T)\mapsto  \Big\||\int_{-h}^{\cdot} \T(\cdot,\zeta)\,\dd \zeta| \, | \wt{v}|\Big\|_{L^2(\Dom)}^2,
\end{equation*}
respectively. 
For details, see Subsections \ref{ss:estimate_tilde_v_T} and \ref{ss:estimate_tilde_v_varphi}. 
The quantities in \eqref{eq:additional_to_bound_term_mixed} also arise in the estimate for $\wt{v}$.  
%%%
Interestingly, compared to \cite[Lemma 5.3]{Primitive1}, no further terms appear in the estimate for $\partial_{3} v$, see Subsection \ref{ss:estimate_partial_3_v}.
Finally, as it will turn out, to bound the quantities in \eqref{eq:additional_to_bound_term_mixed}, we need an estimate also for the quantities
\begin{equation*}
\Big\| |\T|\, \big|\int_{-h}^{\cdot} \nabla_{\h}\T\,\dd \zeta\big|\Big\|_{L^2((0,\tau)\times\Dom)}^2
\quad \text{ and }\quad
\Big\|\big|\int_{-h}^{\cdot} \T\,\dd \zeta\big|\, \big|\int_{-h}^{\cdot} \nabla_{\h}\T\,\dd \zeta\big|\Big\|_{L^2((0,\tau)\times\Dom)}^2,
\end{equation*}
respectively.
To estimate the above terms we apply the It\^o formula to the functionals
$$
\T\mapsto \Big\| |\T|\, \big|\int_{-h}^{\cdot} \T(\cdot,\zeta)\,\dd \zeta\big|\Big\|_{L^2(\Dom)}^2
\quad \text{ and }\quad
\T\mapsto \Big\| \int_{-h}^{\cdot} \T(\cdot,\zeta)\,\dd \zeta\Big\|_{L^4(\Dom)}^4.
$$
After that, Lemma \ref{l:main_estimate} follows by multiplying each estimate with a suitable constant and summing them up, see \cite[Step 4 of Lemma 5.3]{Primitive1} for a similar situation.

To economize the notation, below, for $(t,\om)\in \R_+\times \O$, $x_{\h}\in \Tor^2$ and $\z\in (-h,0)$, we let
\begin{equation}
\label{eq:def_TT}
\TT(t,\om,x_{\h},\z)\stackrel{{\rm def}}{=} \int_{-h}^{\z} \T(t,\om,x_{\h},\zeta)\,\dd \zeta.
\end{equation}
%%%
Next, we give an overview of this section.
\begin{itemize}
\item Subsection \ref{ss:set_up_main_estimate}: Equations for the new quantities $(\overline{v},\wt{v},\wh{\T})$.
\item Subsection \ref{ss:preparation}: Set-up of the proof of Lemma \ref{l:main_estimate}. 
\item Subsection \ref{ss:estimate_overline_v}: Estimate for $\sup_t\| \overline{v}\|_{H^1_x}$ and $\| \overline{v}\|_{L^2_t H^2_x}$.
\item Subsection \ref{ss:estimate_wh_T}: 
Estimate for $\sup_{t}\|\wh{\T} \|_{H^1_x}$ and $\|\wh{\T}\|_{L^2_t H^2_x}$.
\item Subsection \ref{ss:estimate_partial_3_v}: Estimate for $\sup_t \|\partial_{3} v\|_{L^2_x}$ and $\|\partial_{3}v\|_{L^2_t H^1_x}$.
\item  Subsection \ref{ss:estimate_partial_T_3}: Estimate for $\sup_t\|\partial_{3}\T \|_{L^2_x}$ and $\|\partial_{3}\T\|_{L^2_t H^1_x}$.
\item Subsection \ref{ss:estimate_wt_v_L_4}: Estimate for $\sup_t\|\wt{v} \|_{L^4_x}$ and $\||\wt{v}||\nabla \wt{v}|\|_{L^2_t L^2_x}$.
\item Subsection \ref{ss:estimate_tilde_v_T}: 
Estimate for $\||\wt{v}||\nabla \T|\|_{L^2_t L^2_x}$ and 
$\||\T| |\nabla \wt{v}|\|_{L^2_t L^2_x}$.
\item Subsection \ref{ss:estimate_tilde_v_varphi}: Estimate for $\||\TT|  |\nabla\wt{v}|\|_{L^2_t L^2_x}$ and $ 
\| |\wt{v}||\nabla\TT|\|_{L^2_t L^2_x}$.
\item Subsection \ref{ss:estimate_T_TT}: Estimate for $\||\T|  |\nabla\TT\|_{L^2_t L^2_x}$.
\item Subsection \ref{ss:estimate_TT}: Estimate for $\||\TT|  |\nabla\TT\|_{L^2_t L^2_x}$.
\item Subsection \ref{ss:proof_lemma_main_estimate}: Lemma \ref{l:main_estimate} obtained by multiplying with suitable constants the estimates of Subsections \ref{ss:estimate_overline_v}-\ref{ss:estimate_TT} and then summing them up.
\end{itemize}

In the following subsections, the assumptions of Lemma \ref{l:main_estimate} are in force. In particular, $((v,\T),\tau)$ is the $L^2$-maximal strong solution to \eqref{eq:primitive}-\eqref{eq:boundary_conditions_full} provided by Theorem \ref{t:local_primitive_strong_strong}, see Definition \ref{def:sol_strong_strong}.

\subsection{System of SPDEs for the unknown $(\overline{v},\wt{v},\wh{\T})$}
\label{ss:set_up_main_estimate}
By Definition \ref{def:sol_strong_strong} and Assumptions \ref{ass:well_posedness_primitive_double_strong} and \eqref{ass:global_primitive_strong_strong}, $((v,\T),\tau)$ is an $L^2$-local strong solution to (cf.\ Definition \ref{def:sol_strong_strong})
\begin{subequations}
\label{eq:primitive_v_proof_global}
\begin{alignat}{3}
\label{eq:primitive_v_proof_global_1}
\begin{split}
\dd  v 
&=\Big(\Delta v+\p\big[-(v\cdot \nabla_{\h})v-\w(v)\partial_{3} v+ \Lt \T+ \Lpp (v,\T) + \fvt \big] \Big)\, \dd t \\
&\qquad \ \ +\sum_{n\geq 1}\p\Big[(\phi_{n}\cdot\nabla) v+ \ktwon \int_{-h}^{\cdot}\nabla_{\h}\T(\cdot,\zeta)\,\dd \zeta + \gvtn \Big] \, \dd {\beta}_t^n,
\end{split}\\
\label{eq:primitive_v_proof_global_2}
\begin{split}
\dd  \T&=\Big[\Delta \T -(v\cdot\nabla_{\h}) \T -\w(v) \partial_{3} \T+ f_{\T} \Big]\, \dd t
+\sum_{n\geq 1}\Big[(\psi_n\cdot \nabla) \Temp+g_{\T,n}\Big] \, \dd \beta_t^n,
\end{split}\\
v(0,\cdot)&=v_0, \quad \T(0,\cdot)=\T_0,
\end{alignat}
\end{subequations}
where for $n\geq 1$, $\q$ as in Subsection \ref{ss:set_up} and on $[0,\tau)\times \O$, we set
\begin{subequations}
\label{eq:inhomogeneity_v_T}
\begin{alignat}{3}
\label{eq:inhomogeneity_v_T_1}
\Lt \T&\stackrel{{\rm def}}{=} 
 (\tp_{\h}\cdot\nabla_{\h}) \int_{-h}^{\cdot} \nabla_{\h} \T(\cdot,\zeta)\,\dd \zeta + \int_{-h}^{\cdot} \tp^3(\cdot,\zeta) \partial_3\nabla_{\h} \T(\cdot,\zeta)\,\dd \zeta ,\\
\label{eq:inhomogeneity_v_T_2}
 \Lpp(v,\T) &\stackrel{{\rm def}}{=}\sum_{n\geq 1}\sum_{1\leq j\leq 2} \hp^{j,\cdot}_n  \Big(\q\Big[(\phi_n\cdot \nabla) v+ \ktwon \int_{-h}^{\cdot} \nabla_{\h} \T(\cdot,\zeta)\,\dd \zeta\Big]\Big)^j ,\\
\label{eq:inhomogeneity_v_T_4}
\gvtn
&\stackrel{{\rm def}}{=}\gvn(\cdot,v,\T) + (\nabla_{\h}\ktwon) \int_{-h}^{\cdot}\T(\cdot,\zeta)\,\dd \zeta, \qquad 
  g_{v}\stackrel{{\rm def}}{=}(g_{v,n})_{n\geq 1},\\
  \label{eq:inhomogeneity_v_T_5}
   g_{\T,n}&\stackrel{{\rm def}}{=}\gtn(\cdot,v,\T),
  \qquad \qquad \qquad  \qquad \qquad  \qquad \  g_{\T}\stackrel{{\rm def}}{=}(g_{\T,n})_{n\geq 1},\\
\label{eq:inhomogeneity_v_T_3}
\begin{split}
\fvt
&\stackrel{{\rm def}}{=}\fv(\cdot,v,\T,\nabla v,\nabla \T)+ \int_{-h}^{\cdot} \nabla(\kone(\cdot,\zeta)\T(\cdot,\zeta))\,\dd \zeta\\
& \qquad  \qquad\qquad  \qquad\quad  \   +\sum_{n\geq 1}\sum_{1\leq j\leq 2} \hp_{n}^{j,\cdot}(\q [ \gvtn])^j,
\end{split}\\
\label{eq:inhomogeneity_v_T_6}
f_{\T}&\stackrel{{\rm def}}{=} \ft(\cdot,v,\T,\nabla v,\nabla \T). 
\end{alignat}
\end{subequations}
Finally, let us recall that \eqref{eq:primitive_v_proof_global} is complemented with the following boundary conditions:
\begin{subequations}
	\label{eq:primitive_v_proof_global_bc}
	\begin{alignat}{2}
		\label{eq:primitive_v_proof_global_bc_1}
		\partial_{3} v (\cdot,-h)=\partial_{3} v(\cdot,0)=0 \ \  \text{ on }\Tor^2,&\\
		\label{eq:primitive_v_proof_global_bc_2}
		\partial_{3} \T(\cdot,-h)= \partial_{3} \T(\cdot,0)+\alpha \T(\cdot,0)=0\ \  \text{ on }\Tor^2.&
	\end{alignat}
\end{subequations}
Note that to derive \eqref{eq:primitive_v_proof_global} we used that $\ktwon$ and $\tp_{\h}=(\tp^1,\tp^2)$ are $\z$-independent by Assumption \ref{ass:global_primitive_strong_strong} and 
$\int_{-h}^{\cdot}\partial_j \tp^3(\cdot,\zeta) \partial_3 \T(\cdot,\zeta)\,\dd \zeta=\partial_j \tp^3 \T$ by \eqref{eq:primitive_v_proof_global_bc_2}.
As above, here $(\cdot)^j$ denotes the $j$-th coordinate of the corresponding vector. $L^2$-local strong solutions to \eqref{eq:primitive_v_proof_global} can be defined as in Definition \ref{def:sol_strong_strong}, we omit the details for brevity. 

The logic behind the definition \eqref{eq:inhomogeneity_v_T} is that the quantities in \eqref{eq:inhomogeneity_v_T_3}-\eqref{eq:inhomogeneity_v_T_5} are lower-order in the sense that they can be estimated (in strong $L^2$-norms) due to the standard energy estimates of Lemma \ref{l:basic_estimates} and Assumption \ref{ass:well_posedness_primitive_double_strong}\eqref{it:nonlinearities_strong_strong} (cf.\ \eqref{eq:def_N_v_T}-\eqref{eq:tail_probability_N_v_T} below). This is not the case for the \emph{linear} operators in $(v,\T)$ appearing \eqref{eq:inhomogeneity_v_T_2}, due to our (relatively) weak regularity assumptions on $(\hp,\tp)$ in Assumption \ref{ass:well_posedness_primitive_double_strong}. It is easy to see that, under additional assumption on $(\hp,\tp)$, also the quantities  in \eqref{eq:inhomogeneity_v_T_1}-\eqref{eq:inhomogeneity_v_T_2} can by the energy estimates in 
Lemma \ref{l:basic_estimates}. However, it would be unnatural to enforce the regularity assumptions on $(\hp,\tp)$ as they will appear naturally when dealing with the Stratonovich formulation of \eqref{eq:primitive}, see Section \ref{s:Stratonovich}.

Next we derive SPDEs for the unknown $(\overline{v},\wt{v},\wh{\T})$. We begin by considering $\overline{v}=\fint v(\cdot,\zeta)\,\dd \zeta$.
To this end, let us recall that $\pr$ denotes the Helmholtz projection acting on the horizontal variable $x_{\h}\in \Tor^2$ where $x=(x_{\h},\z)\in \Dom$, see Subsection \ref{ss:set_up}. Since $\overline{\p v}= \pr \overline{v}$, applying the vertical average $\overline{\cdot}=\fint_{-h}^0\cdot\,\dd \zeta$ in \eqref{eq:primitive_v_proof_global_1} and using Assumption \ref{ass:global_primitive_strong_strong}, $(\overline{v}, \tau)$ is a $L^2$-local strong solution the following problem on $\Tor^2$:
\begin{subequations}
\label{eq:primitive_bar}
\begin{alignat}{3}
\label{eq:primitive_bar_1}
\begin{split}
\dd \overline{v}
 &=\Big(\Delta_{\h} \overline{v}+ \pr\Big[-(\overline{v}\cdot \nabla_{\h})\overline{v}-\force(\wt{v}) 
 -(\tp\cdot \nabla_{\h})\nabla_{\h}\wh{\T}\\
 &\qquad \qquad \qquad\qquad\qquad  - \reallywidehat{\tp^3 \partial_3 \nabla_{\h}\T}+ \overline{\fvt}+ \Lpp(v,\T) \Big]\Big)\, \dd t\\
&\quad \ \ +\sum_{n\geq 1}\pr \Big[(\phi_{n,\h}\cdot\nabla_{\h}) \overline{v}+\overline{\phi^3_n \partial_{3} v } 
- \ktwon \nabla_{\h} \wh{\T} +\overline{\gvtn}\Big] \, \dd {\beta}_t^n,
\end{split}\\
\label{eq:primitive_bar_2}
\force(\wt{v})
&\stackrel{{\rm def}}{=}
\overline{(\wt{v}\cdot \nabla_{\h}) \wt{v}+\wt{v} (\div_{\h} \wt{v})},
\\
\overline{v}(0,\cdot)&=\overline{v}_0\stackrel{{\rm def}}{=}\fint_{-h}^0 v_0(\cdot,\zeta)\,\dd \zeta, 
\end{alignat}
\end{subequations}
where $\phi_{n,\h}\stackrel{{\rm def}}{=}(\phi^1_n,\phi^2_n)$. To obtain \eqref{eq:primitive_bar_1} we also used \eqref{eq:avarage_whT_appears}, 
\begin{align*}
\overline{(v\cdot \nabla_{\h})v+ \w(v)\partial_{3} v} = 
(\overline{v }\cdot\nabla_{\h} ) \overline{v}+
\overline{(\wt{v}\cdot\nabla_{\h}) \wt{v}+(\div_{\h} \wt{v})\, \wt{v} }
\end{align*}
which follows from  $\overline{\wt{v}}=0$ and an integration by parts, and $\overline{\Lpp(v,\T)}=\Lpp(v,\T)$ which follows from the $\z$-independence of $\hp_{n}^{j,k}$ (see Assumption \ref{ass:global_primitive_strong_strong}) and the fact that $\q[\cdot]= \qr[\overline{\cdot}]$ is $\z$-independent as well (see \eqref{eq:Helmholtz_hydrostatic}).
%
%%%%%%%%%
Here, as above, by $L^2$-local strong solution to \eqref{eq:primitive_bar} we understand that $(\overline{v},\tau)$ solves \eqref{eq:primitive_bar} in its natural integral form, cf.\ Definition \ref{def:sol_strong_strong}. Note that $\div_{\h} \overline{v}_0=0$ since $v_0\in \Hs^1$. Hence, by \eqref{eq:primitive_bar_1}, 
\begin{equation}
\label{eq:divergence_free_bar_v}
\div_{\h} \overline{v}=0 \  \ \text{ a.e.\ on }[0,\tau)\times \O\times \Tor^2.
\end{equation}

Next, we derive a system of SPDEs for $\wt{v}$. To this end, we apply the deviation from the vertical average operator $\wt{\cdot}=\cdot-\overline{\cdot}$ in \eqref{eq:primitive_v_proof_global}. Note that $\p f-\overline{\p f}=f-\overline{f}$ for all $f\in L^2(\Dom;\R^2)$ by \eqref{eq:Helmholtz_hydrostatic}. Using \eqref{eq:avarage_whT_appears}, one sees that $(\wt{v}, \tau)$ is a $L^2$-local strong solution to
\begin{subequations}
\label{eq:primitive_tilde}
\begin{alignat}{5}
\label{eq:primitive_tilde_1}
\begin{split}
\dd  \,\wt{v} 
&= \Big[\Delta \wt{v}-(\wt{v}\cdot \nabla_{\h})\wt{v}- \w(v) \partial_{3} \wt{v} +\forcetwo(\wt{v},\overline{v}) \\
&\qquad \qquad \qquad + \Lt \T+(\tp_{\h}\cdot\nabla_{\h})\nabla_{\h}\wh{\T}+\reallywidehat{\tp^3\partial_3 \nabla_{\h}\T}+ \wt{\fvt} \,\Big]\, \dd  t\\
&\quad 
+\sum_{n\geq 1}\Big[(\phi_{n}\cdot\nabla) \wt{v}-\overline{\phi^3_n \partial_{3} v}+\opt_n \T +\wt{\gvtn}\,\Big] \, \dd {\beta}_t^n, 
\end{split}\\
\forcetwo(\wt{v},\overline{v})
&\stackrel{{\rm def}}{=} -(\wt{v}\cdot \nabla_{\h}) \overline{v} - (\overline{v}\cdot \nabla_{\h}) \wt{v} +\force(\wt{v}),\\
\opt_n \T
&\stackrel{{\rm def}}{=} \ktwon \Big(\int_{-h}^{\cdot} \nabla_{\h }\T\,\dd \zeta
 - \nabla_{\h}\wh{\T}\Big),\\
\wt{v}(0,\cdot)&=\wt{v}_0\stackrel{{\rm def}}{=}v_0-\overline{v}_0.
\end{alignat}
\end{subequations}
where $\force$ is as in \eqref{eq:primitive_bar_2} and we used that $\partial_{3} v=\partial_{3}\wt{v}$. By \eqref{eq:primitive_v_proof_global_bc_1} we also have
\begin{equation}
\partial_{3} v(\cdot,-h)=
\partial_{3} v(\cdot,0)=0 \ \ \text{ on }\Tor^2.
\label{eq:primitive_tilde_bc}
\end{equation}
Before going further, let us note that by \eqref{eq:divergence_free_bar_v}, we have 
$$
w(v)=w(\wt{v})  \  \ \text{ a.e.\ on }[0,\tau)\times \O\times \Dom.
$$
The previous identity will be used often in the following without further mentioning it.

Finally, we consider $\wh{\T}$. By taking the weighted average operator $\wh{\cdot}=\fint_{-h}^0 \cdot\, \zeta \dd \zeta$ in the second equation of \eqref{eq:primitive_v_proof_global}, we have that $(\wh{\T},\tau)$ is an $L^2$-local strong solution to
\begin{subequations}
\label{eq:theta_hat_estimate}
\begin{alignat}{4}
\label{eq:theta_hat_estimate_1}
\begin{split}
\dd  \wh{\Temp} 
&=\Big[\Delta_{\h} \wh{\T} -
\reallywidehat{(\wt{v}\cdot\nabla_{\h}) \T} 
- (\overline{v}\cdot\nabla_{\h})\wh{\T} -\Ht(v,\T)
+ f_{\wh{\T}} \Big]\, \dd t \\
&\quad \ \ +\sum_{n\geq 1}\Big[(\psi_{n,\h}\cdot \nabla_{\h}) \T+\wh{\psi^3_n \partial_{3}\T} +\wh{g_{\theta,n}}\Big] \, \dd \beta_t^n,
\end{split}\\
\label{eq:theta_hat_estimate_2}
\Ht(v,\T)&\stackrel{{\rm def}}{=}\fint_{-h}^{0} \Big[
- \TT\, \div_{\h}\wt{v} +\T\, \div_{\h} \wt{v}\,\zeta
\Big]\,\dd \zeta, \\
\label{eq:theta_hat_estimate_3}
f_{\wh{\T}}&\stackrel{{\rm def}}{=}
\wh{f_{\T}}+h^{-1}[ \T(\cdot,0)-\T(\cdot,-h)],\\
\label{eq:theta_hat_estimate_4}
\wh{\T}(0,\cdot)&=\wh{\T}_0\stackrel{{\rm def}}{=}\fint_{-h}^0 \T_0(\cdot,\zeta)\, \zeta\, \dd \zeta,
\end{alignat}
\end{subequations}
where we used that $\psi^1_n,\psi^2_n$ are $\z$-independent by Assumption \ref{ass:global_primitive_strong_strong}, the identity \eqref{eq:wh_partial_3_T_equality} 
and 
$$
\wh{\partial_{3}^2 \T}=h^{-1}[ \T(\cdot,0)-\T(\cdot,-h)] \text{ on }\Tor^2, \quad \text{ since } \quad \partial_{3} \T(\cdot,-h)=0 \text{ on }\Tor^2.
$$

\subsection{Preparation of the proof of Lemma \ref{l:main_estimate}}
\label{ss:preparation}
%%%%%
In this subsection, we prepare the proof of Lemma \ref{l:main_estimate}. 
To this end, let $(f_v,f_{\T},g_{v},g_{\T})$ be as in \eqref{eq:inhomogeneity_v_T_4}-\eqref{eq:inhomogeneity_v_T_6}. 
As remarked below \eqref{eq:inhomogeneity_v_T}, such terms can be estimated by using Lemma \ref{l:basic_estimates}.
More precisely, let 
\begin{equation}
\begin{aligned}
\label{eq:def_N_v_T}
\low_t
\stackrel{{\rm def}}{=}& \big[1+\|v(t)\|_{L^2}^2+\|\T(t)\|_{L^4}^4\big]\\
&\cdot\big[1+
 \big( \|f_{v}(t)\|^2_{L^2}+ \|f_{\T}(t)\|^2_{L^2}+\|g_{v}(t)\|_{H^1(\ell^2)}^2
+\|g_{\T}(t)\|_{H^1(\ell^2)}^2\big)\\
&+ \big( \|v(t)\|^2_{H^1} +\|\T(t)\|_{L^{6}}^8+
\big\|(1+|\T(t)|)|\nabla \T(t)|\big\|_{L^2}^2\big)\big].
\end{aligned}
\end{equation}
By \eqref{eq:def_y}, Assumptions \ref{ass:well_posedness_primitive_double_strong}\eqref{it:well_posedness_primitive_phi_psi_smoothness}-\eqref{it:nonlinearities_strong_strong} and \eqref{ass:global_primitive_strong_strong} (see also \eqref{eq:boundedness_ktwon}-\eqref{eq:boundedness_tp} below), there exists $K\geq 1$ independent of $(v_0,\T_0)$ such that, a.s.\ for all $t\in [0,\tau)$,
$$
\low_t\leq K\big(1+\|v(t)\|_{L^2}^2+\|\T(t)\|_{L^4}^4\big)
\big(1+ (\y(t))^2+ \|v(t)\|_{H^1}^2+\|\T(t)\|_{L^{6}}^8+\big\|(1+|\T(t)|^2)|\nabla \T(t)|\big\|_{L^2}^2\big).
$$
Hence, by the Chebyshev inequality, Lemma \ref{l:basic_estimates} and \eqref{eq:P_T_eta_tail_estimate} with $\eta=1/2$, we have
\begin{equation}
\label{eq:tail_probability_N_v_T}
\begin{aligned}
\P\Big(\int_0^{\tau\wedge T}L_s\,\dd s\geq \g\Big)
 &\lesssim_T\frac{1+\E\|\y\|_{L^2(0,T)}^2+\E\|v_0\|_{H^1}^4+\E\|\T_0\|_{H^1}^4}{\log(\g)} \ \ \  \text{ for }\g>1,
\end{aligned}
\end{equation}
where the implicit constant on the right hand side of \eqref{eq:tail_probability_N_v_T} is independent of $(v_0,\T_0)$.

We are ready to set up the proof of Lemma \ref{l:main_estimate}. Fix $T\in (0,\infty)$ and let $(\tau_j)_{j\geq 1}$ be as in \eqref{eq:def_tau_j}. Recall that $\lim_{j\to \infty}\tau_j=\tau\wedge T$ a.s.\ and \eqref{eq:T_v_stopped_are_good} holds. 
Let $(X_t,Y_t)$ and $L_t$ be as in Lemma \ref{l:main_estimate} and \eqref{eq:def_N_v_T}, respectively. Finally fix two stopping times $(\eta,\xi)$ such that $0\leq \eta\leq \xi\leq \tau_j$ a.s.\ for some $j\geq 1$. 
The aim of this section is to prove the existence of $c_0\geq 1$ independent of $(j,\eta,\xi,v_0,\T_0)$ such that
\begin{align}
\nonumber
\E\Big[\sup_{t\in [\eta,\xi]} (X_t+ \|\wh{\T}(t)\|_{H^1(\Tor^2)}^2)\Big]
&+ \E\int_{\eta}^{\xi}(Y_s+\|\wh{\T}(s)\|_{H^2(\Tor^2)}^2)\,\dd s\\
\label{eq:main_estimate_gronwall}
&\leq c_0\big(1+\E[X_{\eta}]+\E\|\T(\eta)\|_{L^4}^4+ \E\|\wh{\T}(\eta)\|_{H^1(\Tor^2)}^2\big)\\
\nonumber
&  \ +c_0\E\int_{\eta}^{\xi} L_s (1+X_s+\|\T(s)\|_{L^4}^4) \,\dd s.
\end{align}
%+\|\wh{\T}(s)\|_{H^1(\Tor^2)}^2
The presence of $\|\T\|_{L^4}^4$ on the right hand side of \eqref{eq:main_estimate_gronwall} will 
%prove 
turn out to be convenient later, cf.\ the last comments in Subsection \ref{sss:proof_estimate_mixed_v_T}. However, these terms do not create additional problems as they have been already estimated in Lemma \ref{l:basic_estimates}.

Next, we show the sufficiency of \eqref{eq:main_estimate_gronwall} for Lemma \ref{l:main_estimate} to hold. 
Let $X_t' \stackrel{{\rm def}}{=} X_t+ \|\T(t)\|_{L^4}^4$. By adding the estimates \eqref{eq:grownall_L_6} and \eqref{eq:main_estimate_gronwall}, we can apply the stochastic Gronwall lemma of \cite[Lemma A.1]{AV_variational} with $(X,Y,f,c_0)$ replaced by $(X',Y,L,4c_0)$. Since $X_t\leq X_t'$, the previously mentioned Grownall lemma implies, for all $R,\g>1$,
\begin{align*}
\P\Big(\sup_{t\in[0,\tau\wedge \tau_j ]} X_s
+ \int_0^{\tau\wedge \tau_j} Y_s\,\dd s\geq \g\Big)
\lesssim_{c_0}\frac{e^{16c_0 R}}{\g} (1+\E [X_0]+\E\|\T_0\|_{L^4}^4)
+ \P \Big( \int_{0}^{\tau\wedge T} L_s\, \dd  s\geq  \frac{R}{4 c_0}\Big)&\\
\lesssim_{c_0} \Big( \frac{e^{16 c_0 R}}{\g}+ \frac{C}{\log R}\Big)(1+\E\|\y\|_{L^2(0,T)}^2+\E\|v_0\|_{H^1}^4+\E\|\T_0\|_{H^1}^4),&
\end{align*}
where in the last step we used \eqref{eq:tail_probability_N_v_T} and $\E[X_0]+\E\|\T_0\|_{L^4}^4\lesssim1+ \E\|v_0\|_{H^1}^4+\E\|\T_0\|_{H^1}^4$. Choosing $R=\frac{1}{16c_0}\log(\frac{\g}{\log(\g)})$ for $\g>1$ large and letting $j\to \infty$, one can readily check that the above estimate yields \eqref{eq:tail_probability_main_estimate}.

The remaining part of this section is devoted to the proof of \eqref{eq:main_estimate_gronwall} where $(\eta,\xi)$ are two stopping times such that $0\leq \eta\leq \xi\leq \tau_j$ a.s.\ for some $j\geq 1$ and $T\in (0,\infty)$ is also fixed. 
The proof of \eqref{eq:main_estimate_gronwall} requires a long preparation which will be the scope of Subsections \ref{ss:estimate_overline_v}-\ref{ss:estimate_TT}. The proof of \eqref{eq:main_estimate_gronwall} is postponed to Subsection \ref{ss:proof_lemma_main_estimate}. 
%%%%%
Before starting into the proof of the estimates, we collect some facts which will be used frequently. 
Firstly, by Assumption \ref{ass:well_posedness_primitive_double_strong}\eqref{it:well_posedness_primitive_ktwo_smoothness_strong_strong} and \ref{ass:global_primitive_strong_strong} as well as the Sobolev embedding $H^{1,2+\delta}(\Tor^d;\ell^2)\embed L^{\infty}(\Tor^d;\ell^2)$ we have, a.s.\ for all $t\in \R_+$,
\begin{align}
\label{eq:boundedness_ktwon}
\|\partial_j^{k} \ktwo(t,\cdot)\|_{L^{\infty}(\Tor^2;\ell^2)}
&\lesssim_{M,\delta} 1 \ \ \text{ for all $j\in \{1,2\}$ and $k\in \{0,1\}$},\\
\label{eq:boundedness_tp}
\|\tp^j(t,\cdot)\|_{L^{\infty}(\Tor^2)}&\lesssim_{M,\delta} 1 \ \ \text{ for all $j\in \{1,2\}$}.
\end{align}
Secondly, we recall the following standard interpolation inequalities:
\begin{align}
\label{eq:interpolation_inequality_L4}
\|f\|_{L^4(\Tor^2)}&\lesssim \|f\|_{L^2(\Tor^2)}^{1/2}\|f\|_{H^1(\Tor^2)}^{1/2},& \text{ for }&f\in H^1(\Tor^2),\\
\label{eq:interpolation_inequality_L3} 
\|f\|_{L^3(\Dom)}&\lesssim \|f\|_{L^2(\Dom)}^{1/2}\|f\|_{H^1(\Dom)}^{1/2}, &\text{ for }&f\in H^1(\Dom).
\end{align}
%%%%%%%

To prove \eqref{eq:main_estimate_gronwall} we also use (small) parameters $\varepsilon_{i},\delta_i\in  (0,\infty)$,  where $i\in \{1,\dots,9\}$, which will be used to absorb energy terms on the left-hand side of the corresponding estimate.
The parameter $\delta_i$ is chosen in the $i$-th subsection among Subsections \ref{ss:estimate_overline_v}-\ref{ss:estimate_TT} and the $\varepsilon_i$'s are chosen in Subsection \ref{ss:proof_lemma_main_estimate}.
Finally, to economize the notation, we do not display the dependence of the constants on $T$.

\subsection{Estimate for $\sup_t\| \overline{v}\|_{H^1_x}$ and $\| \overline{v}\|_{L^2_t H^2_x}$}
\label{ss:estimate_overline_v}
In this subsection, we prove that 
\begin{equation}
\label{eq:estimate_v_overline}
\begin{aligned}
&\E \Big[\sup_{t\in [\eta,\xi]} \|\overline{v}(t)\|^2_{H^1(\Tor^2)}\Big]
+ 
\E \int_{\eta}^{\xi} \|\overline{v}\|_{H^2(\Tor^2)}^2\,\dd s 
\leq C_{1} 
\Big( 1+\E  \|\overline{v}(\eta)\|_{H^1(\Tor^2)}^2\\
&\qquad +  \E\int_{\eta}^{\xi}L_s \|\overline{v}\|_{H^1(\Tor^2)}^2\,\dd s+
\E\int_{\eta}^{\xi} \big\| |\wt{v}||\nabla \wt{v}| \big\|_{L^2}^2\,\dd s+ \E\int_{\eta}^{\xi}\|\nabla \partial_3 v\|_{L^2}^2\,\dd s\\
&\qquad \qquad \qquad \qquad \qquad \qquad \qquad   
+ \E\int_{\eta}^{\xi}\|\nabla \partial_3 \T\|_{L^2}^2\,\dd s
+\E\int_{\eta}^{\xi}\|\wh{\T}\|_{H^2(\Tor^2)}^2\,\dd s\Big),
\end{aligned}
\end{equation}
where ${C}_1$ is a constant independent of $(j,\eta,\xi,v_0,\T_0)$.

The estimate \eqref{eq:estimate_v_overline} follows as the one in \cite[Lemma 5.3, Step 1]{Primitive1} with minor modifications.  
The only additional term comes from the presence of $\ktwon \nabla_{\h}\wh{\T}$ in the stochastic part of \eqref{eq:primitive_bar_1}. To estimate the latter, note that (recall that $(M,\delta)$ are as in Assumption \ref{ass:well_posedness_primitive_double_strong}),
\begin{align*}
\E\int_{\eta}^{\xi} \|\Lt \wh{\T}\|_{L^2(\Tor^2)}^2\,\dd s 
&\lesssim_{M,\delta} \E\int_{\eta}^{\xi}\|\wh{\T}\|_{H^2(\Tor^2)}^2\,\dd s,\\
\E\int_{\eta}^{\xi} 
\|\reallywidehat{\tp^3\partial_3 \nabla_{\h}\T}\|_{L^2(\Tor^2)}\,\dd s
&\stackrel{\eqref{eq:boundedness_tp}}{\lesssim_{M,\delta}} \E\int_{\eta}^{\xi}\|\nabla \partial_3 \T\|_{L^2}^2\,\dd s,\\
\E\int_{\eta}^{\xi} \|(\ktwon \nabla_{\h}\wh{\T})_{n\geq 1}\|_{H^1(\Tor^2;\ell^2)}^2\,\dd s 
&\stackrel{\eqref{eq:boundedness_ktwon}}{\lesssim_{M,\delta}} \E\int_{\eta}^{\xi}\|\wh{\T}\|_{H^2(\Tor^2)}^2\,\dd s .
\end{align*}
Using the above, the estimate \eqref{eq:estimate_v_overline} follows %from 
as the one in \cite[Lemma 5.3, Step 1]{Primitive1} adding also the term $\E\int_{\eta}^{\xi}(\|\wh{\T}\|_{H^2(\Tor^2)}^2+\|\nabla \partial_3 \T\|_{L^2}^2)\,\dd s$ on the right hand side of the corresponding estimate.

\subsection{Estimate for $\sup_{t}\|\wh{\T} \|_{H^1_x}$ and $\|\wh{\T}\|_{L^2_t H^2_x}$} 
\label{ss:estimate_wh_T}
The aim of this subsection is to prove the following estimate:
\begin{align}
\nonumber
&\E\Big[\sup_{t\in [\eta,\xi]} \|\wh{\T}(t)\|_{H^1(\Tor^2)}^2 \Big]
+\E \int_{\eta}^{\xi} \|\wh{\T}\|_{H^2(\Tor^2)}^2\, \dd s\\
\label{eq:estimate_T_wh}
&\leq C_{2}\Big(1+\E\|\wh{\T}(\eta)\|_{H^1(\Tor^2)}^2 \,\dd s+\E\int_{\eta}^{\xi}L_s(1+\|\wh{\T}\|_{H^1(\Tor^2)}^2)\,\dd s
+\E \int_{\eta}^{\xi} \Big\||\wt{v}||\nabla \T |\Big\|_{L^2}^2 \,\dd s \\
&\qquad\ \
\nonumber
+ \E \int_{\eta}^{\xi} \Big\||\nabla \wt{v}|| \T |\Big\|_{L^2}^2 \,\dd s+ \E \int_{\eta}^{\xi} \Big\||\nabla\wt{v}||\TT |\Big\|_{L^2}^2 \,\dd s 
+ \E\int_{\eta}^{\xi} \|\partial_{3} \T\|_{H^1}^2\,\dd s
\Big),
\end{align}
where ${C}_{2}$ is a constant independent of $(j,\eta,\xi,v_0,\T_0)$ and $L_s$ is as in \eqref{eq:def_N_v_T}.

As in Subsection \ref{ss:estimate_overline_v}, the proof of \eqref{eq:estimate_T_wh} follows the line of \cite[Lemma 5.3, Step 1]{Primitive1}. Recall that $\wh{\T}$ satisfies \eqref{eq:theta_hat_estimate}. Next, let us denote by  $
\mathcal{SMR}^{\bullet}_{2}(0,T)$ the set of couples of operators having maximal $L^2$-regularity on a time interval $(0,T)$ on given spaces $(X_0,X_1)$, see Lemma \ref{l:smr} and \cite[Section 3]{AV19_QSEE_1} for the notation and examples.
By repeating the arguments in Lemma \ref{l:smr}, one sees that 
$
(-\Delta_{\h}, (\psi_{n,\h}\cdot\nabla)_{n\geq 1})\in \mathcal{SMR}^{\bullet}_{2}(0,T)$
with $X_0=L^2(\Tor^2)$ and $X_1=H^{2}(\Tor^2)$ (see also \cite{AV_torus} for the $L^p$-setting). Thus, by \cite[Proposition 3.10]{AV19_QSEE_1} and \eqref{eq:theta_hat_estimate}, there exists $\wh{C}$ independent of $(j,\eta,\xi,v_0,\T_0)$ such that
\begin{equation}
\label{eq:smr_estimate_wh_T}
\E\Big[\sup_{t\in[\eta,\xi]}\|\wh{\T}(t)\|^2_{H^1(\Tor^2)}\Big]
+ \E \int_{\eta}^{\xi} \|\wh{\T}(t)\|^2_{H^2(\Tor^2)}\,\dd s 
\leq \wh{C}\Big[\E\|\wh{\T}(\eta)\|^2_{H^1(\Tor^2)}+  \sum_{1\leq j\leq 5} \wh{I}_j\Big]
\end{equation}
where 
\begin{align*}
\wh{I}_1
&\stackrel{{\rm def}}{=}\E\int_{\eta}^{\xi}\|\reallywidehat{(\wt{v}\cdot\nabla_{\h})\T}\|_{L^2(\Tor^2)}^2\,\dd s, &
\wh{I}_2
&\stackrel{{\rm def}}{=}\E\int_{\eta}^{\xi}\|(\overline{v}\cdot\nabla_{\h})\wh{\T}\|_{L^2(\Tor^2)}^2\,\dd s,\\
\wh{I}_3 
&\stackrel{{\rm def}}{=}\E\int_{\eta}^{\xi}\|\Ht(v,\T)\|_{L^2(\Tor^2)}^2\,\dd s,&
\wh{I}_4
&\stackrel{{\rm def}}{=}
\E\int_{\eta}^{\xi}\big(\|f_{\wh{\T}}\|_{L^2(\Tor^2)}^2 + \|\wh{g_{\T}}\|_{H^1(\Tor^2;\ell^2)}^2\big)\,\dd s,\\
\wh{I}_5
&\stackrel{{\rm def}}{=}
\E\int_{\eta}^{\xi} \|(\wh{\psi^3_n \partial_{3} \T})_{n\geq 1}\|^2_{H^1(\Tor^2)}\,\dd s. & &
\end{align*}
Let us estimate each term separately. Note that 
$$
\wh{I}_1+\wh{I}_3 \lesssim_h \E\int_{\eta}^{\xi}\Big(\Big\||\wt{v}||\nabla\T|\Big\|_{L^2}^2
+\Big\||\T||\nabla \wt{v}|\Big\|_{L^2}^2+\Big\||\TT||\nabla \wt{v}|\Big\|_{L^2}^2\Big) \,\dd s.
$$
Moreover, applying \eqref{eq:interpolation_inequality_L4} twice,
\begin{align*}
\wh{I}_2
&\leq \E\int_{\eta}^{\xi}\|\overline{v}\|_{L^4(\Tor^2)}^2 \|\nabla \wh{\T}\|_{L^4(\Tor^2)}^2\,\dd s\\
&\lesssim 
\E\int_{\eta}^{\xi}\|\overline{v}\|_{L^2(\Tor^2)}\|\overline{v}\|_{H^1(\Tor^2)}
\|\wh{\T}\|_{H^1(\Tor^2)} \| \wh{\T}\|_{H^2(\Tor^2)}\,\dd s\\
&\leq \wh{C}_0 \E\int_{\eta}^{\xi}\|v\|_{L^2}^2 \|v\|_{H^1}^2\|\wh{\T}\|_{H^1(\Tor^2)}^2\, \dd  s +\frac{1}{2\wh{C}} \E\int_{\eta}^{\xi}\| \wh{\T}\|_{H^2(\Tor^2)}^2\,\dd s,
\end{align*}
where $\wh{C}$ is as in \eqref{eq:smr_estimate_wh_T}, and $\wt{C}_0$ is a constant independent of $(v_0,\T_0,\eta,\xi,j)$.  Finally, from \eqref{eq:boundedness_trace} and Remark \ref{r:boundedness}, we have 
 \begin{equation*}
\wh{I}_4 \lesssim \E\int_{\eta}^{\xi} L_s\,\dd s \quad \text{ and }\quad
 \wh{I}_5\lesssim_{M}\E\int_{\eta}^{\xi} L_s\,\dd s+ \E\int_{\eta}^{\xi}\|\nabla \partial_{3} \T\|_{L^2}\,\dd s.
 \end{equation*}
Putting together the previous estimate, one sees that there exists a constant $C_2$ independent of $(v_0,\T_0,\eta,\xi,j)$ for which \eqref{eq:estimate_T_wh} holds.
%%%%
%%%%

\subsection{Estimate for $\sup_t \|\partial_{3} v\|_{L^2_x}$ and $\|\partial_{3} v\|_{L^2_t H^1_x}$}
\label{ss:estimate_partial_3_v}
The aim of this subsection is to prove the following estimate: For all $\varepsilon_3\in (0,\infty)$,
\begin{equation}
\label{eq:estimate_v_3}
\begin{aligned}
&\E\Big[\sup_{t\in [\eta,\xi]} \|\partial_{3} v (t)\|_{L^2}^2\Big]
+ \E\int_{\eta}^{\xi} \| \nabla \partial_3 v(t)\|_{L^2}^2\,\dd s \\
&\qquad  \leq C_{3} \Big( 1+\E \|\partial_3v(\eta)\|_{L^2}^2+\E\int_{\eta}^{\xi} \Big\| |\wt{v}||\nabla \wt{v}| \Big\|_{L^2}^2\,\dd  s \Big)\\
&\qquad\qquad 
+ C_{3,\varepsilon_3} \E\int_{\eta}^{\xi} L_s (1+\|\partial_3v\|_{L^2}^2)\,\dd s
+\varepsilon_3 \E\int_{\eta}^{\xi} \| \partial_3 \nabla \T \|_{L^2}^2\,\dd s,
\end{aligned}
\end{equation}
where $C_3,C_{3,\varepsilon_3}$  are constants independent of $(j,\eta,\xi,v_0,\T_0)$ and $C_3$ is also independent of $\varepsilon_3$. Finally, $L_s$ is as in \eqref{eq:def_N_v_T}.

As before, here we follow the arguments in Step 2 of \cite[Lemma 5.3]{Primitive1} with minor modifications. For notational convenience, as in the previously mentioned reference, we set $v_{3}\stackrel{{\rm def}}{=}\partial_{3} v$. The estimate \eqref{eq:estimate_v_3} follows almost verbatim as in \cite[Lemma 5.3, Step 2]{Primitive1} up to considering the additional term coming from $\Lt \T \,\dd t$ and $\sum_{n\geq 1}\ktwon \int_{-h}^{\cdot} \nabla_{\h} \T(\cdot,\zeta)\,\dd \zeta\dd \beta_t^n$ in \eqref{eq:primitive_v_proof_global_1} in the It\^o formula for $v\mapsto \| \partial_3 v\|_{L^2}^2$. 
Let us begin by noticing that, the $\ktwon$-contribution does not provide any additional problem as
(recall that $\ktwon$ is $x_3$-independent by Assumption \ref{ass:global_primitive_strong_strong}) 
$$
\sum_{n\geq 1}\E\int_{\eta}^{\xi}\int_{\Dom} \Big|\partial_3 [\ktwon \int_{-h}^{\cdot} \nabla_{\h} \T(\cdot,\zeta)\,\dd \zeta]\Big|^2\,\dd x\dd s\lesssim \E\int_{\eta}^{\xi}\|\nabla \T\|_{L^2}^2\,\dd s \lesssim \E\int_{\eta}^{\xi} L_s\,\dd s.
$$ 
To estimate the contribution of $\Lt \T\,\dd t$, note that, in the It\^o formula for  $v\mapsto \| \partial_3 v\|_{L^2}^2$ it gives rise to the term $\E\int_{\eta}^{\xi} R\,\dd s$ where 
$$
R \stackrel{{\rm def}}{=}
\int_{\Dom} \p[\Lt\T] \partial_{3} v_{3}\,\dd x.
$$
Recall that $\partial_3 \p f=\partial_3 f $ by \eqref{eq:Helmholtz_hydrostatic}. Integrating by parts and using \eqref{eq:primitive_v_proof_global_bc_1}, we have
$$
R= -\underbrace{\int_{\Dom} \nabla_{\h}[(\tp_{\h}\cdot\nabla_{\h})\T] \cdot v_3\,\dd x}_{R_1\stackrel{{\rm def}}{=}}
-\underbrace{\int_{\Dom} \tp^3 \nabla_{\h} \partial_3\T\cdot  v_3\,\dd x}_{R_2\stackrel{{\rm def}}{=}}.
$$
Note that, integrating by parts in the horizontal variables, for all $\varepsilon_0>0$,
\begin{align*}
|R_1|
= \Big| \int_{\Dom} \big[(\tp_{\h}\cdot\nabla_{\h}) \T \big] \div_{\h}v_3  \,\dd x \Big|
\stackrel{\eqref{eq:boundedness_tp}}{\leq}\varepsilon_0 \|\nabla v_3\|_{L^2}^2 + C_{\varepsilon_0} \|\nabla \T\|_{L^2}^2.
\end{align*}
To estimate $R_2$ note that $\tp^3 \in H^{1,2+\delta}(\Tor^2;L^{2}(-h,0))\embed  L^{\infty}(\Tor^2;L^2(-h,0))$ 
uniformly in $ \R_+\times \O$ by Assumption \ref{ass:well_posedness_primitive_double_strong}\eqref{it:well_posedness_primitive_kone_smoothness_strong_strong}.  
Since $H^{r}(\Dom)\embed L^2(\Tor^2;H^{r}(-h,0))\embed L^2(\Tor^2;L^{\infty}(-h,0))$ for all $r\in (\frac{1}{2},1)$, by interpolation, one sees that
$$
|R_2|\leq \varepsilon_3 \|\nabla\partial_3 \T\|_{L^2}^2+ \delta_3 \|\nabla v_3\|_{L^2}^2+ C_{\delta_3,\varepsilon_3}\|v_3\|_{L^2}^2.
$$
By using the above estimates for $R$ and choosing $\delta_3$ small enough (independently of $(j,\eta,\xi,v_0,\T_0,\varepsilon_3)$), one can check that the arguments in Step 2 of \cite[Lemma 5.3]{Primitive1} yield the estimate \eqref{eq:estimate_v_3}.

\subsection{Estimate for $\sup_t\|\partial_{3} \T \|_{L^2_x}$ and $\|\partial_{3}\T\|_{L^2_t H^1_x}$}
\label{ss:estimate_partial_T_3}
In this subsection, we prove that:
\begin{equation}
\label{eq:estimate_T_3}
\begin{aligned}
\E\Big[\sup_{t\in [\eta,\xi]}\|\partial_{3} \T(t)\|_{L^2}^2\Big]
&+\E \int_{\eta}^{\xi} \| \partial_{3} \T\|_{H^1}^2\,\dd s 
\leq C_4(1+\E\|\partial_{3}\T(\eta)\|^{2}_{L^2}+\E\|\T(\eta)\|_{L^4}^4)\\
&+C_4\Big(
\E\int_{\eta}^{\xi} L_s (1+\|\partial_{3}\T\|_{L^2}^2)\,\dd s+ \E \int_{\eta}^{\xi}\Big\||\T||\nabla \wt{v}|\Big\|_{L^2}^2\,\dd s\Big),
\end{aligned}
\end{equation}
where $C_T^{(4)}$  are constants independent of $(j,\eta,\xi,v_0,\T_0)$.

Here the idea is to apply the It\^{o} formula to (see the proof of \cite[Proposition 6.8]{Primitive1} for a similar situation)
$$
\T\mapsto \Fun(\T) \stackrel{{\rm def}}{=}\|\partial_{3} \T\|_{L^2}^2+ \alpha\|\T(\cdot,0)\|_{L^2(\Tor^2)}^2.$$ 
 For notational convenience, we set $\T_{3}\stackrel{{\rm def}}{=}\partial_{3} \T$ and $\T^ {\eta,\xi}_{3}\stackrel{{\rm def}}{=}\T_{3}((\cdot\vee \eta)\wedge \xi)$. Combining a standard approximation argument (cf.\ the proof of \cite[Proposition 6.8]{Primitive1}), the It\^{o} formula, the boundary conditions \eqref{eq:primitive_v_proof_global_bc_2} and integrating by parts, one can check that, a.s.\ for all $t\in [0,T]$,
\begin{align}
\label{eq:Ito_T_3}
\|\T^ {\eta,\xi}_{3}(t)\|_{L^2}^2
&+\alpha \|\T^{\eta,\xi}(t,\cdot,0)\|_{L^2(\Tor^2)}^2
=\|\T_3(\eta)\|_{L^2}^2+\alpha \|\T(\eta,\cdot,0)\|_{L^2(\Tor^2)}^2\\
\nonumber
&+ 2 \int_{0}^t \one_{[\eta,\xi]} E(s)\,\dd s
+ \sum_{1\leq j\leq 3} \int_0^t \one_{[\eta,\xi]}  I_j(s)\,\dd s + M(t),
\end{align}
where $E\stackrel{{\rm def}}{=}- \int_{\Dom}\Delta \T\partial_3 \T_3\,\dd x$ gives the energy contribution 
and 
\begin{align*}
I_1&\stackrel{{\rm def}}{=} 2\int_{\Dom}   f_{\T}\partial_{3} \T_{3}\,\dd x, \qquad
I_2\stackrel{{\rm def}}{=} -
2\int_{\Dom}  [(v\cdot\nabla_{\h})\T+ w(v)\partial_{3} \T]\partial_{3} \T_{3}\,\dd x,\\
I_3&\stackrel{{\rm def}}{=} \sum_{n\geq 1}\Big(
\int_{\Dom} \Big|\partial_{3} [(\psi_n \cdot\nabla)\T]+\partial_{3} g_{\T,n}\Big|^2\,\dd x\\
&\qquad + \int_{\Tor^2} \Big|\partial_{3} [(\psi_n(\cdot,0) \cdot\nabla)\T(\cdot,0)]+\partial_{3} g_{\T,n}(\cdot,0)\Big|^2\,\dd x_{\h}\Big),\\
M(t)&\stackrel{{\rm def}}{=} 2 \sum_{n\geq 1} \int_{0}^t\one_{[\eta,\xi]}\Big( \int_{\Dom}
\big(\partial_{3}[(\psi_n \cdot\nabla)\T] + \partial_{3} g_{\T,n}\big)\T_{3}\,\dd x \\
&\qquad  +\int_{\Dom}
\big(\partial_{3}[(\psi_n(\cdot,0) \cdot\nabla)\T(\cdot,0)] + \partial_{3} g_{\T,n}(\cdot,0)\big)\T_{3}(\cdot,0)\,\dd x \Big)
\, \dd \beta^n_s.
\end{align*}

One can readily check that,  for all $\delta_4>0$ and a.e.\ on $[\eta,\xi]\times \O$, 
\begin{equation}
\label{eq:estimate_I_1_T3}
|I_1|\leq \delta_4 \int_{\eta}^{\xi} \|\nabla \T_{3}\|_{L^2}^2\,\dd s
+C_{\delta_4}
\int_{\eta}^{\xi}  \|f_{\T}\|_{L^2}^2\,\dd s.
\end{equation}
In the following, we need a slight improvement of \eqref{eq:boundedness_trace}, in particular to bound the boundary terms in \eqref{eq:Ito_T_3}. To this end, note that, the 1d Sobolev embeddings ensures that $|f(x_{\h},0)|\lesssim \|f(x_{\h},\cdot)\|_{H^{1/2+r}(-h,0)}$ for all $x_{\h}\in\Tor^2$ and for all $r>0$, with implicit constant independent of $x_{\h}$. Hence, by integrating over $x_{\h}\in \Tor^2$, we have
\begin{equation}
\label{eq:boundedness_trace_sharp}
\|f(\cdot,0)\|_{L^2(\Tor^2)}\lesssim \Big\|x_{\h}\mapsto \|f(x_{\h},\cdot)\|_{H^{1/2+r}(-h,0)}\Big\|_{L^2(\Tor^2)}.
\end{equation}
In particular, the second term on the left-hand side of \eqref{eq:Ito_T_3} is lower order compared to $\|\T^{\eta,\xi}_3(t)\|_{L^2}$ and we do not need to estimate it further. The same also applies to the second term on the right-hand side of \eqref{eq:Ito_T_3} for which we can use that \eqref{eq:boundedness_trace_sharp} implies $\|\T(\eta,\cdot,0)\|_{L^2(\Tor^2)}\lesssim \|\T(\eta)\|_{L^2}+ \|\partial_3\T(\eta)\|_{L^2}$.

The estimates of the remaining terms are worked out in the following subsections. 
The proof of \eqref{eq:estimate_T_3} is given in Subsection \ref{sss:estimate_T_3_conclusion} below. In the following $\varepsilon_4,\delta_4\in (0,\infty)$ are positive parameters which will be chosen in Subsections \ref{ss:proof_lemma_main_estimate} and 
\ref{sss:estimate_T_3_conclusion}, respectively.

\subsubsection{Estimate of $E$} Since $\T\in \Hr^2$, by standard approximation arguments we may assume that $\T\in C^3(\Dom)$ and satisfies \eqref{eq:primitive_v_proof_global_bc_2}. Note that, integrating by parts, 
\begin{align*}
E &=\alpha \int_{\Tor^2} (\Delta \T(\cdot,0))\T(\cdot,0)\,\dd x_{\h} + \int_{\Dom} \Delta \T_3 \T_3\,\dd x\\
&=\alpha \int_{\Tor^2} (\Delta \T(\cdot,0))\T(\cdot,0)\,\dd x_{\h} + \int_{\Tor^2} \partial_3 \T_3(\cdot,0) \T_3(\cdot,0)\,\dd x_{\h} - \int_{\Dom} | \nabla \T_3|^2\,\dd x\\
&\stackrel{\eqref{eq:primitive_v_proof_global_bc_2}}{=}\alpha \underbrace{\int_{\Tor^2} (\Delta \T(\cdot,0))\T(\cdot,0)\,\dd x_{\h} }_{e_0\stackrel{{\rm def}}{=}} -\alpha^3 \underbrace{\int_{\Tor^2} |\T(\cdot,0)|^2 \,\dd x_{\h}}_{e_1\stackrel{{\rm def}}{=}} - \int_{\Dom} | \nabla \T_3|^2\,\dd x.
\end{align*}
The last term on the right-hand side of the previous equality gives rise to the second term on the left-hand side of \eqref{eq:estimate_T_3}. To conclude, we show that $(e_0,e_1)$  are of lower-order compared to such term, i.e.\ for all $\varepsilon>0$
\begin{equation}
\label{eq:e0_e1_claim}
|e_0|+|e_1|\leq \varepsilon \|\nabla \T_3\|_{L^2}^2+C_{\varepsilon}\|\T_3\|_{L^2}^2.
\end{equation}
Note that \eqref{eq:boundedness_trace} already implies that $e_1$ is of lower order. To estimate $e_1$, note that, by \eqref{eq:primitive_v_proof_global_bc_2} and integrating by parts,
\begin{equation}
\label{eq:estimate_e0_boundary}
e_0= -\int_{\Tor^2} |\nabla_{\h} \T(\cdot,0)|^2\,\dd x_{\h}-\alpha^2 \int_{\Tor^2} |\T(\cdot,0)|^2\,\dd x_{\h}.
\end{equation}
Due to \eqref{eq:boundedness_trace}, it is clear that the second term on the right-hand side of \eqref{eq:estimate_e0_boundary} is of lower order. 
The same also holds for the first term as one can readily check by applying \eqref{eq:boundedness_trace_sharp} and a standard interpolation argument.

\subsubsection{Estimate of $I_2$}
For notational convenience, as above, we set $u\stackrel{{\rm def}}{=}(v,w(v))$. Note that, integrating by parts and using \eqref{eq:primitive_v_proof_global_bc_2}, we have, a.e.\ on $\O\times [\eta,\xi]$, 
\begin{align*}
I_2
&=  \int_{\Tor^2} (v(\cdot,0)\cdot \nabla_{\h}) \T(\cdot,0) \T(\cdot,0)\,\dd x_{\h}-
\int_{\Dom} [(u_{3}\cdot\nabla) \T] \T_{3}\,\dd x
-
\int_{\Dom} [(u\cdot\nabla) \T_{3}] \T_{3}\,\dd x\\
&=\underbrace{\int_{\Tor^2} (v(\cdot,0)\cdot \nabla_{\h}) \T(\cdot,0) \T(\cdot,0)\,\dd x_{\h}}_{b_{0}\stackrel{\rm{def}}{=}}-
\underbrace{\int_{\Dom} [(u_{3}\cdot\nabla) \T] \T_{3}\,\dd x}_{I_{2}'\stackrel{\rm{def}}{=}},
\end{align*}
where the last equality follows from Lemma \ref{l:cancellation} and an approximation argument. Next we rewrite $I_{2}'$. To this end, note that $u_{3}=(v_{3},-\div_{\h} v)$. Hence, using  an integration by parts and $\div\,u_{3}=0$, we have, a.e.\ on $[\eta,\xi]\times \O$, 
\begin{align*}
I_{2}'
&= -\int_{\Tor^2}\div_{\h}v (\cdot,0) \T(\cdot,0)\T_3(\cdot,0)\,\dd x_{\h}  
-\int_{\Dom} \T [(u_{3}\cdot \nabla) \T_{3} ]\,\dd x\\
&\stackrel{\eqref{eq:primitive_v_proof_global_bc_2}}{=}\alpha\underbrace{  \int_{\Tor^2}\div_{\h}v (\cdot,0) |\T(\cdot,0)|^2\,\dd x_{\h} }_{b_1\stackrel{\rm{def}}{=}}
- \underbrace{\int_{\Dom} \T [(u_{3}\cdot \nabla) \T_{3} ]\,\dd x}_{I_{2}''\stackrel{\rm{def}}{=}}.
\end{align*}
%%%%
Finally, since $\div_{\h} \wt{v}=\div_{\h} v$ and $\wt{v}_3=v_3$, we have, a.e.\ on $[\eta,\xi]\times \O$,
\begin{align*}
|I_2''|=
- \sum_{1\leq j\leq 2}\int_{\Dom} \wt{v}_{3}^j\, \T \,(\partial_j \T_{3} )\,\dd x
+\int_{\Dom} (\div_{\h} \wt{v})\, \T \,(\partial_{3} \T_{3})\,\dd x.
\end{align*}
Therefore, by the Cauchy-Schwartz inequality we have, for all $\delta_4>0$ and a.e.\ on $[\eta,\xi]\times \O$,
\begin{align*}
| I_{2}''|
\lesssim  
\int_{\Dom} |\nabla \wt{v}| |\T| |\nabla \T_{3}|\,\dd x 
\leq \delta_4  \|\nabla \T_{3}\|_{L^2}^2
+C_{\delta_4}
  \big\||\nabla \wt{v}| |\T|\big\|_{L^2}^2 .
\end{align*}
It remains to estimate the boundary terms $(b_0,b_1)$. Recall that  $L$ is as in \eqref{eq:def_N_v_T}. We claim that, a.e.\ on $[\eta,\xi]\times \O$,
\begin{equation}
\label{eq:boundedness_boundary_terms}
|b_0|+ |b_1|\lesssim L.
\end{equation}
We prove the latter fact for $b_0$, for the $b_1$ term the proof is analogue. To this end, note that 
\begin{align*}
|b_0|
&\eqsim \Big|\int_{\Tor^2}( v(\cdot,0)\cdot\nabla_{\h}) [\T(\cdot,0)^2]\,\dd x_{\h}\Big|\leq \|v(\cdot,0)\|_{H^{\frac{1}{2}}(\Tor^2;\R^2)} \|\nabla_{\h}\T(\cdot,0)^2 \|_{H^{-\frac{1}{2}}(\Tor^2;\R^2)}\\
&\lesssim \|v(\cdot,0)\|_{H^{\frac{1}{2}}(\Tor^2;\R^2)} \|\T(\cdot,0)^2 \|_{H^{\frac{1}{2}}(\Tor^2)}
\stackrel{\eqref{eq:boundedness_trace}}{\lesssim} \|v\|_{H^1}\|\T^2\|_{H^1}.
\end{align*}
Since $\|\T^2\|_{H^1}^2\lesssim \|\T\|_{L^4}^4+\||\T||\nabla\T|\|_{L^2}^2$, we have $|b_0|\lesssim L$ as desired. Thus \eqref{eq:boundedness_boundary_terms} is proved.

% is as in Assumption \ref{ass:well_posedness_primitive_double_strong}\eqref{it:well_posedness_primitive_parabolicity_strong_strong}.

\subsubsection{Estimate of $I_3$}
The Cauchy-Schwartz inequality, \eqref{eq:boundedness_trace} and standard interpolation arguments show that, a.e.\ on $[\eta,\xi]\times \O$,
\begin{align*}
|I_3|
&\leq (1+\delta_4)\sum_{n\geq 1} \int_{\Dom} |(\psi_n\cdot\nabla) \T_{3}|^2\,\dd x
+C_{\delta_4} \int_{\Dom} \sum_{n\geq 1} \big(|\nabla \psi_n|^2 |\T_{3}|^2+ |\nabla g_{\T,n}|^2\big)\,\dd x\\
&\leq\ellip (1+\delta_4) \int_{\Dom} |\nabla \T_{3}|^2\,\dd x
+C_{\delta_4}\big(\|( \psi_n)_{n\geq 1}\|_{H^{1,3+\delta}(\ell^2)}\|\T_3\|_{L^{r}}^2+ \|\nabla g_{\T}\|_{H^1(\ell^2)}^2\big),
\end{align*}
where in the last inequality we used Assumption \ref{ass:well_posedness_primitive_double_strong}\eqref{it:well_posedness_primitive_parabolicity_strong_strong} and $r\in (1,6)$ satisfies $\frac{1}{3+\delta}+\frac{1}{r}=\frac{1}{2}$.

Recall that $\|( \psi_n)_{n\geq 1}\|_{H^{1,3+\delta}(\ell^2)}\leq M$, by Assumption \ref{ass:well_posedness_primitive_double_strong}\eqref{it:well_posedness_primitive_phi_psi_smoothness}. 
Since $H^{\theta}(\Dom)\embed L^r(\Dom)$ for some $\theta\in (0,1)$, by standard interpolation theory, we have a.e.\ on $[\eta,\xi]\times \O$
$$
|I_3|
\leq
\ellip (1+2\delta_4) \int_{\Dom} |\nabla \T_{3}|^2\,\dd x
+ C_{\delta_4}\big(\|\T_3\|_{L^{2}}^2+ \|\nabla g_{\T}\|_{H^1(\ell^2)}^2\big).
$$

\subsubsection{Estimate of the martingale $M$ and proof of  \eqref{eq:estimate_T_3}}
\label{sss:estimate_T_3_conclusion}
Taking expectations in \eqref{eq:Ito_T_3} with $t=T$, choosing $\delta_4>0$ sufficiently small (independently of $(j,\eta,\xi,v_0,\T_0)$), and using that $\E[M(T)]=0$, one has 
\begin{align}
\label{eq:intermediate_estimate_T3}
&\E\int_{\eta}^{\xi}\|\nabla \T_{3}\|_{L^2}^2\,\dd s
\leq c_4 \big(1+\E\|\T_3(\eta)\|^2_{L^2}\big)\\
\nonumber
 &\qquad +c_4\Big(
 \E \int_{\eta}^{\xi}  \big\||\nabla \wt{v}| |\T|\big\|_{L^2}^2\,\dd s + \E\int_{\eta}^{\xi}L_s (1+ \|\T_3\|_{L^2}^2)\,\dd s\Big),
\end{align}
where $c_4$ is a constant independent of $(j,\eta,\xi,v_0,\T_0)$.

Arguing as in Step 2 of Lemma \ref{l:basic_estimates}, the Burkholder-Davis-Gundy inequality and 
Assumption \ref{ass:well_posedness_primitive_double_strong}\eqref{it:well_posedness_primitive_phi_psi_smoothness} readily yield, for some $C>0$ independent of $(j,\eta,\xi,v_0,\T_0)$, 
\begin{align*}
\E\Big[\sup_{t\in [\eta,\xi]} |M_t| \Big]
&
\leq \frac{1}{2} \E \Big[\sup_{s\in [\eta,\xi]}\|\T_3(s)\|_{L^2}^2\Big] +  
C \E \int_{\eta}^{\xi}\big(\|\nabla \partial_3 \T\|_{L^2}^2+ \|\T\|_{H^1}^2+\|g_{\T}\|_{H^1(\ell^2)}^2\big)\, \dd s\\
&\stackrel{\eqref{eq:intermediate_estimate_T3}}{\leq}
\frac{1}{2} \E \Big[\sup_{s\in [\eta,\xi]}\|\T_3(s)\|_{L^2}^2\Big]+  C(1+\E\|\T_3(\eta)\|^2_{L^2})\\ 
&\quad +
C \E \int_{\eta}^{\xi}\Big[ \big\||\nabla \wt{v}| |\T|\big\|_{L^2}^2+L_s (1+ \|\T_3\|_{L^2}^2)\Big]\, \dd s.
\end{align*}

Now \eqref{eq:estimate_T_3} follows by taking $\E[\sup_{t\in [\eta,\xi]}|\cdot|]$ in \eqref{eq:Ito_T_3} and using the above estimates.

\subsection{Estimate for $\sup_t\|\wt{v} \|_{L^4_x}$ and $\big\||\wt{v}||\nabla \wt{v}|\big\|_{L^2_tL^2_x}$} 
\label{ss:estimate_wt_v_L_4}
In this subsection we prove the following estimate: For all $\varepsilon_5\in (0,\infty)$,
\begin{align}
\nonumber
&\E\Big[\sup_{t\in [\eta,\xi]}\|\wt{v}(t)\|_{L^4}^4\Big] +\E\int_{\eta}^{\xi}\Big\| |\wt{v}||\nabla \wt{v}|\Big\|_{L^2}^2\,\dd s 
\leq C_{5,\varepsilon_5} \big(1+\E\|\wt{v}(\eta)\|_{L^4}^4\big)\\
\label{eq:estimate_v_wt}
&\qquad\quad
+C_5 \E\int_{\eta}^{\xi} \Big\| |\wt{v}||\nabla_{\h} \TT| \Big\|_{L^2}^2\,\dd s 
+ C_{5,\varepsilon_5} \E\int_{\eta}^{\xi}L_s (1+\|\wt{v}\|_{L^4}^4)\,\dd s\\
& \qquad \quad
\nonumber
+\varepsilon_5 \E\int_{\eta}^{\xi}\big( \|\partial_3 v\|_{H^1}^2 + \|\wh{\T}\|_{H^2}^2\big)\,\dd s,
\end{align}
where $C_5,C_{5,\varepsilon_5}$  are constants independent of $(j,\eta,\xi,v_0,\T_0)$ and $C_5$ is also independent of $\varepsilon_5$. Finally, $L_s$ is as in \eqref{eq:def_N_v_T}.

As in Subsections \ref{ss:estimate_overline_v} and  \ref{ss:estimate_partial_3_v}, here we can follow the proof of  \cite[Lemma 5.3, Step 4]{Primitive1}. More precisely, following \cite{Primitive1} we apply the It\^o formula to $\wt{v}\mapsto \|\wt{v}\|_{L^4}^4$. Comparing \eqref{eq:primitive_bar} with \cite[eq.\ (5.23)]{Primitive1}, we have the following additional terms $(\Lt \T + \Ltb \wh{\T}- \reallywidehat{\pi^3 \partial_3 \nabla_{\h} \T})\,\dd s$ and 
$\sum_{n\geq 1}\opt_n\T\,\dd \beta_t^n$. 
Here, we content ourselves to provide a suitable estimate for the It\^o corrections related to the $\opt_n$-term when applying the It\^{o} formula to $v\mapsto \|\wt{v}\|_{L^4}^4$, i.e., the term
\begin{equation}
\label{eq:correction_wt_additional_term_opt}
\E
\int_{\eta}^{\xi}\int_{\Dom}
\sum_{n\geq 1}  |\wt{v}|^2 |\opt_n\T |^2\,\dd x\dd s.
\end{equation}
The contributions related to the terms in the deterministic part can be estimated similarly, noticing that, by \eqref{eq:inhomogeneity_v_T_1},
$\Lt\T=(\pi_{\h}\cdot\nabla_{\h})\TT+ R_0
$
where 
$$
R_0 
\leq \int_{-h}^{0} |\tp^3(\cdot,\zeta)\partial_3\nabla_{\h}(\cdot,\zeta) \T(\cdot,\zeta)|\,\dd \zeta.
$$ 

To estimate the quantity in \eqref{eq:correction_wt_additional_term_opt}, note that, a.e.\ on $\O\times [0,\tau)$,
\begin{align*}
\sum_{n\geq 1} \int_{\Dom} |\wt{v}|^2 |\opt_n (\T)|^2 \,\dd x
&\stackrel{\eqref{eq:boundedness_ktwon}}{\lesssim_M}
\Big( \int_{\Dom}  |\wt{v}(\cdot,x)|^2  \Big| \int_{-h}^{\cdot}\T(\cdot,x_{\h},\zeta)\,\dd \zeta\Big|^2 \,\dd x
+
 \int_{\Dom} |\wt{v}|^2 |\wh{\T}|^2 \,\dd x\Big)\\
&\leq 
 \int_{\Dom}  |\wt{v}|^2 |\nabla_{\h} \TT|^2 \,\dd x
+
 \int_{\Dom} |\wt{v}|^2 |\nabla_{\h} \wh{\T}|^2 \,\dd x.
\end{align*}
The second term on the right-hand side of the previous can be further estimated as follows:
\begin{align*}
\int_{\Dom} |\wt{v}|^2 |\nabla_{\h} \wh{\T}|^2\,\dd x
&\leq \big\||\wt{v}|^2\big\|_{L^2} \big\||\nabla_{\h}\wh{\T}|^2\big\|_{L^2}\\
&\leq \|\wt{v}\|_{L^4}^2 \|\wh{\T}\|_{W^{1,4}(\Tor^2)}^2\\
&\stackrel{\eqref{eq:interpolation_inequality_L4}}{\lesssim} 
\|\wt{v}\|_{L^4}^2 \|\wh{\T}\|_{H^1(\Tor^2)}\|\wh{\T}\|_{H^2(\Tor^2)}\\
&\lesssim_h \|\wt{v}\|_{L^4}^2\|\T\|_{H^1} \|\wh{\T}\|_{H^2(\Tor^2)}
\leq \varepsilon_5 \|\wh{\T}\|_{H^2(\Tor^2)}^2+ C_{\varepsilon_5}L \|\wt{v}\|_{L^4}^4,
\end{align*}
where $L_s$ is as in \eqref{eq:def_N_v_T}.
With the above estimates available, one can check that the estimate \cite[eq.\ (5.54)]{Primitive1} extends to \eqref{eq:primitive_tilde} and one gets \eqref{eq:estimate_v_wt}.

\subsection{Estimate for $\||\wt{v}||\nabla \T|\|_{L^2_t L^2_x}$ and 
$\||\T| |\nabla \wt{v}|\|_{L^2_t L^2_x}$}
\label{ss:estimate_tilde_v_T}
The aim of this subsection is to prove the following estimate: For all $\varepsilon_6\in (0,\infty)$,
\begin{align}
\label{eq:estimate_mixed_v_T}
&\E \int_{\eta}^{\xi} \Big\||\wt{v}||\nabla \T|\Big\|_{L^2}^2\,\dd s 
+ \E\int_{\eta}^{\xi}\Big\||\T| |\nabla \wt{v}|\Big\|_{L^2}^2\,\dd s
\leq C_{6,\varepsilon_6}\big(1+\E\|\T(\eta)\|_{L^4}^4+\E\|\wt{v}(\eta)\|_{L^4}^4 \big)\\
\nonumber
&\qquad \qquad 
\leq C_{6,\varepsilon_6}
\E\int_{\eta}^{\xi} L_t (1+\|\wt{v}\|_{L^4}+\|\T\|_{L^4}^4)\,\dd s\\
\nonumber
&\qquad \qquad 
+\varepsilon_6 \E\int_{\eta}^{\xi} \Big[ \big\||\wt{v}|^2|\nabla\wt{v}|\big\|_{L^2}^2+ 
\|\overline{v}\|^2_{H^2(\Tor^2)}+\|\wh{\T}\|^2_{H^2(\Tor^2)}+ \| \partial_3 \nabla v\|_{L^2}^2+\| \partial_3 \nabla \T\|_{L^2}^2\Big]\,\dd s\\
\nonumber
&\qquad \qquad \qquad
+ C_6 \E\int_{\eta}^{\xi} \big\||\T||\nabla_{\h}\TT|\big\|_{L^2}^2\,\dd s,
\end{align}
where $C_6,C_{6,\varepsilon_6}$  are constants independent of $(j,\eta,\xi,v_0,\T_0)$ and $C_6$ is also independent of $\varepsilon_6$. Finally, $L_s$ is as in \eqref{eq:def_N_v_T}.

%%%%
To prove \eqref{eq:estimate_mixed_v_T}, we apply the It\^{o} formula to the functional $(\wt{v},\T)\mapsto \big\||\wt{v}||\T|\big\|_{L^2}^2$.  To this end, recall that $\wt{v}$ and $\T$ satisfy the SPDEs \eqref{eq:primitive_tilde} and \eqref{eq:primitive_v_proof_global_2}, respectively. Moreover, we let
$$
\T^{\eta,\xi}\stackrel{{\rm def}}{=}\T((\cdot\vee \eta)\wedge \xi)
\quad \text{ and }\quad 
\wt{v}^{\eta,\xi}\stackrel{{\rm def}}{=}\wt{v}((\cdot\vee \eta)\wedge \xi).
$$
Applying the It\^{o} formula to $(\wt{v},\T)\mapsto \big\||\T|^2 |\wt{v}|^2\big\|_{L^2}^2$ we have, a.s.\ for all $t\in \R_+$,
\begin{align}
\label{eq:T_v_equation_Ito}
\big\||\T^{\eta,\xi}(t)|^2 |\wt{v}^{\eta,\xi}(t)|^2\big\|_{L^2}^2
&= 
\big\||\T(\eta)|^2 |\wt{v}(\eta)|^2 \,\big\|_{L^2}^2\\
\nonumber
&+\sum_{1\leq j\leq 4}\int_0^t \one_{[\eta,\xi]}  I_{2,j}(s)\,\dd s
+N_t, 
\end{align}
where $N$ is a $L^1(\O)$-martingale, such that $\E[N_t]=0$ for all $0\leq t\leq T$, and 
\begin{align*}
J_{1}&\stackrel{{\rm def}}{=}
2 \int_{\Dom} (\T^2 \wt{v}\cdot \Delta \wt{v} + |\wt{v}|^2 \T \Delta \T)\,\dd x, \\
J_{2}&\stackrel{{\rm def}}{=}
2\int_{\Dom} (\T^2 \wt{v}\cdot (f_v+ \force(\wt{v})) +  |\wt{v}|^2 \T f_{\T})\,\dd x,\\
J_{3}&\stackrel{{\rm def}}{=}
2\int_{\Dom} \T^2 \wt{v}\cdot 
\Big(\Lt \T+(\tp_{\h}\cdot\nabla_{\h})\wh{\T}+\reallywidehat{\tp^3\partial_3\nabla_{\h}\T}\Big)\,\dd x\dd s \\
J_{4}&\stackrel{{\rm def}}{=}
-2\int_{\Dom} \T^2 \wt{v}\cdot \big[(\wt{v}\cdot\nabla_{\h})\overline{v}\big]\,\dd x,\\
J_{5}&\stackrel{{\rm def}}{=}
\sum_{n\geq 1}\int_{\Dom} |\wt{v}|^2 [(\psi_n \cdot\nabla)\T+g_{\T,n}]^2 \,\dd x,\\
J_{6}&\stackrel{{\rm def}}{=} 
\int_{\Dom} |\T|^2 |(\phi_n \cdot\nabla) \wt{v}-\overline{\phi^3_n \partial_3 v}+ \opt_n (\T)+\wt{g_{n,v}}|^2\,\dd x,\\
J_{7}&\stackrel{{\rm def}}{=} 
2\sum_{n\geq 1}
\int_{\Dom} \T [(\psi_n \cdot\nabla)\T+g_{\T,n}] 
\wt{v}\cdot [(\phi_n \cdot\nabla) \wt{v}-\overline{\phi^3_n \partial_3 v}+ \opt_n (\T)+\wt{g_{n,\T}}]\,\dd x,
\end{align*}
and we used that, a.e.\ on $[0,\tau)\times \O$,
\begin{align*}
\int_{\Dom} \Big(|\wt{v}|^2 \T \big[(\overline{v}\cdot \nabla_{\h} )\T\big] + |\T|^2 \wt{v}\cdot \big[(\overline{v}\cdot\nabla_{\h})\wt{v}\big]\Big)\,\dd x&=0,\\
\int_{\Dom} \Big(|\wt{v}|^2 \T \big[(\wt{u}\cdot \nabla )\T\big] + |\T|^2 \wt{v}\cdot \big[(\wt{u}\cdot\nabla)\wt{v}\big]\Big)\,\dd x&=0,
\end{align*}
where $\wt{u}=(\wt{v},w(\wt{v}))$ and $w(\wt{v})$ is as in  \eqref{eq:def_w}.
The above follows from Lemma \ref{l:cancellation}, \eqref{eq:divergence_free_bar_v} and a standard approximation argument.
Let us remark that the application of the It\^{o} formula in \eqref{eq:T_v_equation_Ito} requires an approximation argument similar to the one used in Step 3 of \cite[Lemma 5.3]{Primitive1}. To avoid repetitions, we omit the details.

For the reader's convenience, we collect the estimates of $(J_{j})_{j=1}^7$ in the following subsections. The proof of \eqref{eq:estimate_mixed_v_T} will be given in Subsection \ref{sss:proof_estimate_mixed_v_T}.
Below $\varepsilon_6,\delta_6\in (0,\infty)$ are positive parameters which will be chosen in Subsections \ref{ss:proof_lemma_main_estimate} and 
\ref{sss:proof_estimate_mixed_v_T}, respectively.

\subsubsection{Estimate of $J_{1}$}
\label{sss:estimate_I_1}
Integrating by parts and using the boundary conditions \eqref{eq:boundary_conditions_full}, we have 
\begin{align*}
\int_{\Dom} \T^2 \wt{v}\cdot \Delta \wt{v} \,\dd x
&= -\int_{\Dom} \T^2 |\nabla \wt{v}|^2\,\dd x 
-2 \sum_{1\leq i,j\leq 3}\int_{\Dom} \T \wt{v}^i \partial_j \wt{v}^i \partial_j \T\,\dd x
\end{align*}
and
\begin{align*}
\int_{\Dom} |\wt{v}|^2 \T \Delta \T \,\dd x
&=-\alpha\int_{\Tor^2} |\wt{v}(\cdot,x_{\h},0)|^2 |\T(\cdot,x_{\h},0)|^2\,\dd x_{\h} \\
& \ \ \ -\int_{\Dom} |\wt{v}|^2 |\nabla \T|^2\,\dd x 
- 2\sum_{1\leq i,j\leq 3}\int_{\Dom} \T \wt{v}^i \partial_j \wt{v}^i \partial_j \T\,\dd x.
\end{align*}
By the boundedness of the trace operator \eqref{eq:boundedness_trace}, for any $r\in (\frac{1}{2},1)$,
\begin{align*}
&\int_{\Tor^2} |\wt{v}(\cdot,x_{\h},0)|^2 |\T(\cdot,x_{\h},0)|^2\,\dd x_{\h}\lesssim_{r} \|\wt{v}\T\|_{H^{r}}^2
\lesssim_{r} \|\wt{v}\T\|_{L^2}^{2(1-r)} \|\wt{v}\T\|_{H^1}^{2r}\\
&\qquad\qquad\qquad 
\lesssim_{r} \|\wt{v}\T\|_{L^2}^{2(1-r)} 
\Big(\|\wt{v}\T\|^2_{L^2}+\int_{\Dom}|\nabla {v}|^2|\T|^2\,\dd x +\int_{\Dom} |\wt{v}|^2 |\nabla\T|^2 \,\dd x \Big)^{r}\\
&\qquad\qquad\qquad
\leq C_{r,\delta_6}\|\wt{v}\T\|_{L^2}^2+ \delta_6\Big( \int_{\Dom}|\nabla \wt{v}|^2|\T|^2\,\dd x +
\int_{\Dom} |{v}|^2 |\nabla\T|^2 \,\dd x\Big)\\
&\qquad\qquad\qquad 
\leq  C_{r,\delta_6}
\big(\|\wt{v}\|_{L^4}^4+ \|\T\|_{L^4}^4\big)+ \delta_6\Big( \int_{\Dom}|\nabla \wt{v}|^2|\T|^2\,\dd x +
\int_{\Dom} |\wt{v}|^2 |\nabla\T|^2 \,\dd x\Big).
\end{align*}
By the Cauchy-Schwartz inequality, we have
\begin{equation*}
\sum_{1\leq i,j\leq 3}\Big|\int_{\Dom} \T \wt{v}^i \partial_j \wt{v}^i \partial_j \T\,\dd x\Big| \leq 
\varepsilon_6 \int_{\Dom} |\wt{v}|^2|\nabla \wt{v}|^2\,\dd x 
+ C_{\varepsilon_6} \int_{\Dom} |\T|^2|\nabla \T|^2\,\dd x . 
\end{equation*}
Summarizing the previous estimates, we have, a.e.\ on $[0,\tau)\times \O$,
\begin{align*}
J_{1}
&\leq
- (2-\delta_6) \Big(\int_{\Dom} \T^2 |\nabla \wt{v}|^2\,\dd x
+ \int_{\Dom} |\wt{v}|^2 |\nabla \T|^2\,\dd x\Big) \\
&+\varepsilon_6 \int_{\Dom} |\wt{v}|^2|\nabla \wt{v}|^2\,\dd x  
+ C_{\varepsilon_6,\delta_6}\big(L+ \|\wt{v}\|_{L^4}^4+ \|\T\|_{L^4}\big),  
\end{align*}
where we have also used that $\int_{\Dom}|\T|^2|\nabla \T|^2\,\dd s \leq L$ by \eqref{eq:def_N_v_T}.

\subsubsection{Estimate of $J_{2}$}
Let us write $J_{2}=J_{2,1}+J_{2,2}$ where 
\begin{equation*}
J_{2,1} \stackrel{{\rm def}}{=}2
\int_{\Dom} \T^2 \wt{v}\cdot \big(f_v+ \force(\wt{v})\big)\,\dd x \ \ \ \text{ and } \ \ \ 
J_{2,2}
\stackrel{{\rm def}}{=}2
\int_{\Dom}
 |\wt{v}|^2 \T f_{\T}\,\dd x.
\end{equation*}
To estimate such terms, observe that
\begin{align}
\label{eq:interpolation_L6_wt_v}
\|\wt{v}\|_{L^6}^2
= 
\big\||\wt{v}|^2\big\|_{L^3}
&\stackrel{\eqref{eq:interpolation_inequality_L3}}{\lesssim }
\big\||\wt{v}|^2\big\|_{L^2}^{1/2}\Big(\big\||\wt{v}|^2\big\|_{L^2}^{1/2}
+ \big\|\nabla |\wt{v}|^2\big\|_{L^2}^{1/2}\Big)\\
\nonumber
&\ \ =\ 
\|\wt{v}\|_{L^4}^2
+ 
\|\wt{v}\|_{L^4}\Big\||\wt{v}||\nabla \wt{v}|\Big\|_{L^2}^{1/2}.
\end{align}
Thus, since $\|\force(\wt{v})\|_{L^2}\lesssim \big\||\wt{v}||\nabla \wt{v}|\big\|_{L^2}$ due to \eqref{eq:primitive_bar_2}, we have,
a.e.\ on $[\eta,\xi]\times \O$,
\begin{align*}
|J_{2,1}| 
&\lesssim \|\T^2\|_{L^3}\|\wt{v}\|_{L^6}\big(\|f_v\|_{L^2}+\|\force(\wt{v})\|_{L^2} \big)\\
&\lesssim \|\T\|_{L^6}^2\Big(\|\wt{v}\|_{L^4}+ 
\|\wt{v}\|_{L^4}^{1/2}\big\||\wt{v}||\nabla \wt{v}|\big\|_{L^2}^{1/4}\Big)
 \Big(\|f_v\|_{L^2}+\big\||\wt{v}||\nabla \wt{v}|\big\|_{L^2}\Big)\\
&\leq \varepsilon_6 \big\||\wt{v}||\nabla \wt{v}|\big\|_{L^2}^2 +C_{\varepsilon_6} \|f_v\|_{L^2}^2
+ C_{\varepsilon_6}\|\T\|_{L^6}^8(1+ \|\wt{v}\|_{L^4}^4),
\end{align*}
where in the last step we applied the Young inequality twice.

Similarly, we can estimate $J_{2,2}$. Indeed, a.e.\ on $[\eta,\xi]\times\O$,
\begin{align*}
|J_{2,2}|&\lesssim \|f_{\T}\|_{L^2}\||\wt{v}|^2\|_{L^3}\|\T\|_{L^6}\\
&= \|f_{\T}\|_{L^2}\|\wt{v}\|_{L^6}^2\|\T\|_{L^6}\\
&\leq \varepsilon_6 
\big\||\wt{v}||\nabla \wt{v}|\big\|_{L^2}^2+ C_{\varepsilon_6} \|f_{\T}\|_{L^2}^2+ C_{\varepsilon_6}(1+\|\T\|_{L^6}^8)( 1+\|\wt{v}\|_{L^4}^4),
\end{align*}
where in the last step we applied Young's inequality twice again.

\subsubsection{Estimate of $J_3$}
Let us decompose $J_3$ as $J_3=J_{3,1}+J_{3,2}+J_{3,3}$ where
\begin{align*}
J_{3,1}
&\stackrel{{\rm def}}{=}
\int_{\Dom} \T^2 \wt{v}\cdot\big[ (\tp_{\h}\cdot\nabla_{\h}) \nabla_{\h}\TT\big]\,\dd x  ,\\
J_{3,2}
&\stackrel{{\rm def}}{=}
\int_{\Dom} \T^2 \wt{v}\cdot\Big( \int_{-h}^{\cdot} \tp^3 (\cdot,\zeta) \partial_3 \nabla_{\h}\T(\cdot,\zeta)\,\dd \zeta\Big)\,\dd x  ,\\
J_{3,3}
&\stackrel{{\rm def}}{=}
\int_{\Dom} \T^2 \wt{v}\cdot (\tp_{\h}\cdot\nabla_{\h})\nabla_{\h} \wh{\T} \,\dd x  ,\\
J_{3,4}
&\stackrel{{\rm def}}{=}
\int_{\Dom} \T^2 \wt{v}\cdot \reallywidehat{\tp^3 \partial_3 \nabla_{\h}\T} \,\dd x .
\end{align*}
We begin by looking at $J_{3,1}$. Integrating by parts, we have, a.e.\ on $[\eta,\xi]\times \O$,
\begin{align*}
|J_{3,1}|
&\leq  \int_{\Dom} |\nabla_{\h}\tp_{\h}|\, \T^2  \,|\wt{v}|\,  |\nabla_{\h}\TT|  \,\dd x\\
&+
\int_{\Dom} |\T|\, |\nabla_{\h}\T| \,|\wt{v}|\, |\nabla_{\h}\TT|  \,\dd x
+
\int_{\Dom}  \T^2 |\nabla_{\h}\wt{v}| \,|\nabla_{\h}\TT|  \,\dd x.
\end{align*}
By the Cauchy-Schwartz inequality and Assumption \ref{ass:well_posedness_primitive_double_strong}\eqref{it:well_posedness_primitive_kone_smoothness_strong_strong},
\begin{align*}
|J_{3,1}|&\leq \delta_6 \big(\int_{\Dom} |\wt{v}|^2 |\nabla \T|^2\,\dd x +
\int_{\Dom} |\wt{v}|^2 |\nabla \wt{v}|^2\,\dd x\big)\\
&+C_{\delta_6}\int_{\Dom}|\T|^2 |\nabla_{\h} \TT|^2\,\dd x + C_{\delta_6}(1+\|\T\|_{L^6}^8)( 1+\|\wt{v}\|_{L^4}^4).
\end{align*}
%%%%%%%
Similarly, one can readily check that,  a.e.\ on $[\eta,\xi]\times \O$,
\begin{align*}
\sum_{2\leq j\leq 4} |J_{3,j}|
&\leq  \delta_6 \Big(\int_{\Dom} |\wt{v}|^2 |\nabla \T|^2\,\dd x +
\int_{\Dom} |\wt{v}|^2 |\nabla \wt{v}|^2\,\dd x\Big)\\
&+ \varepsilon_6 \big(\|\partial_3 \nabla \T\|_{L^2}^2+\|\partial_3 \nabla v\|_{L^2}^2 + \|\wh{\T}\|_{H^2(\Tor^2)}^2\big)
+
C_{\delta_6,\varepsilon_6}(1+\|\T\|_{L^6}^8)( 1+\|\wt{v}\|_{L^4}^4).
\end{align*}

\subsubsection{Estimate of $J_{4}$}
The H\"{o}lder inequality and the embedding $H^1\embed L^6$ yield, a.e.\ on $[\eta,\xi]\times \O$,
\begin{align*}
 |J_{4}|
 &\lesssim \|\T^2\|_{L^3} \||\wt{v}|^2\|_{L^2} \|\nabla_{\h} \overline{v}\|_{L^6(\Tor^2)}\\
&\lesssim\|\T\|_{L^6}^2 \|\wt{v}\|_{L^4}^2 \|\overline{v}\|_{H^2(\Tor^2)}
\leq \varepsilon_6 \|\overline{v}\|_{H^2}^2+ C_{\varepsilon_6} \|\T\|_{L^6}^4 \|\wt{v}\|_{L^4}^4.
\end{align*}

\subsubsection{Estimate of $J_{5}$}
We begin by noticing that, for all $\varepsilon_0\in (0,\infty)$ and a.e.\ on $[0,\tau)\times \O$,
\begin{align*}
|J_{5}| 
&\stackrel{(i)}{\leq} (\ellip+\delta_6)
\int_{\Dom} |\wt{v}|^2 |\nabla \T|^2\,\dd x +
C_{\delta_6} \int_{\Dom}|\T|^2 \|g_{\T}\|_{\ell^2}^2\,\dd x\\
&\stackrel{(ii)}{\leq} (\ellip+\delta_6)
\int_{\Dom} |\wt{v}|^2 |\nabla \T|^2\,\dd x +
C_{\delta_6} \|\T\|_{L^4}^2 \|g_{\T}\|_{H^1(\ell^2)}^2,
\end{align*}
where in $(i)$ we used Assumption \ref{ass:well_posedness_primitive_double_strong}\eqref{it:well_posedness_primitive_parabolicity_strong_strong} and in $(ii)$ that $H^1(\ell^2)\embed L^4(\ell^2)$.

\subsubsection{Estimate of $J_6$} To begin, note that, a.e.\ on $[0,\tau)\times \O$,
\begin{align*}
|J_{6}|
&\leq (\ellip+\delta_6)
\int_{\Dom} |\T|^2 |\nabla \wt{v}|^2\,\dd x\\
&+ 
C_{\delta_6}  \Big(\|\T\|_{L^4}^2 \|g_{v}\|_{H^1(\ell^2)}^2
+ \Big\||\T||\nabla \TT|\Big\|_{L^2}^2
+\sum_{n\geq 1} \int_{\Dom} |\T|^2 |\overline{\phi^3_n \partial_3 v}|^2\,\dd x\Big).
\end{align*}
Next, we estimate the last term on the right-hand side of the previous inequality.
To this end, note that  $|\overline{\phi^3_n \partial_3 v}|^2$ is $\z$-independent. Therefore, 
\begin{align*}
\sum_{n\geq 1}\int_{\Dom} |\T|^2 |\overline{\phi^3_n \partial_3 v}|^2 \,\dd x 
&\lesssim_h 
\|\T\|_{L^{2}(-h,0;L^4)}^2\sum_{n\geq 1}  \|\overline{\phi^3_n \partial_3 v}\|_{L^4(\Tor^2)}^2\\
&\stackrel{(i)}{\lesssim} \|\T\|_{L^{2}(-h,0;L^4)}^2  \|(\overline{\phi^3_n \partial_3 v})_{n\geq 1}\|_{L^2(\Tor^2;\ell^2)}\|(\overline{\phi^3_n \partial_3 v})_{n\geq 1}\|_{H^1(\Tor^2;\ell^2)}\\
&\stackrel{(ii)}{\lesssim}_M 
\|\T\|_{L^4}^2 \| v\|_{H^1}(\| v\|_{H^1}+\|\nabla \partial_3 v\|_{L^2})\\
&\leq C_{\varepsilon_6}(1+ \|\T\|_{L^4}^4)(1+ \| v\|_{H^1}^2)+  \varepsilon_6 \|\nabla \partial_3 v\|_{L^2}^2,
\end{align*}
where in $(i)$ we used \eqref{eq:interpolation_inequality_L4}, the Cauchy-Schwartz inequality and $\ell^2(L^2)=L^2(\ell^2)$. Finally, $(ii)$ follows from $\|(\phi_n^j)_{n\geq 1}\|_{L^{\infty}(\ell^2)}\lesssim_M 1$ as commented in Remark \ref{r:boundedness}.

\subsubsection{Estimate of $J_{7}$}
\label{sss:estimate_I_6}
The Cauchy-Schwarz inequality yields 
\begin{align*}
|J_7|
&\leq \varepsilon_6 \sum_{n\geq 1}\int_{\Dom} |\wt{v}|^2\big|(\phi_n \cdot\nabla) \wt{v}-\overline{\phi^3_n \partial_3 v}+ \opt_n (\T)+\wt{g_{n,\T}}\big|^2\,\dd x \\
&+C_{\varepsilon_6}\sum_{n\geq 1}
\int_{\Dom} \big|\T [(\psi_n \cdot\nabla)\T+g_{\T,n}]\big|^2 \,\dd x\\ 
&\leq \varepsilon_6 \int_{\Dom}\big( |\wt{v}|^2|\nabla \wt{v}|^2
+ |\nabla \partial_3 v|^2 +|\wt{v}|^2|\nabla \TT|^2\big)\,\dd x\\
&+
C_{\varepsilon_6} \Big(\big\||\T||\nabla \T|\big\|_{L^2}^2+\|\wt{v}\|_{L^4}^2 \|g_{\T}\|_{H^1(\ell^2)}^2+\|\T\|_{L^4}^2 
\|g_v\|_{H^1(\ell^2)}^2 \Big),
\end{align*}
where in the last inequality we used that $\|(\phi_n^j)_{n\geq 1}\|_{L^{\infty}(\ell^2)}\lesssim_M 1$.

\subsubsection{Proof of \eqref{eq:estimate_mixed_v_T}}
\label{sss:proof_estimate_mixed_v_T}
Recall that $\ellip<2$ by Assumption \ref{ass:well_posedness_primitive_double_strong}\eqref{it:well_posedness_primitive_parabolicity_strong_strong}. 
Due to the estimates of Subsection \ref{sss:estimate_I_1}-\ref{sss:estimate_I_6} with $\varepsilon_0$ sufficiently small and independent of $(j,\eta,\xi,v_0,\T_0)$, the claimed estimate follows by taking $t=T$ and the expected value on both sides of \eqref{eq:T_v_equation_Ito} as well as by using $\E[N_T]=0$. 
%Note that it is sufficient to take $\varepsilon_0= \frac{2-\ellip}{4}>0$. 

Note that, in contrast to the previous subsections, we do not take  $\E[\sup_{t\in [0,T]} |\cdot|]$ on both sides of \eqref{eq:T_v_equation_Ito}. This would eventually give us an estimate for $\E[\sup_{t\in [\eta,\xi]} \||\wt{v}(t)||\T(t)|\|_{L^2}^2]$. However, this already follows from the $L^{\infty}_t(L^4_x)$-estimates for $\wt{v}$ and $\T$ proven in Subsection \ref{ss:estimate_wt_v_L_4} and Lemma \ref{l:basic_estimates}, respectively.

\subsection{Estimate for $\big\||\TT|  |\nabla\wt{v}|\big\|_{L^2_t L^2_x}$ and $ 
\big\| |\wt{v}||\nabla_{\h}\TT| \big\|_{L^2_t L^2_x}$} 
\label{ss:estimate_tilde_v_varphi}
The aim of this subsection is to prove the following estimate: For all $\varepsilon_7\in (0,\infty)$,
%\todo{The constant in front of $\E\int_{\eta}^{\xi}\Big\| |\TT| |\nabla \TT|\Big\|_{L^2}^2\, ds$ depends on $\varepsilon$}
\begin{align}
\label{eq:estimate_mixed_v_varphi}
&\E \int_{\eta}^{\xi} \Big\| |\TT| |\nabla\wt{v}| \Big\|_{L^2}^2\,\dd s 
+ 
\E \int_{\eta}^{\xi} \Big\| |\wt{v}||\nabla_{\h}\TT| \Big\|_{L^2}^2\,\dd s
\leq 
C_{7,\varepsilon_7}\big(1+ \E\|\T(\eta)\|_{L^4}^4+\E\|\wt{v}(\eta)\|_{L^4}^4\big)\\
&\qquad\qquad
\nonumber
+\varepsilon_7 \Big(
\E\int_{\eta}^{\xi} \Big\| |\wt{v}||\nabla \T| \Big\|_{L^2}^2\,\dd s 
+
\E\int_{\eta}^{\xi} \Big\||\nabla\wt{v}||\T| \Big\|_{L^2}^2\,\dd s
+\E\int_{\eta}^{\xi} \|\wh{\T} \|_{H^2(\Tor^2)}^2\,\dd s \Big)\\
\nonumber
&\qquad\qquad +\varepsilon_7 \Big(
\E\int_{\eta}^{\xi} \Big\||\nabla\wt{v}||\wt{v}| \Big\|_{L^2}^2\,\dd s 
+\E\int_{\eta}^{\xi} \|\nabla \partial_3 v \|_{L^2}^2\,\dd s
+\E\int_{\eta}^{\xi} \|\nabla \partial_3 \T \|_{L^2}^2\,\dd s
\Big)\\
\nonumber
&\qquad\qquad%\qquad \quad 
+C_{7,\varepsilon_7} \E\int_{\eta}^{\xi} L_s (1+\|\wt{v}\|_{L^4}+\|\T\|_{L^4}^4)\,\dd s\\
\nonumber
&\qquad \qquad +C_{7,\varepsilon_7}  \E\int_{\eta}^{\xi}\Big\| |\TT| |\nabla_{\h} \TT|\Big\|_{L^2}^2\, \dd s ,
\end{align}
where $C_7,C_{7,\varepsilon_7}$  are constants independent of $(j,\eta,\xi,v_0,\T_0)$ and $C_7$ is also independent of $\varepsilon_7$. Finally, $L_s$ is as in \eqref{eq:def_N_v_T}.

Here the idea is to apply the It\^{o} formula to the functional 
\begin{equation}
\label{eq:functional_vv_TT}
(\wt{v},\T)\mapsto \Big\||\wt{v} | |\int_{-h}^{\cdot} \T(\cdot,\zeta)\,\dd \zeta| \Big\|_{L^2(\Dom)}^2.
\end{equation}
Recall that $\TT=\int_{-h}^{\cdot}\T(\cdot,\zeta)\,\dd \zeta$, see \eqref{eq:def_TT}.
In the following result, we show cancellation properties involving convective terms, which will be useful in the application of such a formula.

\begin{lemma}[Cancellation]
\label{l:cancellation_vv_TT}
Let $v\in C^{\infty}(\Dom;\R^2)$ and set $w(v)=-\int_{-h}^{\cdot} \div_{\h} v(\cdot,\zeta)\,\dd \zeta$, $u=(v,w(v))$. Then, for all $\theta\in C^{\infty}(\Dom)$,
\begin{equation}
\label{eq:cancellation_vv_TT}
\begin{aligned}
\int_{\Dom} \TT^2 v\cdot[(u\cdot\nabla)v] \,\dd x 
&+\int_{\Dom} |v|^2 \TT \Big(\int_{-h}^{\cdot} (u\cdot\nabla)\T\,\dd \zeta\Big)\,\dd x\\
&= \int_{\Dom} |v|^2 \TT \Big[ \int_{-h}^{\cdot} (v\cdot\nabla_{\h})\T\,\dd \zeta- (v\cdot\nabla_{\h})\TT 
+ \int_{-h}^{\cdot} \div_{\h} v\,\T\,\dd \zeta\Big]\,\dd x
\end{aligned}
\end{equation}
where $\TT=\int_{-h}^{\cdot} \T(\cdot,\zeta)\,\dd \zeta$, see \eqref{eq:def_TT}.
\end{lemma}

The key point is that on the right-hand side of \eqref{eq:cancellation_vv_TT} the vertical component $w(v)$ of $u$ does not appear.

\begin{proof}[Proof of Lemma \ref{l:cancellation_vv_TT}]
Since $[w(v)](\cdot,-h)=0$ on $\Tor^2$, 
\begin{align*}
\int_{-h}^{\cdot} w(v)\partial_3 \T\,\dd \zeta
&= w(v)\T +\int_{-h}^{\cdot} \div_{\h} v\,\T\,\dd \zeta\\
&= w(v)\partial_3\TT -\int_{-h}^{\cdot} \div_{\h} v\,\T\,\dd \zeta\\
&= (u\cdot\nabla) \TT- (v\cdot\nabla_{\h}) \TT +\int_{-h}^{\cdot} \div_{\h} v\, \T\,\dd \zeta.
\end{align*}
Hence 
\eqref{eq:cancellation_vv_TT} follows by using that   
$
\int_{\Dom} \big[|v|^2\TT (u\cdot\nabla )\TT + |\TT|^2 v\cdot (u\cdot\nabla) v\big]\,\dd x=0
$
, cf.\ Lemma \ref{l:cancellation}.
\end{proof}

Next, we apply the It\^{o}'s formula to the functional in \eqref{eq:functional_vv_TT}. As in Subsection \ref{ss:estimate_tilde_v_T}, a standard approximation argument shows that 
\begin{align}
\label{eq:TT_v_equation_Ito}
\big\||\TT^{\eta,\xi}(t)|^2 |\wt{v}^{\eta,\xi}(t)|^2\big\|_{L^2}^2
&= 
\big\||\TT(\eta)|^2 |\wt{v}(\eta)|^2 \,\big\|_{L^2}^2\\
\nonumber
&+\sum_{1\leq j\leq 7}\int_0^t \one_{[\eta,\xi]}  K_{2,j}(s)\,\dd s
+\mathcal{N}_t, 
\end{align}
where $\mathcal{N}_t$ is an $L^1(\O)$-martingale, such that $\E[\mathcal{N}_t]=0$ for all $0\leq t\leq T$, and 
\begin{align*}
K_{1}&\stackrel{{\rm def}}{=}
2 \int_{\Dom} \Big(\TT^2 \wt{v}\cdot \Delta \wt{v} + |\wt{v}|^2 \TT \int_{-h}^{\cdot}\Delta \T(\cdot,\zeta)\,\dd \zeta\Big)\,\dd x, \\
K_{2}&\stackrel{{\rm def}}{=}
\int_{\Dom} |\wt{v}|^2 \TT \Big[ \int_{-h}^{\cdot} (\wt{v}\cdot\nabla_{\h})\T\,\dd \zeta- (\wt{v}\cdot\nabla_{\h})\TT 
+ \int_{-h}^{\cdot} \div_{\h} \wt{v}\,\T\,\dd \zeta\Big]\,\dd x,\\
K_{3}&\stackrel{{\rm def}}{=}
2\int_{\Dom} \Big(\TT^2 \wt{v}\cdot (f_v+ \force(\wt{v})) +  |\wt{v}|^2 \TT \int_{-h}^{\cdot} f_{\T}(\cdot,\zeta)\,\dd \zeta\Big)\, \dd x,\\
K_{4}&\stackrel{{\rm def}}{=}
2\int_{\Dom} \TT^2 \wt{v}\cdot 
\Big(\Lt \T+(\tp_{\h}\cdot\nabla_{\h})\wh{\T}+\reallywidehat{\tp^3\partial_3\nabla_{\h}\T}\Big)\,\dd x\dd s, \\
K_{5}&\stackrel{{\rm def}}{=}
-2\int_{\Dom} \TT^2 \wt{v}\cdot \big[(\wt{v}\cdot\nabla_{\h})\overline{v}\big]\,\dd x,\\
K_{6}&\stackrel{{\rm def}}{=}
\sum_{n\geq 1}\int_{\Dom} |\wt{v}|^2 \Big(\int_{-h}^{\cdot} \big[(\psi_n(\cdot,\zeta) \cdot\nabla)\T(\cdot,\zeta)+g_{\T,n}(\cdot,\zeta)\big]\,\dd \zeta\Big)^2 \,\dd x,\\
K_{7}&\stackrel{{\rm def}}{=} 
\int_{\Dom} |\TT|^2 |(\phi_n \cdot\nabla) \wt{v}-\overline{\phi^3_n \partial_3 v}+ \opt_n (\T)+\wt{g_{n,v}}|^2\,\dd x,\\
K_{8}&\stackrel{{\rm def}}{=} 
2\sum_{n\geq 1}
\int_{\Dom} \TT \Big(\int_{-h}^{\cdot} \big[(\psi_n(\cdot,\zeta) \cdot\nabla)\T(\cdot,\zeta)+g_{\T,n}(\cdot,\zeta)\big]\,\dd \zeta\Big) \\
&\qquad \qquad \qquad \qquad \qquad 
\wt{v}\cdot [(\phi_n \cdot\nabla) \wt{v}-\overline{\phi^3_n \partial_3 v}+ \opt_n (\T)+\wt{g_{n,\T}}]\,\dd x,
\end{align*}
%%%%%%%%%
where we used Lemma \ref{l:cancellation_vv_TT} with $(v,\T)$ replaced by $(\wt{v},\TT)$, and by Lemma \ref{l:cancellation},
\begin{align*}
\int_{\Dom} \Big[\TT^2 \wt{v}\cdot[(\overline{v}\cdot\nabla_{\h})\wt{v}]  
&+ |\wt{v}|^2 \TT \Big(\int_{-h}^{\cdot} (\overline{v}\cdot\nabla_{\h}) \T\,\dd \zeta\Big)\Big]\,\dd x\\
&=
\int_{\Dom} \TT^2 \wt{v}\cdot[(\overline{v}\cdot\nabla_{\h})\wt{v}] 
+|\wt{v}|^2 \TT (\overline{v}\cdot\nabla_{\h}) \TT\,\dd x=0.
\end{align*}

As before, we collect the estimates of $(K_{j})_{j=1}^7$ in the following subsections. %The estimates of  
Below $\varepsilon,\varepsilon_0\in (0,\infty)$ are positive parameters which will be chosen in Subsections \ref{ss:proof_lemma_main_estimate} and 
\ref{sss:proof_estimate_mixed_v_T}, respectively.

\subsubsection{Estimate of $K_1$} Integrating by parts, we have, a.e.\ on $[\eta,\xi]\times \O$,
\begin{align*}
\int_{\Dom} \TT^2\Delta \wt{v}\cdot\wt {v}\,\dd x= -2\int_{\Dom} \TT^2 |\nabla \wt{v}|^2\,\dd x 
-2\sum_{1\leq i,j\leq 3} \int_{\Dom} \TT \partial_i \TT \,\wt{v}^j \partial_i \wt{v}^j\,\dd x,
\end{align*}
and by \eqref{eq:primitive_v_proof_global_bc_2},
\begin{align*}
&\int_{\Dom}|\wt {v}|^2 \TT\Big( \int_{-h}^{\cdot} \Delta \T(\cdot,\zeta)\,\dd \zeta\Big)\,\dd x
= 
\int_{\Dom}|\wt {v}|^2 \TT \Delta_{\h} \TT\,\dd x+ \int_{\Dom} |\wt{v}|^2 \TT \partial_3 \T\,\dd x\\
&  
= 
-2\int_{\Dom}|\wt {v}|^2 |\nabla_{\h} \TT|^2\,\dd x
- 2\sum_{{ 1\leq j\leq 3}} \int_{\Dom} \TT \nabla_{\h} \TT \cdot \nabla_{\h} \wt{v}^j \wt{v}^j\,\dd x
+ \int_{\Dom} |\wt{v}|^2 \TT \partial_3 \T\,\dd x.
\end{align*}
Hence, a.e.\ on $[\eta,\xi]\times \O$,
\begin{align*}
K_1 
&\leq -2 \int_{\Dom} \big(\TT^2 |\nabla \wt{v}|^2+ |\wt{v}|^2 |\nabla_{\h} \TT|^2 \big)\,\dd x\\
&+ \varepsilon_7 \int_{\Dom}\big( |\wt{v}|^2 |\nabla \wt{v}|^2+ |\partial_3 \nabla \T|^2\big)\, \dd x\\
&+ C_{\varepsilon_7}\Big( \int_{\Dom} \TT^2 |\nabla_{\h}\TT|^2	\,\dd x+  (1+\|\T\|_{L^6}^8) (1+\|\wt{v}\|_{L^4}^4)\Big).
\end{align*}

\subsubsection{Estimate of  $K_2$} We write $K_{2}=K_{2,1}+K_{2,2}+K_{2,3}$ where
\begin{align*}
K_{2,1}&\stackrel{{\rm def}}{=}
\int_{\Dom} |\wt{v}|^2 \TT \Big( \int_{-h}^{\cdot} (\wt{v}\cdot\nabla_{\h})\T\,\dd \zeta\Big) \,\dd x ,\\
K_{2,2}&\stackrel{{\rm def}}{=} - 
\int_{\Dom} |\wt{v}|^2 \TT \big[ (\wt{v}\cdot\nabla_{\h})\TT \big]\,\dd x,\\
K_{2,3}&\stackrel{{\rm def}}{=}
\int_{\Dom} |\wt{v}|^2 \TT \Big( \int_{-h}^{\cdot} \div_{\h} \wt{v}\,\T\,\dd \zeta\Big)\,\dd x.
\end{align*}
Note that, a.e.\ on $[\eta,\xi]\times \O$,
\begin{align*}
\Big|\int_{\Dom}
 |\wt{v}|^2 \TT  \Big( \int_{-h}^{\cdot} (\wt{v}\cdot\nabla_{\h})\T\,\dd \zeta\Big) \,\dd x\Big|
 &\lesssim_h   \big\||\wt{v}|^2\big\|_{L^3} \|\TT\|_{L^6} \|(\wt{v}\cdot\nabla_{\h})\T\|_{L^2} \\
 &\leq  \varepsilon_7 \big\||\wt{v}||\nabla\T|\big\|_{L^2}^2 +C_{\varepsilon_7} \big\||\wt{v}|^2\big\|_{L^3}^2 \|\T\|_{L^6}^2\\
 &\stackrel{(i)}{\leq}  
 \varepsilon_7 \big\||\wt{v}||\nabla\T|\big\|_{L^2}^2 +\varepsilon_7\big\||\wt{v}||\nabla \wt{v}|\big\|_{L^2}^{2}+ C_{\varepsilon_7} 
 (1+\|\T\|_{L^6}^4)
 \|\wt{v}\|_{L^4}^4,
\end{align*}
where in $(i)$ we used \eqref{eq:interpolation_L6_wt_v}. With similar arguments, we have
\begin{align*}
K_{2,2} &\leq \varepsilon_7 \big\||\wt{v}||\nabla_{\h}\TT|\big\|_{L^2}^2 +\varepsilon_7\big\||\wt{v}||\nabla \wt{v}|\big\|_{L^2}^{2}+ C_{\varepsilon_7} 
 (1+\|\T\|_{L^6}^4)
 \|\wt{v}\|_{L^4}^4,\\
K_{2,3} &\leq \varepsilon_7 \big\||\nabla \wt{v}||\T|\big\|_{L^2}^2 +\varepsilon_7\big\||\wt{v}||\nabla \wt{v}|\big\|_{L^2}^{2}+ C_{\varepsilon_7} 
 (1+\|\T\|_{L^6}^4)
 \|\wt{v}\|_{L^4}^4.
\end{align*}
Putting together the estimates of $(K_{2,j})_{j=1}^3$, one sees that $\E\int_{\eta}^{\xi}K_2\,\dd s$ is bounded by the right hand side of \eqref{eq:estimate_mixed_v_varphi}.

\subsubsection{Proof of \eqref{eq:estimate_mixed_v_varphi}} One can readily check that the terms $(\E\int_{\eta}^{\xi}K_i\,\dd s )_{i=3}^8$ appearing in \eqref{eq:estimate_mixed_v_varphi}  can be estimated by the right hand side of \eqref{eq:estimate_mixed_v_varphi} by  modifying the arguments of Subsection \ref{ss:estimate_tilde_v_T} slightly. 
Now \eqref{eq:estimate_mixed_v_varphi} follows the estimates of $(\E\int_{\eta}^{\xi}K_i\,\dd s)_{i=1}^8$ by taking the expected value in \eqref{eq:TT_v_equation_Ito}. 

\subsection{Estimate for $\big\||\T||\nabla_{\h} \TT|\big\|_{L^2_t L^2_x}^2$}
\label{ss:estimate_T_TT}
The aim of this subsection is to prove the following estimate: For all $\varepsilon_8\in (0,\infty)$,
\begin{align}
\label{eq:estimate_T_Theta}
&\E \int_{\eta}^{\xi} \Big\||\T||\nabla_{\h} \TT|\Big\|_{L^2}^2\,\dd s 
\leq C_{8,\varepsilon_8} \big(1+\E\|\T(\eta)\|_{L^4}^4\big)\\
\nonumber
&\qquad 
+C_{8,\varepsilon_8}\E\int_{\eta}^{\xi} L_s (1+\|\T\|_{L^4}^4 +\|\wt{v}\|_{L^4}^4) \,\dd s\\
\nonumber
&\qquad 
+\varepsilon_8 \Big( \E\int_{\eta}^{\xi}\Big\||\wt{v}|^2 |\nabla \T|\Big\|_{L^2}^2 \,\dd s +\E\int_{\eta}^{\xi}\Big\||\wt{v}|^2 |\nabla_{\h} \TT|\Big\|_{L^2}^2
\,\dd s+ \E\int_{\eta}^{\xi}\Big\||\nabla \wt{v}|^2 |\T|\Big\|_{L^2}^2\,\dd s\Big),
\end{align}
where $C_8,C_{8,\varepsilon_8}$ are constants independent of $(j,\eta,\xi,v_0,\T_0)$ and $C_8$ is also independent of $\varepsilon_7$. As above, $L_s$ is as in \eqref{eq:def_N_v_T}.

As before, to prove the main estimate, the idea is to apply the It\^o formula to a particular function. Here we employ the following 
\begin{equation}
\label{eq:functional_T_TT}
\T\mapsto \int_{\Dom} |\T|^2  \big|\int_{-h}^{\cdot} \T(\cdot,\zeta)\,\dd \zeta\big|^2\,\dd x.
\end{equation}
The proof of \eqref{eq:estimate_T_Theta} essentially follows the line of Subsections \ref{ss:estimate_tilde_v_T} and \ref{ss:estimate_tilde_v_varphi} except using the functional  \eqref{eq:functional_T_TT} instead of the one used there. Here we content ourselves in estimating the term appearing in the corresponding It\^o formula involving the convection term, i.e.
\begin{align*}
Q\stackrel{{\rm def}}{=}
\int_{\Dom} \Big(\TT^2\T (\wt{u}\cdot\nabla)\T+
\T^2 \TT \big(\int_{-h}^{\cdot} [(\wt{u}\cdot\nabla)\T]\,\dd \zeta\big)\Big)\,\dd x.
\end{align*}
Here, as above, $\wt{u}=(\wt{v},w(\wt{v}))$. 
Note that the analogue term with $\wt{u}$ replaced by $\overline{u}=(\overline{v},0)$ vanishes due to Lemma \ref{l:cancellation}.
Repeating the argument in Lemma \ref{l:cancellation_vv_TT}, we have 
\begin{align*}
Q =
\int_{\Dom} |\T|^2 \TT \Big( \int_{-h}^{\cdot} (\wt{v}\cdot\nabla_{\h})\T\,\dd \zeta- (\wt{v}\cdot\nabla_{\h})\TT 
+\int_{-h}^{\cdot} \div_{\h} \wt{v}\,\T\,\dd \zeta\Big)\,\dd x.
\end{align*}
Hence
\begin{align*}
Q \leq \varepsilon_8 \big( \||\wt{v}|^2 |\nabla \T|\|_{L^2}^2 +
\||\wt{v}|^2 |\nabla_{\h} \TT|\|_{L^2}^2
+ \||\nabla \wt{v}|^2 |\T|\|_{L^2}^2
\big)
+ C_{\varepsilon_8} \||\T|^2 |\TT|\|_{L^2}^2.
\end{align*}
It remains to estimate the last term in the previous inequality. Note that, using the H\"{o}lder inequality with exponents $(3,6)$, we have
$$
\||\T|^2 |\TT|\|_{L^2}^2\leq \|\T\|^4_{L^6} \|\TT\|_{L^6}^2 \lesssim_{h} \|\T\|_{L^6}^6.
$$
Since $\|\T\|_{L^6}^6\leq L$ by \eqref{eq:def_N_v_T}, one sees that $\E\int_{\eta}^{\xi}Q\,\dd s$ can be estimated by the right hand side of \eqref{eq:estimate_T_Theta}.

\subsection{Estimate for $\||\TT||\nabla_{\h} \TT|\|_{L^2_t L^2_x}$}  
\label{ss:estimate_TT}
The aim of this subsection is to prove the following estimate: For all $\varepsilon_9\in (0,\infty)$,
\begin{align}
\label{eq:estimate_Theta_Theta}
&\E \int_{\eta}^{\xi} \Big\||\TT||\nabla_{\h} \TT|\Big\|_{L^2}^2\,\dd s 
\leq C_{9,\varepsilon_9} \big(1+\E\|\T(\eta)\|_{L^4}^4\big)\\
\nonumber
&+C_{9,\varepsilon_9}\E\int_{\eta}^{\xi} L_s\big(1+\|\T\|_{L^4}^4+\|\wt{v}\|_{L^4}^4\big) \,\dd s\\
\nonumber
&
+\varepsilon_9 \Big( \E\int_{\eta}^{\xi}\Big\||\wt{v}|^2 |\nabla \T|\Big\|_{L^2}^2 \,\dd s +\E\int_{\eta}^{\xi}\Big\||\wt{v}|^2 |\nabla_{\h} \TT|\Big\|_{L^2}^2
\,\dd s+ \E\int_{\eta}^{\xi}\Big\||\nabla \wt{v}|^2 |\T|\Big\|_{L^2}^2\,\dd s\Big),
\end{align}
where $C_9,C_{9,\varepsilon_9}$  are constants independent of $(j,\eta,\xi,v_0,\T_0)$ and $C_9$ is also independent of $\varepsilon_9$. Finally, $L_s$ is as in \eqref{eq:def_N_v_T}.

Here we apply the It\^o formula to the functional 
 $\T\mapsto\big\|  \int_{-h}^{\cdot} \T(\cdot,\zeta)\,\dd \zeta \big\|_{L^4}^4$. 
As in Subsection \ref{ss:estimate_T_TT}, we content ourselves to estimate the term coming from the convective term:
\begin{align*}
Q_0
&\stackrel{{\rm def}}{=}\int_{\Dom} \TT^3 \big(\int_{-h}^{\cdot} (\wt{u}\cdot\nabla )\T\,\dd \zeta\big)\,\dd x\\
&= \int_{\Dom} \TT^3 \Big[\big(\int_{-h}^{\cdot} (\wt{v}\cdot\nabla )\T\,\dd \zeta\big)-
(\wt{v}\cdot\nabla_{\h})\TT 
+ \int_{-h}^{\cdot} \div_{\h} \wt{v}\,\T\,\dd \zeta\Big]\,\dd x,
\end{align*}
where the last equality follows from the argument of Lemma \ref{l:cancellation_vv_TT}.
Hence, as in the previous subsection, one can readily check that $\E\int_{\eta}^{\xi}Q_0\,\dd s$ can be estimated by the right-hand side of \eqref{eq:estimate_Theta_Theta}.

\subsection{Proof of Lemma \ref{l:main_estimate}}
\label{ss:proof_lemma_main_estimate}
Here we conclude the proof of Lemma \ref{l:main_estimate} using the estimates proven in Subsections \ref{ss:estimate_overline_v}-\ref{ss:estimate_TT}. As explained in Subsection \ref{ss:preparation}, it remains to prove \eqref{eq:main_estimate_gronwall} with $c_0$ is independent of $(j,\eta,\xi,v_0,\T_0)$. 
Now the idea is to multiply the estimates of Subsections \ref{ss:estimate_overline_v}-\ref{ss:estimate_TT} by suitable positive constants $(\alpha_i)_{i=1}^{9}$ and then to sum up the resulting estimates. Then we choose $\alpha_i$'s such that the latter estimate is equivalent to \eqref{eq:main_estimate_gronwall}. 

To highlight the core of the argument we denote the quantities appearing in the estimates of Subsections \ref{ss:estimate_overline_v}-\ref{ss:estimate_TT} as follows:
\begin{align*}\boxed{\star}&\stackrel{{\rm def}}{=}
\sum_{1\leq i\leq 9} \boxed{i},  &
\boxed{1}&\stackrel{{\rm def}}{=}\E\int_{\eta}^{\xi} \|\overline{v}\|_{H^2}^2\,\dd s,\\
\boxed{2}
&\stackrel{{\rm def}}{=} \E\int_{\eta}^{\xi} \|\wh{\T}\|_{H^2}^2\,\dd s,&
\boxed{3}
&\stackrel{{\rm def}}{=}\E\int_{\eta}^{\xi} \|\nabla \partial_3 v\|_{L^2}^2\,\dd s,\\
\boxed{4}
&\stackrel{{\rm def}}{=}\E\int_{\eta}^{\xi} \|\nabla \partial_3 \T\|_{L^2}^2\,\dd s,& 
\boxed{5}
&\stackrel{{\rm def}}{=}\E\int_{\eta}^{\xi}\Big\||\wt{v}||\nabla \wt{v}|\Big\|_{L^2}^2\,\dd s, \\
\boxed{6}
&\stackrel{{\rm def}}{=}
\E\int_{\eta}^{\xi} \Big(\Big\| |\wt{v}||\nabla \T| \Big\|_{L^2}^2+ \Big\| |\T||\nabla \wt{v}| \Big\|_{L^2}^2\Big)\,\dd s,  &
\boxed{7}
&\stackrel{{\rm def}}{=}
\E\int_{\eta}^{\xi}\Big(\Big\||\wt{v}||\nabla_{\h} \TT|\Big\|_{L^2}^2+ \Big\| | \TT ||\nabla \wt{v}| \Big\|_{L^2}^2\Big)\,\dd s,\\
\boxed{8}
&\stackrel{{\rm def}}{=}
\E\int_{\eta}^{\xi} \Big\| |\T||\nabla_{\h} \TT | \Big\|_{L^2}^2\,\dd s,&
\boxed{9}
&\stackrel{{\rm def}}{=}
\E\int_{\eta}^{\xi} \Big\| | \TT ||\nabla_{\h} \TT | \Big\|_{L^2}^2\,\dd s,\\
\II&\stackrel{{\rm def}}{=}
 \E\int_{\eta}^{\xi}L_t(1+ X_t+ \|\wh{\T}(t)\|_{H^1(\Tor^2)}^2)\,\dd s, &  \III &\stackrel{{\rm def}}{=} 1+ \E X_t+ \E\|\wh{\T}(t)\|_{H^1(\Tor^2)}^2+\E\|\T(\eta)\|_{L^4}^4,
\end{align*}
where $X_t$ is as in Lemma \ref{l:main_estimate}.
Comparing the estimates of Subsections \ref{ss:estimate_overline_v}-\ref{ss:estimate_TT}  and \eqref{eq:main_estimate_gronwall}, one sees that the energy terms $\boxed{i}$ for $i\in \{1,\dots,9\}$ are the one we would like to absorb. 

It will be proved conveniently later to derive a consequence of \eqref{eq:estimate_mixed_v_varphi} and \eqref{eq:estimate_Theta_Theta}. Indeed, we would like to have a constant in front of the last term on the right-hand side of \eqref{eq:estimate_mixed_v_varphi} which does not blow-up as $\varepsilon_7\downarrow 0$. To this end, using the estimate \eqref{eq:estimate_Theta_Theta} with $\varepsilon_9=\varepsilon_7/ (2 (C_{7,\varepsilon_7}\vee 1))$ in \eqref{eq:estimate_mixed_v_varphi} with $\varepsilon_7$ replaced by $\varepsilon_7/2$, we get
\begin{align}
\label{eq:estimate_mixed_v_varphi_revised}
\boxed{7}
\leq 
\varepsilon_7 (\boxed{2} +\boxed{3}+\boxed{4}+\boxed{5}+ \boxed{6}+\boxed{7})+C_{7,\varepsilon_7}' (\III+\II) ,
\end{align}
where $C_{7,\varepsilon_7}'$ is a constant independent of $(j,\eta,\xi,v_0,\T_0)$.

In the following we apply the estimates of Subsections \ref{ss:estimate_overline_v}-\ref{ss:estimate_wt_v_L_4} and \ref{ss:estimate_T_TT}-\ref{ss:estimate_TT}  as well as \eqref{eq:estimate_mixed_v_varphi_revised} with 
$$
\varepsilon_i \equiv \varepsilon\in (0,\infty) \text{ for all }i\in \{1,\dots,9\},
$$
where $\varepsilon$ is chosen below. 
Let $(C_{i,\varepsilon}, C_{i})_{i=1}^9$ be the constants introduced in Subsections \ref{ss:estimate_overline_v}-\ref{ss:estimate_wt_v_L_4} and \ref{ss:estimate_T_TT}-\ref{ss:estimate_TT} and set
$$
C_{0,\varepsilon}\stackrel{\rm def}{=}\max_{1\leq i\leq 9, i\neq 7} C_{i,\varepsilon} \vee C_{7,\varepsilon}' \ \ \text{ and }\ \ 
C_{0}\stackrel{\rm def}{=}\max_{1\leq i\leq 9,  i\neq 7} C_{i}.
$$ 
As before, the constants 
$C_{0,\varepsilon},C_{0}$ are independent of $(j,\eta,\xi,v_0,\T_0)$ and $C_0$ is also independent of $\varepsilon$. 
With the above setting and notation, the estimates of Subsections \ref{ss:estimate_overline_v}-\ref{ss:estimate_wt_v_L_4} and \ref{ss:estimate_T_TT}-\ref{ss:estimate_TT}  as well as \eqref{eq:estimate_mixed_v_varphi_revised} imply:
\begin{align*}
\E \Big[\sup_{s\in [\eta,\xi]} \|\overline{v}(s)\|^2_{H^1}\Big]+ \boxed{1}
&\leq C_0(\boxed{3}+\boxed{5})+C_{0,\varepsilon}(\III+\II)+ \varepsilon \boxed{\star}, 
&\text{consequence of \eqref{eq:estimate_v_overline}},&\\
\E \Big[\sup_{s\in [\eta,\xi]} \|\wh{\T}(s)\|^2_{H^1}\Big]+
\boxed{2}
&\leq C_0 (\boxed{4}+ \boxed{6}+ \boxed{7})+C_{0,\varepsilon}(\III+\II)+ \varepsilon \boxed{\star},
&\text{consequence of \eqref{eq:smr_estimate_wh_T}},&\\
\E \Big[\sup_{s\in [\eta,\xi]} \|\partial_3 v(s)\|^2_{L^2}\Big]+
\boxed{3}
&\leq C_0\boxed{5}+ C_{0,\varepsilon}(\III+\II)+ \varepsilon \boxed{\star},
&\text{consequence of \eqref{eq:estimate_v_3}},&\\
\E \Big[\sup_{s\in [\eta,\xi]} \|\partial_3 \T(s)\|^2_{L^2}\Big]
+
\boxed{4}
&\leq C_0\boxed{6}+ C_{0,\varepsilon}(\III+\II)+ \varepsilon \boxed{\star},
&\text{consequence of \eqref{eq:estimate_T_3}},&\\
\E \Big[\sup_{s\in [\eta,\xi]} \|\wt{ v}(s)\|^4_{L^4}\Big]+
\boxed{5}
&\leq C_0\boxed{7} + C_{0,\varepsilon}(\III+\II)+ \varepsilon \boxed{\star},  
&\text{consequence of \eqref{eq:estimate_v_wt}},&\\
\boxed{6}
&\leq C_0 \boxed{8}+C_{0,\varepsilon}(\III+\II)+ \varepsilon \boxed{\star},
&\text{consequence of \eqref{eq:estimate_mixed_v_T}},&\\
\boxed{7}
&\leq C_{0,\varepsilon}(\III+\II)+ \varepsilon \boxed{\star},  
&\text{consequence of \eqref{eq:estimate_mixed_v_varphi_revised}},&\\
\boxed{8}
&\leq C_{0,\varepsilon}(\III+\II)+ \varepsilon \boxed{\star},
&\text{consequence of \eqref{eq:estimate_T_Theta}},&\\
\boxed{9}&\leq C_{0,\varepsilon}(\III+\II)+ \varepsilon \boxed{\star}, 
&\text{consequence of \eqref{eq:estimate_Theta_Theta}}.&
\end{align*}
In the above estimates, $\varepsilon\in (0,\infty)$ is a free parameter which will be fixed later.  
Multiplying the above estimates by $\alpha_i\in [1,\infty)$ and then summing them up, we have:
\begin{equation}
\label{eq:c_i_summed_final_lemma}
\E\Big[\sup_{s\in [\eta,\xi]} X_s\Big]
+
\sum_{1\leq i\leq 9} c_i \boxed{i}\leq C_{0,\varepsilon} \alpha(\III+\II)+\alpha \varepsilon \boxed{\star}  ,
\end{equation}
where we used the definition of $X_t$ in Lemma \ref{l:main_estimate} and we set $\alpha\stackrel{\rm def}{=} \sum_{1\leq i\leq 9}\alpha_i$,
\begin{align*}
c_1 &=\alpha_1,  & c_2&=\alpha_2,\\ 
c_3 &=\alpha_3-C_0 \alpha_1,  & c_4&=\alpha_4-C_0 \alpha_2,\\
c_5 &=\alpha_5-C_0( \alpha_1+\alpha_3),  & c_6&=\alpha_6-C_0 ( \alpha_2+\alpha_4),\\ 
c_7 &=\alpha_7-C_0( \alpha_2+\alpha_5),   & c_8&=\alpha_8-C_0\alpha_6,\\ 
c_9&=\alpha_9. & & 
\end{align*}
Note that the terms $\boxed{i}$ can be absorbed on the right-hand side of \eqref{eq:c_i_summed_final_lemma} as $\boxed{i}\lesssim j$ for all $1\leq i\leq 9$.

Next, we show that there exists a choice of $(\alpha_i)_{i=1}^9$ such that $c_i\equiv 1$. To see this, one can split the argument into several steps as follows:
\begin{itemize}
\item Choose $\alpha_1=\alpha_2=1$ and $\alpha_9=1$.
\item Choose $\alpha_3=\alpha_4=C_0+1$.
\item Choose $\alpha_5=\alpha_6=C_0(C_0+2)+1$.
\item Choose $\alpha_7=C_0(\alpha_2+\alpha_5)+1$ and $\alpha_8=C_0\alpha_6+1$.
\end{itemize}
With the above choices we have $c_i\equiv 1$ and $\min_{1\leq i\leq 9}\alpha_i\geq 1$. Thus \eqref{eq:c_i_summed_final_lemma} yields \eqref{eq:main_estimate_gronwall} choosing 
$
\varepsilon  = (2C_0\alpha)^{-1}
$
and recalling that $C_0$ is independent of $(j,\eta,\xi,v_0,\T_0)$.

\section{Proof of Proposition \ref{prop:energy_estimate_primitive_strong_strong} }
\label{s:proof_energy_estimate_conclusion}
To prove Proposition \ref{prop:energy_estimate_primitive_strong_strong} we collect some useful facts.  
Let $X_t,Y_t$ be as in Lemma \ref{l:main_estimate}.
For notational convenience, we set
\begin{align*}
\xx_s\stackrel{{\rm def}}{=}1+\|v(s)\|_{L^2}^2+\|\T(s)\|_{L^2}^2+X_s, \quad \text{ and } \quad  
\yy_s \stackrel{{\rm def}}{=}1+\|v(s)\|_{H^1}^2+\|\T(s)\|_{H^1}^2+Y_s.
\end{align*}
By Lemmas \ref{l:basic_estimates} and \ref{l:main_estimate}, we have, for some constant  $c_{0,T}$ independent of $(v_0,\T_0)$, for all $\g>1$,
\begin{equation}
\label{eq:distribution_of_mathcal_X_Y}
\P\Big(\sup_{s\in[0,\tau\wedge T)}\xx_s+\int_0^{\tau\wedge T} \yy_s\,\dd s\geq \g\Big)
 \leq c_{0,T}\frac{(1+\E\|\Xi\|_{L^2(0,T)}^2)+\E\|v_0\|_{H^1}^4+\E\|\T_0\|_{H^1}^4}{\log\log(\g)}.
\end{equation}

\begin{proof}[Proof of Proposition \ref{prop:energy_estimate_primitive_strong_strong}]
Let $(\tau_j)_{j\geq 1}$ be as in \eqref{eq:def_tau_j}. As above the estimate is reduced to an application of the stochastic Gronwall lemma \cite[Lemma A.1]{AV_variational}. To simplify the notation we write $U=(v,\T)$ and $(A,B,F,G)$ are as in \eqref{eq:def_A}-\eqref{eq:def_G}. Recall that $H=\Hs^1\times H^1$ and $V=\Hs_{\n}^2\times \Hr^2$.
We claim that there exists $C_0>0$ independent of $(v_0,\T_0)$ such that, for all $j\geq 1$ and all stopping times $(\eta,\xi)$ satisfying $0\leq \eta\leq \xi\leq \tau_j$,
\begin{align}
\label{eq:claim_proposition_estimate_H2}
\E \sup_{s\in [\eta,\xi]} \| U(s)\|_{H}^2
&+ 
\E \int_{\eta}^{\xi} \| U(s)\|_{V}^2\,\dd s \\
\nonumber
&\leq  C_{0} \Big[ 1 +\E\|U(\eta)\|_{H}^2+\E\int_{\eta}^{\xi}
(1+\xx_s^2)
\yy_s(1+\|U\|_{H}^2)\,\dd s\Big] .
\end{align}

\emph{Step 1: Sufficiency of \eqref{eq:claim_proposition_estimate_H2}}. Recall that $\lim_{j\to \infty}\tau_j= \tau\wedge T$ a.s.\ by  \eqref{eq:def_tau_j} and $U\in  C([0,\tau);H)\cap L^2_{\loc}([0,\tau);V)$ a.s.\ (recall that $(H,V)$ are as in \eqref{eq:def_HV}). The stochastic Gronwall lemma \cite[Lemma A.1]{AV_variational}, \eqref{eq:distribution_of_mathcal_X_Y} and the fact that $c_0$ is independent of  $(j,U_0)$ ensure that, for all $R,\g>1$, 
\begin{align*}
\P\Big(\sup_{s\in [0,\tau\wedge T)}\|U(s)\|_H^2 \ + &\int_{0}^{\tau\wedge T}\|U(s)\|^2_V\,\dd s\geq \g\Big)\\
&\leq \Big( \frac{c_{T}}{\g}e^{c_{T} R} + \frac{c_{T}}{\log\log(R)}\Big)
(1+\E\|\Xi\|_{L^2(\R_+;L^2)}^2+\E\|U_0\|_{H}^4).
\end{align*}
Here $c_{T}$ is a constant which depends only on $(c_0,c_{0,T})$. 
Choosing $R=R(\g)=\frac{1}{c_{T}}\log(\frac{\g}{\log\g})$ for $\g$ large, one obtain the estimates claimed in Proposition \ref{prop:energy_estimate_primitive_strong_strong}.

\emph{Step 2: Proof of  \eqref{eq:claim_proposition_estimate_H2}}. Reasoning as in the proof of Proposition \ref{prop:Lip_estimate_continuity}, by Lemma \ref{l:smr} and \cite[Proposition 3.9]{AV19_QSEE_1} there exists $C_0>0$,  independent of $U_0,U_0'$, such that for all stopping times $(\eta,\xi)$ satisfying $0\leq \eta\leq \xi\leq T$ a.s.\ one has
\begin{equation}
\begin{aligned}
\label{eq:smr_strong_theta_inequality_energy_estimate}
\E \sup_{s\in [\eta,\xi]} \|U(s)\|_H^2
+ 
\E \int_{\eta}^{\xi} \|U(s)\|_{V}^2 \,\dd s 
 \leq C_0 \Big[ \E\|U(\eta)\|_{H}^2 +\sum_{j=1}^6 I_{j}\Big],
\end{aligned}
\end{equation}
where 
\begin{align*}
I_{1}&\stackrel{{\rm def}}{=} \E\int_{\eta}^{\xi} \|(v\cdot\nabla_{\h})v\|_{L^2}^2\,\dd s, &
I_{2}&\stackrel{{\rm def}}{=}\E\int_{\eta}^{\xi} \|w(v)\partial_3 v\|_{L^2}^2\,\dd s,\\
I_{3}&\stackrel{{\rm def}}{=}\E\int_{\eta}^{\xi} \|(v\cdot\nabla_{\h}) \T\|_{L^2}^2\,\dd s,&
I_{4}&\stackrel{{\rm def}}{=}\E\int_{\eta}^{\xi} \|w(v)\partial_3 \T\|_{L^2}^2\,\dd s,\\
I_{5}&\stackrel{{\rm def}}{=}\E\int_{\eta}^{\xi} \|\fv (\cdot,v,\T,\nabla v,\nabla \T)\|_{L^2}^2\,\dd s,&
I_{6}&\stackrel{{\rm def}}{=}\E\int_{\eta}^{\xi} \|\ft (\cdot,v,\T,\nabla v,\nabla \T)\|_{L^2}^2\,\dd s,\\
I_{7}&\stackrel{{\rm def}}{=}\E\int_{\eta}^{\xi} \|\gv (\cdot,v,\T)\|_{H^1(\ell^2)}^2\,\dd s,&
I_{8}&\stackrel{{\rm def}}{=}\E\int_{\eta}^{\xi} \|\gt (\cdot,v,\T)\|_{H^1(\ell^2)}^2\,\dd s.
\end{align*}
By Assumption \ref{ass:well_posedness_primitive_double_strong}\eqref{it:nonlinearities_strong_strong} and \eqref{eq:def_y} (or, more generally, the condition in Remark \ref{r:weaken_assumption_nonlinearities}\eqref{it:global_sublinear}), we have
\begin{align*}
\sum_{5\leq j\leq 8}I_j 
&\lesssim (1+\E\|\Xi\|_{L^2(0,T)}^2+ \E\|v\|_{L^2(0,T;L^2)}^2+\E\|\T\|_{L^2(0,T;L^2)}^2)\\
&\lesssim 
1+\E\|\Xi\|_{L^2(0,T)}^2+ \E\int_{\eta}^{\xi}\yy_s\,\dd s.
\end{align*}
To estimate the remaining terms, let us recall the following useful estimate:
\begin{equation}
\label{eq:estimate_w_v}
\|w(v)\|_{L^{\infty}(-h,0;L^4(\Tor^2))}\lesssim_h \|v\|_{L^2(-h,0;L^4(\Tor^2))}\lesssim \|v\|_{H^1}^{1/2}\|v\|_{H^2}^{1/2},
\end{equation}
where the last inequality follows from \eqref{eq:interpolation_inequality_L4}.
The terms $I_1$ and $I_2$ can be estimated as in the proof of \cite[Proposition 5.1]{Primitive1}. The arguments given there show:
\begin{equation}
\label{eq:I_1_I_2_final_estimate}
I_1 + I_2 \leq \frac{1}{4C_0} \E\int_{\eta}^{\xi} \|v(s)\|_{H^2}^2\,\dd s+ C_1 
\E\int_{\eta}^{\xi}(1+ \xx_s^2 )\yy_s(1+ \|v(s)\|_{H^1}^2 )\,\dd s,
\end{equation}
where $C_0\geq 1$ is as in \eqref{eq:smr_strong_theta_inequality_energy_estimate} and $C_1$ is independent of $(j,\eta,\xi,v_0,\T_0)$. 
However, the above estimate can also be obtained by (slightly) modifying the argument below where we estimate $I_3$ and $I_4$.

To estimate $I_{3}$, note that, $I_{3}\leq 2( I_{3,1}+I_{3,2})$ where
$$
I_{3,1}\stackrel{{\rm def}}{=}
\E\int_{\eta}^{\xi} \|(\overline{v}\cdot\nabla_{\h}) \T\|_{L^2}^2\,\dd s
\ \  \text{ and }\ \ 
I_{3,2}
\stackrel{{\rm def}}{=}\E\int_{\eta}^{\xi} \|(\wt{v}\cdot\nabla_{\h}) \T\|_{L^2}^2\,\dd s,
$$
since $v=\overline{v}+\wt{v}$.
Note that 
$I_{3,2}\leq \E\int_{\eta}^{\xi} \big\||\wt{v}||\nabla \T|\big\|_{L^2}^2\,\dd s\leq \E\int_{\eta}^{\xi} \yy_{s}\,\dd s $  and
\begin{align*}
I_{3,2}
&\lesssim\E\int_{\eta}^{\xi} \|\overline{v}(s)\|_{L^6}^2 \|\nabla\T(s)\|_{L^3}^2\,\dd s \stackrel{(i)}{\lesssim} 
\E\Big[\int_{\eta}^{\xi}   \|\overline{v}(s)\|_{H^1}^2 \|\nabla\T(s)\|_{L^3}^2\,\dd s\Big]\\
&\lesssim \E\Big[\int_{\eta}^{\xi} \xx_s \|\T(s)\|_{H^{1,3}}^2\,\dd s \Big] 
\stackrel{(ii)}{\lesssim}  \E\int_{\eta}^{\xi} \xx_s\|\T(s)\|_{H^1}\|\T(s)\|_{H^2}\,\dd s \\
&\leq \frac{1}{ 8 C_0}
 \E\int_{\eta}^{\xi} \|\T(s)\|_{H^2}^2\,\dd s + C
 \E\int_{\eta}^{\xi} \xx_s^2\|\T(s)\|^2_{H^1}\,\dd s  ,
\end{align*}
where in $(i)$ we used the Sobolev embedding $H^1\embed L^6$ and in $(ii)$ \eqref{eq:interpolation_inequality_L3}. Here $C$ depends only on $C_0$. In particular $C$ is independent of $(j,\eta,\xi,v_0,\T_0)$.

Finally we estimate $I_{4}$. The H\"{o}lder inequality, \eqref{eq:interpolation_inequality_L4} and \eqref{eq:estimate_w_v} yield
\begin{align*}
I_{4}
&\leq \E\int_{\eta}^{\xi}
\|w(v)\|_{L^{\infty}(-h,0;L^4(\Tor^2))}^2
\|\partial_3\T\|_{L^2(-h,0;L^4(\Tor^2))}^2\,\dd s \\
&\lesssim  \E\int_{\eta}^{\xi}
\|v\|_{H^1}\|v\|_{H^2}
\|\partial_3\T\|_{L^2}\| \partial_3\T\|_{H^1}\,\dd s \\
&\leq \frac{1}{8C_0}  \E\int_{\eta}^{\xi} \| v\|_{H^2}^2\,\dd s 
+C_{2} \E\int_{\eta}^{\xi}
\xx_s\yy_s
\| v(s)\|_{H^1}^2\,\dd s.
\end{align*}
Hence, for some $C_2$ is independent of $(j,\eta,\xi,v_0,\T_0)$,
\begin{equation}
\label{eq:I_3_I_4_final_estimate}
\begin{aligned}
I_3 + I_4 
&\leq \frac{1}{4C_0} \E\int_{\eta}^{\xi} (\|v(s)\|_{H^2}^2+\|\T(s)\|_{H^2}^2)\,\dd s\\
&+ C_2
\E\int_{\eta}^{\xi}(1+ \xx_s^2 )\yy_s(1+ \|v(s)\|_{H^1}^2+\|\T(s)\|_{H^1}^2 )\,\dd s.
\end{aligned}
\end{equation}
Using the estimates \eqref{eq:I_1_I_2_final_estimate}-\eqref{eq:I_3_I_4_final_estimate} in \eqref{eq:smr_strong_theta_inequality_energy_estimate}, one gets \eqref{eq:claim_proposition_estimate_H2} as desired.
\end{proof}

\section{Stratonovich formulation}
\label{s:Stratonovich}
In this section, we analyze the case of primitive equations \eqref{eq:primitive_full} where the noise is understood in the  Stratonovich formulation. 
%%%%%
More precisely, following the reformulation of \eqref{eq:primitive_full} as \eqref{eq:primitive_v_T_pressure} with $\hp=\tp=0$ (cf.\ \eqref{eq:primitive_v_T_pressure_1_projected}), here we consider
\begin{subequations}
	\label{eq:primitive_stratonovich}
\begin{alignat}{3}
	\label{eq:primitive_stratonovich_1}
\begin{split}
\dd  v-\Delta v\,\dd t & =  \p\Big[ -(v\cdot \nabla_{\h}) v- \w(v)\partial_3 v  \\
&\quad\quad\ \ 
 -\nabla_{\h}\int_{-h}^{\cdot} (\kone(\cdot,\zeta)\T(\cdot,\zeta))\,\dd \zeta + \fv (v,\T,\nabla v,\nabla \T) \Big]
\, \dd t \\
& 
+\sum_{n\geq 1}\p \Big[(\phi_{n}\cdot\nabla) v+\int_{-h}^{\cdot}\nabla_{\h}(\ktwon(\cdot,\zeta) \T(\cdot,\zeta))\,\dd \zeta  \Big] \circ \dd \beta_t^n,
\end{split}\\
	\label{eq:primitive_stratonovich_2}
\begin{split}
\dd  \Temp-\Delta \T\,\dd t &=\Big[(v\cdot\nabla_{\h}) \T -\w(v) \partial_3 \T+ \ft(v,\T,\nabla v,\nabla \T ) \Big]\, \dd t
+\sum_{n\geq 1}(\psi_n\cdot \nabla) \Temp\circ \dd \beta_t^n,
\end{split}\\
	\label{eq:primitive_stratonovich_3}
v(0,\cdot)&=v_0,\qquad \T(0,\cdot)=\T_0
\end{alignat}
\end{subequations}
on $\Dom\stackrel{{\rm def}}{=}\Tor^2\times(-h,0)$,
where $\circ$ and $\p$ denote the Stratonovich integration and the hydrostatic Helmholtz projection; see e.g.\ \cite{Gar09} and Subsection \ref{ss:reformulation}, respectively. As in the previous sections, the above problem is complemented by the following boundary conditions
\begin{subequations}
\label{eq:primitive_stratonovich_BC}
\begin{alignat}{3}
	\label{eq: primitive_strong_BC_1}
		\partial_3 v (\cdot,-h)=\partial_3 v(\cdot,0)&=0\quad \text{ on }\Tor^2,\\
	\label{eq: primitive_strong_BC_2}
		\partial_3 \T(\cdot,-h)= \partial_3 \T(\cdot,0)+\alpha \T(\cdot,0)&=0\quad \text{ on }\Tor^2.&
\end{alignat}
\end{subequations}
For the sake of simplicity, in contrast to the previous parts of this manuscript, in \eqref{eq:primitive_stratonovich}-\eqref{eq:primitive_stratonovich_BC} we do not consider lower order terms in the stochastic perturbation keeping only the transport and gradient type terms which are the most relevant from an application point of view. The reader is referred to \cite{BiFla20,MR01,MR04} and Section \ref{s:physical_derivation} for physical motivations of the transport noise terms and of the $\ktwon$-term, respectively. Last but not least, lower-order terms are mathematically easier to deal with. We leave the details to the interested reader.

Let us mention that, in applications, the Stratonovich formulation of the noise is often preferred to the It\^o one, as the former is closer to numerical simulations due to Wong-Zakai type results \cite{Fla15_book} and to two scale type arguments \cite{DP22_two_scale,FlaPa21}.

As common in SPDEs, and as in \cite[Section 8]{Primitive1}, our approach is to view the Stratonovich noise in  \eqref{eq:primitive_stratonovich}-\eqref{eq:primitive_stratonovich_BC} as an It\^{o} one plus additional correction terms. Therefore, as announced at the beginning of Section \ref{s:strong_strong}, while rephrasing \eqref{eq:primitive_stratonovich}-\eqref{eq:primitive_stratonovich_BC} in a system of It\^{o} SPDEs, the terms $\partial_{\g} \wt{p}$ and $(\tp\cdot\nabla) \T$ in \eqref{eq:primitive_3}-\eqref{eq:primitive_4} will appear naturally.
The same also applies to inhomogeneous viscosity and/or conductivity discussed in Remark \ref{r:inhomogeneous_viscosity_conductivity}. As we will see below, the term $(\tp\cdot\nabla) \T$ is a consequence of the Stratonovich formulation and the temperature-dependent turbulent pressures, cf.\ \eqref{eq:pi_psi_etc} below.
Instead the term $\partial_{\hp} \wt{p}$ depends only on the presence of the transport noise in \eqref{eq:primitive_stratonovich_1}.

To study \eqref{eq:primitive_stratonovich}-\eqref{eq:primitive_stratonovich_BC} we need the following assumptions. 

\begin{assumption} 
\label{ass:stratonovich}
There exist $M,\delta>0$ for which the following hold.
\begin{enumerate}[{\rm(1)}]
\item\label{it:stratonovich_1}
For all $j\in \{1,2,3\}$ and $n \geq 1$, the mappings
$$
\phi_n^j,\ \psi_n^j ,  \ \kone,   \ \ktwon: \O\times \Dom\to \R \ \text{ are \ $\F_0\otimes \mathscr{B}(\Dom)$-measurable}.
$$
%%%%%
\item\label{it:stratonovich_2} a.s.\ for all $n\geq 1$, $x=(x_{\h},\z)\in \Tor^2\times (-h,0)= \Dom$ and $j,k\in \{1,2\}$
\begin{align*}
\phi_n^j(x),\psi^j_n (x)\text{ and } \ktwon(x)\text{ are independent of $\z$}.
\end{align*}
%%%%
\item\label{it:stratonovich_3} 
{\em (Regularity)}
A.s.\ for all $j,k\in \{1,2,3\}$ and $i\in \{1,2\}$,
\begin{align*}
\Big\|\Big(\sum_{n\geq 1}| \phi^j_n|^2\Big)^{1/2} \Big\|_{L^{3+\delta}(\Dom)}+
\Big\|\Big(\sum_{n\geq 1}|\partial_k \phi^j_n|^2\Big)^{1/2} \Big\|_{L^{3+\delta}(\Dom)} \leq M,&\\
\Big\|\Big(\sum_{n\geq 1}| \psi^j_n|^2\Big)^{1/2} \Big\|_{L^{3+ \delta}(\Dom)}+
\Big\|\Big(\sum_{n\geq 1}|\partial_k \psi^j_n|^2\Big)^{1/2} \Big\|_{L^{3+ \delta}(\Dom)} \leq M,&\\
\| \kone(t,\cdot) \|_{L^{\infty}(\Tor^2;L^2(-h,0))} +\|\partial_i \kone(t,\cdot) \|_{L^{2+\delta}(\Tor^2;L^2(-h,0))} \leq M,&\\
\| (\ktwon(t,\cdot))_{n\geq 1}\|_{H^{2,2+\delta}(\Tor^2\ell^2)}\leq M.&
\end{align*}
\item\label{it:stratonovich_4} a.s.\ for all $n\geq 1$ and $x_{\h}\in \Tor^2$,
$$
\phi^3(x_{\h},0)
=
\phi^3(x_{\h},-h)=0.
$$
\end{enumerate}
\end{assumption}

Next, under Assumption \ref{ass:stratonovich} we (formally) rewrite \eqref{eq:primitive_stratonovich}-\eqref{eq:primitive_stratonovich_BC} in the form \eqref{eq:primitive_v_T_pressure} with suitable $(\tp,\hp)$. 
As usual, for two stochastic processes $(X_t,Y_t)$, we denote by $[X,Y]_t$ their joint quadratic variation at time $t$. 
By \eqref{eq:primitive_stratonovich_2}, at least formally we have $ [\T,\beta^n]_{\cdot}=
\int_0^t (\psi_n \cdot\nabla) \T\,\dd s$. Moreover, formally from \cite[Thereom 2.3.5, p.\ 60]{Kunita97}, 
\begin{align}
\nonumber
\int_0^t (\psi_n\cdot \nabla) \Temp\circ \dd \beta_s^n
&= 
\int_0^t 
(\psi_n\cdot \nabla) \Temp\, \dd \beta_s^n
+\frac{1}{2} (\psi_n \cdot\nabla )[\T,\beta^n]_t \\
&=
\label{eq:correction_T}
\int_0^t 
(\psi_n\cdot \nabla) \Temp\, \dd \beta_s^n
+ \int_0^t \dop_{\psi} \T\,\dd s,
\end{align}
where 
\begin{align*}
\dop_{\psi} \T
&\stackrel{{\rm def}}{=}
\frac{1}{2}\sum_{n\geq 1}(\psi_n \cdot\nabla )[(\psi_n \cdot\nabla)\T]\\
&=
\frac{1}{2}\sum_{n\geq 1}\sum_{1\leq i,j\leq 3}\Big(\psi^j_n \psi^i_n \partial_{i,j}^2 \T +\psi^i_n ( \partial_i \phi^j_n ) \partial_j\T\Big).
\end{align*}
%%%
The reformulation of the Stratonovich noise in \eqref{eq:primitive_stratonovich_1} is computationally more involved.
To shorten the notation, similar to Subsection \ref{ss:proof_local}, we set 
$$
\op f(x)\stackrel{{\rm def}}{=} \nabla_{\h}\int_{-h}^{x_3} \T(x_{\h},\zeta)\,\dd \zeta,
$$
 where $ x=(x_{\h},x_3)\in \Tor^2\times (-h,0)=\Dom$.

Note that by Assumption \ref{ass:stratonovich}\eqref{it:stratonovich_2} and the linearity of $\op$, at least formally,
\begin{align}
\label{eq:ito_correction_v}
\int_0^t\p\big[
 (\phi_n\cdot\nabla) v + \op(\ktwon \T)\big]\circ \dd \beta_s^n
\nonumber
&=\int_0^t 
\p\big[
(\phi_n\cdot\nabla) v + \op(\ktwon \T)\big]\, \dd \beta_s^n\\
&+
\frac{1}{2} 
\underbrace{\p\big( (\phi_n\cdot\nabla) [v,\beta^n]_t\big)}_{\cor_{v}\stackrel{{\rm def}}{=}} + 
\frac{1}{2}\underbrace{\p\Big( \op\big(\ktwon [\T,\beta^n]_t\big)\Big)}_{\cor_{\T}\stackrel{{\rm def}}{=}} .
\end{align}
%\\
%\stackrel{(i)}{=}\int_0^t 
%\p\Big[
%(\phi_n\cdot\nabla) v + \op_{n} \T\Big]\, d\beta_s^n
%+
%\frac{1}{2}
%\p\Big[ (\phi_n\cdot\nabla) [v,\beta^n]_t\Big] + \frac{1}{2}\int_0^t \op_n [(\psi_n\cdot\nabla)\T]\,ds
%
Next we formally compute the corrective terms $(\cor_{v},\cor_{\T})$. We begin by taking a look at $\cor_{\T}$. 
Recall that $ [\T,\beta^n]_{\cdot}=\int_0^t 
 (\psi_n \cdot\nabla) \T\,ds$ by \eqref{eq:primitive_stratonovich_2}. Hence, formally,
\begin{align*}
\cor_{\T}&=
\int_0^t \p \Big[ \op\big(\ktwon [(\psi_n \cdot\nabla )\T]\big)\Big]\,\dd s.
\end{align*}
%Hence, setting $\cor_v\stackrel{{\rm def}}{=}\sum_{n\geq 1} \cor_{v,n}$, we have
%$$
%\cor_v = \int_0^t \int_{-h}^{\cdot} (\tp\cdot\nabla)\T\,d\zeta\,ds \qquad \text{ where }\qquad 
%\tp_{\T}  \stackrel{{\rm def}}{=}\sum_{n\geq 1} \ktwon \psi_n.
%$$ 
To compute $\cor_v$ we begin by looking at $ [v,\beta^n]_{t}$. 
To this end, note that, by \eqref{eq:primitive_stratonovich_1}, we formally have
$$
 [v,\beta^n]_t=\int_{0}^t \p\Big[ (\phi_n \cdot\nabla) v+\op(\ktwon \T)\Big]\,\dd s.
$$
To economize the notation, set $\nabla_{\h}\wt{p}_n=\q [(\phi_n \cdot\nabla) v+\op(\ktwon \T)]$. 
Thus 
\begin{align}
\label{eq:v_correction_computations}
\cor_v
&=\int_0^t \p \big((\phi_n\cdot\nabla )[(\phi_n \cdot\nabla) v]\big)\,\dd s\\
\nonumber
&\ \ \ +\underbrace{\int_0^t \p [(\phi_n \cdot\nabla) \op(\ktwon \T)]\,\dd s}_{\cor_{v,1}\stackrel{{\rm def}}{=}}
-\underbrace{\int_0^t  \p\big[(\phi_n\cdot\nabla)\nabla_{\h}\wt{p}_n \big]\,\dd s}_{\cor_{v,2}\stackrel{{\rm def}}{=}}.
\end{align}
Next we rewrite the terms $(\cor_{v,1},\cor_{v,2})$ conveniently. 
Recall that, due to our notation, $\phi_{n,\h}=(\phi^{1}_{n})_{j=1}^2$ and that $(\phi_{n,\h},\ktwon)$ are $\z$-independent by Assumption \ref{ass:stratonovich}\eqref{it:stratonovich_2}. Hence
\begin{align*}
(\phi_n \cdot\nabla) \op(\ktwon\T)
&= (\phi_{n,\h}\cdot\nabla_{\h})\op(\ktwon\T)+  \phi^{3}_n \nabla_{\h}(\ktwon \T)\\
&= \op\big[ \ktwon(\phi_{n,\h} \cdot\nabla_{\h}) \T\big]+  \phi^{3}_n \nabla_{\h}(\ktwon \T).
\end{align*}
Finally, we consider $\cor_{v,2}$. By \eqref{eq:Helmholtz_hydrostatic} and the fact that the $\wt{p}_n$'s are $x_3$-independent, 
\begin{align*}
\p[(\phi_n\cdot\nabla)\nabla_{\h}\wt{p}_n ]
=
\p[(\phi_{n,\h}\cdot\nabla_{\h})\nabla_{\h}\wt{p}_n ]
= - \sum_{1\leq i\leq 2} \p[(\nabla_{\h} \phi^i_n)\partial_i\wt{p}_n].
\end{align*}
Hence, by collecting the previous identities, we have
\begin{align}
\nonumber
\p\Big[
 (\phi_n\cdot\nabla) v + \op(\ktwon \T)\Big]\circ \dd \beta_s^n
&=\p\Big[
 (\phi_n\cdot\nabla) v + \int_{-h}^{\cdot}\nabla_{\h}\big(\ktwon(\cdot,\zeta) \T(\cdot,\zeta)\big)\,\dd \zeta\Big]\, \dd \beta_s^n
 \\
 \label{eq:correction_v}
 &+\p \big[\dop_{\phi} v +\Lp(v,\T)\big]\,\dd t\\
 \nonumber
 &+\p\Big[\int_{-h}^{\cdot} (\tp(\cdot,\zeta)\cdot \nabla) \T(\cdot,\zeta)\,\dd \zeta+\frac{1}{2}  \sum_{n\geq 1}\phi^{3}_n \nabla_{\h}(\ktwon \T)\Big] \,\dd t,
\end{align}
where $\dop_{\phi} v\stackrel{{\rm def}}{=}\frac{1}{2}\sum_{n\geq 1} (\phi_n \cdot\nabla )[(\phi_n\cdot\nabla) v]$,  $\Lp(v,\T)$ is as in \eqref{eq:Lp_def} with $\gvn=0$ and
$(\tp,\g)$ are given by
\begin{equation}
\label{eq:pi_psi_etc}
\tp^{j}\stackrel{{\rm def}}{=}
\begin{cases}
\vspace{0.1cm}
\displaystyle{\frac{1}{2}\sum_{n\geq 1} \ktwon(\psi_n^{j}+\phi_n^{j})} &\text{ for }j\in \{1,2\},\\
\displaystyle{\frac{1}{2}\sum_{n\geq 1} \ktwon\psi_n^{j}} \ \ \ &\text{ otherwise},
\end{cases}
\quad \text{ and }\quad \hp_n =\frac{1}{2}(\partial_{i}\phi^{j}_n)_{i,j=1}^2.
\end{equation}

Therefore \eqref{eq:correction_T} and \eqref{eq:correction_v}-\eqref{eq:pi_psi_etc} show that \eqref{eq:primitive_stratonovich} can be (formally) rephrased as \eqref{eq:primitive_strong}  (in the reformulation of \eqref{eq:primitive_strong}) by choosing $(\tp,\hp)$ as in \eqref{eq:pi_psi_etc}, $\fv= \frac{1}{2}\sum_{n\geq 1}\phi^{3}_n \nabla_{\h}(\ktwon \T)$, $\ft=0$, $\gv=\gt=0$ and the differential operators $(\Delta v,\Delta \T)$ replaced by 
$
(\Delta v+ \dop_{\phi}v,\Delta \T + \dop_{\psi} \T).
$
As we commented in Remark \ref{r:inhomogeneous_viscosity_conductivity} the case of inhomogeneous viscosity and/or diffusivity fits in our framework.
In particular, the definition of (global) $L^2$-solution to \eqref{eq:primitive}-\eqref{eq:boundary_conditions_full} given in Definition \ref{def:sol_strong_strong}
carries over to
\eqref{eq:primitive_stratonovich}-\eqref{eq:primitive_stratonovich_BC}. 
%%%%%%

%Before 
Now, we formulate the main result of this section. As in \eqref{eq:def_HV}, we let 
$
H=\Hs^{1}(\Dom)\times H^1(\Dom)$ and $  V=\Hs^{2}_{\n}(\Dom)\times \Hr^2(\Dom)
$, 
where $\Hs_{\n}^2(\Dom)$ and $ \Hr^2(\Dom)$ are defined in \eqref{eq:def_Hn} and \eqref{eq:def_Hr}, respectively.

\begin{theorem}[Global well-posedness -- Stratonovich formulation]
\label{t:global_Stratonovich}
Let Assumption \ref{ass:stratonovich} be satisfied. Let $(v_0,\T_0)\in L^0_{\F_0}(\O;H)$. Then  \eqref{eq:primitive_stratonovich}-\eqref{eq:primitive_stratonovich_BC} has a unique global $L^2$-strong solution $(v,\T)$ such that 
$$
(v,\T)\in C([0,\infty);H) \cap L^2_{{\rm loc}}([0,\infty);V) \text{ a.s.\ }
$$
Moreover, the following hold:
\begin{itemize}
\item The estimates of Theorem \ref{t:global_primitive_strong_strong} hold for the global $L^2$-strong solution $(v,\T)$ to \eqref{eq:primitive_stratonovich}-\eqref{eq:primitive_stratonovich_BC}. 
\item The assigment $(v_0,\T_0)\mapsto (v,\T)$ is continuous in the sense of Theorem \ref{t:continuous_dependence}.
\end{itemize}
\end{theorem}

\begin{proof}
One can readily check that  Assumption \ref{ass:stratonovich} is stronger than 
Assumptions \ref{ass:well_posedness_primitive_double_strong} and \ref{ass:global_primitive_strong_strong}. 
For instance the parabolicity assumption of Assumption \ref{ass:well_posedness_primitive_double_strong}\eqref{it:well_posedness_primitive_parabolicity_strong_strong} (see Remark \ref{r:inhomogeneous_viscosity_conductivity} for the case of inhomogeneous viscosity and/or conductivity) is automatically satisfied. Moreover, Assumption \ref{ass:global_primitive_strong_strong} follows from Assumption \ref{ass:stratonovich}\eqref{it:stratonovich_2} and \eqref{eq:pi_psi_etc}. 
Thus Theorem \ref{t:global_Stratonovich} follows from Theorems \ref{t:global_primitive_strong_strong}-\ref{t:continuous_dependence} and Remark \ref{r:inhomogeneous_viscosity_conductivity}.
\end{proof}

\begin{remark}[Weakening Assumption \ref{ass:stratonovich}\eqref{it:stratonovich_4} -- Local existence for \eqref{eq:primitive_stratonovich}]
Theorem \ref{t:local_primitive_strong_strong} also applies to the Stratonovich formulation \eqref{eq:primitive_stratonovich}. In particular, the local existence result of 
Theorem \ref{t:local_primitive_strong_strong} holds for \eqref{eq:primitive_stratonovich} provided Assumption \ref{ass:stratonovich}\eqref{it:stratonovich_1} and \eqref{it:stratonovich_3}-\eqref{it:stratonovich_4}.  
By the first part of Remark \ref{r:inhomogeneous_viscosity_conductivity}, to extend the local existence result of Theorem \ref{t:local_primitive_strong_strong}, the condition in Assumption \ref{ass:stratonovich}\eqref{it:stratonovich_4} can be weakened to the following: There exist $K,\eta>0$ such that, a.s.\ for all $j\in \{1,2\}$,
$$
\Big\|\sum_{n\geq  1} \phi^3_n(\cdot,0)\phi^j_n (\cdot,0)\Big\|_{H^{\frac{1}{2}+\eta}(\Tor^2)}
+\Big\|\sum_{n\geq  1} \phi^3_n(\cdot,-h)\phi^j_n (\cdot,-h)\Big\|_{H^{\frac{1}{2}+\eta}(\Tor^2)}\leq K.
$$
\end{remark}

{\small
\subsubsection*{Acknowledgements}
The first author thanks Umberto Pappalettera for helpful suggestions on Section \ref{s:physical_derivation} and for bringing to his attention the reference \cite{MTVE01}. 
The first author is grateful to Marco Romito for helpful comments related to Remarks \ref{r:sigma_independence} and \ref{r:phi_independence}. Finally, the first author thanks Caterina Balzotti for her support in creating the picture.
}

\bibliographystyle{alpha-sort}
\bibliography{literature}

\end{document}